\pdfoutput=1
\documentclass[10pt,a4paper,reqno,dvipsnames,x11names]{amsart}
\DeclareMathAlphabet{\mathpzc}{OT1}{pzc}{m}{it}
\usepackage[T1]{fontenc}
\usepackage[utf8]{inputenc}
\usepackage[british]{babel}
\usepackage[a4paper,left=20mm,right=20mm,top=33mm,bottom=25mm]{geometry}
\usepackage{microtype}
\usepackage{needspace}
\usepackage{amsmath,amssymb,amsthm,mathtools,mathrsfs}
\IfFileExists{newtxtext.sty}{\usepackage{newtxtext,newtxmath}}{\usepackage{lmodern}}
\usepackage[shortlabels]{enumitem}
\usepackage{xcolor}
\AtBeginDocument{\addtolength{\abovedisplayskip}{-1.5pt}\addtolength{\belowdisplayskip}{-1.5pt}\addtolength{\abovedisplayshortskip}{-0.5pt}\addtolength{\belowdisplayshortskip}{-0.5pt}}
\usepackage[colorlinks=true,linkcolor=black,citecolor=black,urlcolor=black]{hyperref}
\usepackage{aliascnt}
\usepackage[capitalize,nameinlink]{cleveref}

\numberwithin{equation}{section}

\theoremstyle{plain}
\newtheorem{theorem}{Theorem}[section]
\newaliascnt{proposition}{theorem}
\newtheorem{proposition}[proposition]{Proposition}
\aliascntresetthe{proposition}
\newaliascnt{lemma}{theorem}
\newtheorem{lemma}[lemma]{Lemma}
\aliascntresetthe{lemma}
\newaliascnt{corollary}{theorem}
\newtheorem{corollary}[corollary]{Corollary}
\aliascntresetthe{corollary}
\theoremstyle{definition}
\newaliascnt{example}{theorem}
\newtheorem{example}[example]{Example}
\aliascntresetthe{example}
\newaliascnt{definition}{theorem}
\newtheorem{definition}[definition]{Definition}
\aliascntresetthe{definition}
\theoremstyle{remark}
\newaliascnt{remark}{theorem}
\newtheorem{remark}[remark]{Remark}
\aliascntresetthe{remark}
\newaliascnt{question}{theorem}

\aliascntresetthe{question}

\crefname{theorem}{Theorem}{Theorems}
\Crefname{theorem}{Theorem}{Theorems}
\crefname{proposition}{Proposition}{Propositions}
\Crefname{proposition}{Proposition}{Propositions}
\crefname{lemma}{Lemma}{Lemmas}
\Crefname{lemma}{Lemma}{Lemmas}
\crefname{corollary}{Corollary}{Corollaries}
\Crefname{corollary}{Corollary}{Corollaries}
\crefname{example}{Example}{Examples}
\Crefname{example}{Example}{Examples}
\crefname{definition}{Definition}{Definitions}
\Crefname{definition}{Definition}{Definitions}
\crefname{remark}{Remark}{Remarks}
\Crefname{remark}{Remark}{Remarks}
\crefname{question}{Question}{Questions}
\Crefname{question}{Question}{Questions}

\newcommand{\C}{\mathbb C}
\newcommand{\ii}{\mathrm i}
\newcommand{\cA}{\mathcal A}
\newcommand{\cE}{\mathcal E}
\newcommand{\cF}{\mathcal F}
\newcommand{\cG}{\mathcal G}
\newcommand{\cH}{\mathcal H}
\newcommand{\cI}{\mathcal I}
\newcommand{\cL}{\mathcal L}
\newcommand{\cS}{\mathcal S}
\newcommand{\cU}{\mathcal U}
\newcommand{\cV}{\mathcal V}
\newcommand{\cW}{\mathcal W}
\newcommand{\cO}{\mathcal O}
\newcommand{\End}{\operatorname{End}}
\newcommand{\GL}{\operatorname{GL}}
\newcommand{\NSQC}{\operatorname{NSQC}}
\newcommand{\Aut}{\operatorname{Aut}}
\newcommand{\diag}{\operatorname{diag}}
\newcommand{\res}{\operatorname{res}}
\newcommand{\Res}{\operatorname{Res}}
\newcommand{\Logm}{\operatorname{Log}_{-}}
\newcommand{\Gm}{\mathbb G_m}
\newcommand{\Fr}{\operatorname{Fr}}

\newcommand{\Isom}{\operatorname{Isom}}
\newcommand{\Fbun}{\mathsf{Fbund}}
\newcommand{\Atoms}{\mathsf{Atoms}}
\newcommand{\RAtoms}{\mathsf{RAtoms}}
\newcommand{\CF}{\mathsf{CF}}

\newcommand{\ExpShift}{\mathsf{E}}

\author{Mauricio Corr\^ea}
\address{Dipartimento di Matematica, Universit\`a degli Studi di Bari Aldo Moro, Bari, Italy}
\email{mauricio.barros@uniba.it}

\date{}
\subjclass[2020]{34M35, 14D23, 14M15, 32G34, 14N35}
\keywords{residual atoms, prescribed lattices, affine flag varieties, wall crossing, meromorphic connections, quantum cohomology, F-bundles}

\title[Lattice moduli and residual atoms]{Lattice Moduli and Residual Atoms}
\hypersetup{
  pdftitle={Lattice Moduli and Residual Atoms},
  pdfauthor={Mauricio Corr\^ea},
  pdfsubject={Lattice moduli, wall crossing, residual atoms and projective varieties},
  pdfkeywords={residual atoms, prescribed lattices, affine flag varieties, wall crossing, quantum cohomology, F-bundles}
}
\begin{document}

\begin{abstract}
We classify fully framed regular-singular connections on prescribed lattices for arbitrary ordered monodromy eigenvalues, including resonances. The fully framed classification has a unique triangular normal form on each fixed-residue stratum, parametrised by an affine space of dimension $r(r-1)$. We determine the connection-compatible charts of elementary lattice modifications, including resonance walls and higher-root corrections. For geometric $F$-bundles we establish a reconstruction theorem for trace-centred connections with specified horizontal pairings and determine the information lost upon forgetting the trace scalar. For a meromorphic spectral decomposition and a prescribed lattice, a defect module measures the failure of its projectors to preserve that lattice. Its annihilator exponent is the exact common pole bound and changes by at most one under an elementary modification. We compute the defect at cyclic Fano quantum boundaries and show that, whenever formal horizontal lifts of the Higgs projectors exist, their pole orders are unbounded.
\end{abstract}

\maketitle

\section{Introduction}

A meromorphic connection on a punctured curve may have several extensions across the puncture. The underlying meromorphic connection is the same, but the lattices and their behaviour under elementary modifications differ. A related question arises for quantum $F$-bundles: their spectral factors are defined over the non-Archimedean base, whereas their extension across a boundary depends on the chosen lattice~\cite{HYZZ,KKPYY}. We study these two aspects through the normal forms of framed connections, modifications of lattices and the extension of spectral projectors.

We begin with the fully framed regular-singular connections of Kontsevich and Szab\'o~\cite{KS}. Fixing a lattice, a monodromy-invariant full flag and an ordered frame at a regular point leads to a classification valid for arbitrary ordered monodromy eigenvalues. In particular, repeated eigenvalues and resonances are allowed. The residue vector is intrinsic to the framed lattice filtration, although logarithms of non-semisimple monodromy cannot be chosen independently along a Jordan block.

Let $\Delta\subset\C$ be a disc containing $0$ and~$1$, let $\mathfrak n$ denote the strictly upper triangular Lie algebra, and set
\[
 \Logm(c)=\{\mu\in\C\mid e^{-2\pi\ii\mu}=c\}.
\]
For $\mathbf c=(c_1,\ldots,c_r)$, let $\NSQC^{\mathrm{fr}}_{\mathbf c}$ denote the groupoid of framed connections defined in \Cref{def:groupoid}; write $S_{\mathrm{disc}}$ for the discrete groupoid on a set~$S$. Our classification is as follows.

\begin{theorem}\label{thm:main}
Let $r\geq2$ and let $\mathbf c=(c_1,\ldots,c_r)\in(\C^\times)^r$ be arbitrary.  There is a canonical functor
\[
 \mathbf{res}_{\Delta}:\NSQC^{\mathrm{fr}}_{\mathbf c}
 \longrightarrow
 \left(\prod_{i=1}^r\Logm(c_i)\right)_{\mathrm{disc}},
\]
which is essentially surjective.  For every ordered residue vector
$\boldsymbol\mu=(\mu_1,\ldots,\mu_r)\in\prod_i\Logm(c_i)$, the fibre groupoid
$\mathbf{res}_{\Delta}^{-1}(\boldsymbol\mu)$ is equivalent to the discrete groupoid associated with
\[
 \mathfrak n\oplus\mathfrak n\simeq\mathbb A^{r(r-1)}_{\C}.
\]
Every object of the fibre is isomorphic to a unique connection
\begin{equation}\label{eq:main-normal-form}
 \nabla_{\boldsymbol\alpha,\boldsymbol\beta}
 =d+\left[
 \frac{M_{\boldsymbol\mu}}{z}
 +\sum_{1\leq j<i\leq r}
 \bigl(\alpha_{ji}z^{-2}+\beta_{ji}\rho_{\lambda_{ji}}(z)\bigr)E_{ji}
 \right]dz,
\end{equation}
where
$
 M_{\boldsymbol\mu}=\diag(\mu_1,\ldots,\mu_r),$ and $
  \lambda_{ji}=\mu_i-\mu_j,$
and
\[
 \rho_\lambda(z)=
 \begin{cases}
 z^{\lambda-1},&\lambda\in\mathbb Z_{\geq0},\\
 z^{-1},&\lambda\notin\mathbb Z_{\geq0}.
 \end{cases}
\]
The coordinates $(\boldsymbol\alpha,\boldsymbol\beta)$ are complete invariants, and every fully framed automorphism is the identity.  Equivalently, there is a decomposition of groupoids
\[
 \NSQC^{\mathrm{fr}}_{\mathbf c}
 \simeq
 \bigsqcup_{\boldsymbol\mu\in\prod_i\Logm(c_i)}
 \bigl(\mathbb A^{r(r-1)}_{\C}\bigr)_{\mathrm{disc}}.
\]
In particular, repeated monodromy eigenvalues and the eigenvalue~$1$ do not change either the affine type or the groupoid structure of an ordered residue stratum.
\end{theorem}

The normal forms are obtained by a triangular reduction followed by successive framed gauges along the positive roots. Their relative version, \Cref{thm:relative-normal-form}, gives holomorphic coordinates in families with a fixed triangular lattice filtration and residue vector. Elementary modifications lead to explicit wall maps: \Cref{thm:rank-two-wall} gives the resonance condition in rank two; \Cref{thm:rank-three-wall,thm:rank-three-central-cocycle} compute the composite-root correction and its cocycle; and \Cref{thm:all-rank-affine-wall,thm:wall-loop-root-group} describe the corresponding higher-rank root-subgroup relations. These results relate the affine normal slices to the relative positions of lattices in the affine flag variety.

For the geometric $F$-bundles of Katzarkov--Kontsevich--Pantev--Yu~\cite{KKPYY}, the vertical connection is generally irregular. Let $\kappa_\nabla$ be the endomorphism induced by $\nabla_{u^2\partial_u}$ on $\cH/u\cH$, and put
\[
 a=-\frac{\operatorname{tr}(\kappa_\nabla)}{\operatorname{rk}\cH},\qquad
 \nabla^0=\nabla-d(a/u)I.
\]
The trace-centred connection $\nabla^0$ is defined on the same lattice, including its horizontal connection. The equivalences of \Cref{thm:valued-residualisation-equivalence,thm:paired-residualisation-equivalence} reconstruct the original $F$-bundle, with any specified horizontal pairing, from $\nabla^0$ and~$a$ in the formal, complex-analytic and non-Archimedean categories. The scalar correction cancels under the involution $u\mapsto-u$, so horizontality of the pairing is preserved.

The geometric assertion is \Cref{thm:residual-reconstruction-loss}. On a fixed base, forgetting $a$ can identify distinct paired connections; in the complex setting the half-turn Euler form changes by $\chi_{\nabla^0}=e^{-2a}\chi_\nabla$. For geometric atomic factors with pairings induced from the full $A$-model, the centred connection nevertheless determines the paired atomic class under the same-component relation of~\cite[Definition~5.21]{KKPYY}. The proof uses the intrinsic unit foliation and a transverse suspension. It applies, in particular, to the unique geometric atomic class of a Calabi--Yau variety with $H^1(X,\C)=0$. The horizontal pairing is preserved over the non-Archimedean base; neither the affine unit coordinate nor an Euler or integral enhancement is thereby recovered. The trace-labelled correspondence transports the birational criteria of~\cite{KKPYY} (\Cref{prop:atomic-correspondence-residualises}).

The behaviour of spectral factors at a boundary is measured by a lattice-theoretic invariant. Let $x$ be a non-zero-divisor in a noetherian ring~$R$, and let a meromorphic decomposition $H[1/x]=\bigoplus_\alpha V_\alpha$ be given for a finite locally free lattice~$H$. Writing $L_\alpha=H\cap V_\alpha$, define
\[
 D_{\mathrm{sp}}(H)=H\Big/\bigoplus_\alpha L_\alpha.
\]
The exact sequence of \Cref{thm:spectral-recomposition-defect} identifies the least integer $d$ satisfying $x^dD_{\mathrm{sp}}(H)=0$ with the least common pole order of the spectral projectors. The wall estimate of \Cref{prop:wall-control-spectral-defect} gives
\[
 x^aH\subseteq H'\subseteq x^{-b}H
 \quad\Longrightarrow\quad |d(H')-d(H)|\leq a+b.
\]
Thus an elementary modification changes the projector pole bound by at most one; the bound is attained by logarithmic connections (\Cref{ex:sharp-horizontal-wall-defect}). For formal projectors, \Cref{thm:uniform-formal-spectral-defect} characterises meromorphic extension by a uniform bound on the defects of all $u$-jets.

Finally, \Cref{prop:projective-spectral-lattice-collision} computes the Higgs eigenline defect for cyclic Picard-one Fano varieties, including projective spaces and odd-dimensional quadrics. If the Higgs projectors admit mutually orthogonal formal horizontal lifts, \Cref{cor:quantum-boundary-jet-test} shows that their pole orders are unbounded, by the nilpotent residue of the logarithmic quantum connection.

\section{Notation and conventions}\label{sec:notation}

We follow the notation of~\cite[\S\S3--5]{KKPYY} whenever the global theory of atoms is involved.  The local complex coordinate is denoted by $z$, while $u$ is reserved for the coordinate used in the theory of $F$-bundles.
\begin{description}[leftmargin=3.2cm,style=multiline]
 \item[$(\cH,\nabla)/S$] an $F$-bundle over a formal, complex analytic or non-Archimedean base $S$;
 \item[$\cH_0$] the intrinsic lattice fibre $\cH/u\cH$;
 \item[$\kappa_{\nabla}$] the residual endomorphism $\left.\nabla_{u^2\partial_u}\right|_{\cH_0}$;
 \item[$\widetilde U_X\to U_X$] the unramified part of the reduced spectral cover of the maximal $A$-model $F$-bundle of $X$;
 \item[$\Atoms_G^{K,F}$] the set of geometric $G$-atomic $F$-bundles of smooth projective $K$-varieties;
 \item[$\CF_G^F(X)$] the corresponding atomic $F$-bundle chemical formula;
 \item[$\RAtoms_G^K$] the set of trace-labelled residual $G$-atoms introduced in \Cref{sec:residual-atoms};
\end{description}
For a set $S$, the symbols $\mathbb N^{(S)}$ and $\mathbb Z^{(S)}$ denote, respectively, the commutative monoid and the abelian group of finitely supported functions on~$S$.  We use $\sqcup$ for a finite disjoint union and $\bigsqcup$ for an indexed disjoint union.

\section{Prescribed lattices and triangular monodromy}\label{sec:groupoid}

Let $B\subset\GL_r(\C)$ be the standard upper triangular Borel subgroup, let $U$ be its unipotent radical, and let $\mathfrak b$ and $\mathfrak n$ be their Lie algebras.  We use the gauge action
\begin{equation}\label{eq:gauge-action}
 (d+A\,dz)\cdot g
 =d+\bigl(g^{-1}Ag+g^{-1}g'\bigr)dz.
\end{equation}

\begin{definition}\label{def:groupoid}
Kontsevich and Szab\'o define this groupoid under the assumption $c_i\neq1$ for every $i$.  We use the same definition for an arbitrary ordered vector $\mathbf c=(c_1,\ldots,c_r)\in(\C^\times)^r$. The enlarged groupoid $\NSQC^{\mathrm{fr}}_{\mathbf c}$ has as objects triples $(E,\nabla,\mathbf v)$ with the following properties.
\begin{enumerate}[(i)]
 \item $E=\cO^{\oplus r}$, with its standard trivialisation.
 \item $\nabla$ is a meromorphic connection on $E$ which is regular singular at~$0$.
 \item In the prescribed lattice $E$, the connection matrix has a pole of order at most two at~$0$.
 \item $\mathbf v=(v_1(1),\ldots,v_r(1))$ is an ordered basis of $E|_{z=1}$.
 \item The monodromy preserves the full flag defined by $\mathbf v$ and has diagonal entries $c_1,\ldots,c_r$ in that basis.
\end{enumerate}
We regard an ordered basis as a row $\mathbf v=(v_1(1),\ldots,v_r(1))$.  A morphism
\[
 g:(E,\nabla,\mathbf v)\longrightarrow(E,\widetilde\nabla,\widetilde{\mathbf v})
\]
is a holomorphic gauge transformation $g\in\GL_r(\cO)$ satisfying
\[
 \widetilde\nabla=\nabla\cdot g,
 \qquad
 \widetilde{\mathbf v}=\mathbf v\,g(1).
\]
This is the convention used in~\cite{KS}: both the connection matrix and the row of framed basis vectors carry the same right action.  Applying first $g$ and then $h$ is therefore the action of $gh$.  In particular, a morphism between two standard-framed representatives satisfies $g(1)=I_r$.
\end{definition}

All quotient calculations below take place after the ordered frame has been changed to the standard row $\mathbf e(1)$.  At that stage a morphism satisfies $g(1)=I_r$.
Fixing the standard free module causes no loss of generality relative to the formulation of \cite{KS}: every module occurring there is free of rank $r$, and a choice of trivialisation gives an equivalent groupoid.
The full framing eliminates automorphisms: a horizontal holomorphic endomorphism which is the identity at $z=1$ is the identity everywhere by uniqueness for a linear initial-value problem.

\begin{lemma}\label{lem:triangular-log}
Every $T\in B$ admits $R\in\mathfrak b$ such that $e^R=T$.
\end{lemma}

\begin{proof}
Write the set of distinct eigenvalues of $T$ as $\{\sigma_1,\ldots,\sigma_s\}$.  Choose pairwise disjoint simply connected neighbourhoods $D_1,\ldots,D_s\subset\C^\times$, with $\sigma_\nu\in D_\nu$, and choose a holomorphic branch $\ell_\nu$ of the logarithm on each $D_\nu$.  Let $D=\bigsqcup_{\nu=1}^{s} D_\nu$ and let $\ell$ denote the resulting holomorphic function on~$D$.
The holomorphic functional calculus defines
\[
 R=\ell(T)=\frac{1}{2\pi\ii}
 \int_\Gamma \ell(\zeta)(\zeta I_r-T)^{-1}\,d\zeta,
\]
where $\Gamma\subset D$ is a union of positively oriented contours surrounding the eigenvalues.  Since $T$ is upper triangular, every resolvent $(\zeta I_r-T)^{-1}$ is upper triangular; hence $R\in\mathfrak b$.  The functional-calculus identity
\[
 \exp(\ell(T))=(\exp\circ\ell)(T)=T
\]
gives $e^R=T$.  Equivalently, one may use Hermite interpolation to replace $\ell$ by a polynomial having the same derivatives at the eigenvalues up to the sizes of the Jordan blocks; this also shows directly that $R$ belongs to the finite-dimensional algebra $\C[T]\subset\mathfrak b$.  This is the matrix-logarithm step in Levelt's regular-singular reduction~\cite{Levelt,Deligne,Sabbah}.
\end{proof}

\begin{remark}\label{rem:log-constraint}
The diagonal entries of a logarithm cannot always be chosen independently. If $T=c\left(\begin{smallmatrix}1&t\\0&1\end{smallmatrix}\right)$ with $t\ne0$ and $R=\left(\begin{smallmatrix}a&b\\0&d\end{smallmatrix}\right)$ with $a\ne d$, then $(e^R)_{12}=b(e^a-e^d)/(a-d)$. It vanishes when $a-d\in2\pi\ii\mathbb Z\setminus\{0\}$ and $e^a=e^d$, so such an $R$ cannot exponentiate to $T$. The triangular argument requires a compatible logarithm, not independently prescribed entries.
\end{remark}

\section{Triangular reduction and framed rigidity}\label{sec:triangular}

We begin with a lattice-theoretic triangular reduction which is intrinsic to the monodromy flag and the prescribed lattice.

\begin{proposition}\label{prop:triangular-reduction}
Every object of $\NSQC^{\mathrm{fr}}_{\mathbf c}$ is isomorphic to a standard-framed connection
\begin{equation}\label{eq:triangular-expansion}
 d+\left(\sum_{n\geq-2}D_nz^n\right)dz,
 \qquad
 D_{-2}\in\mathfrak n,
 \quad D_n\in\mathfrak b\quad(n\geq-1),
\end{equation}
where the Laurent series converges on~$\Delta$.
No condition is imposed on repetitions among the $c_i$, and the eigenvalue~$1$ is allowed.
\end{proposition}

\begin{proof}
Let
\[
 0=F_0(1)\subset F_1(1)\subset\cdots\subset F_r(1)=E|_{z=1}
\]
be the full flag determined by the ordered frame.  Its monodromy invariance gives, by parallel transport, a filtration by flat subbundles $F_i^\circ\subset E|_{\Delta^\times}$. Since the connection is regular singular, each $F_i^\circ$ determines a unique meromorphic $\nabla$-stable subbundle $F_i\subset E(*0)$. After choosing a logarithmic model, the corresponding sub-local system extends to a meromorphic subconnection; uniqueness follows because two such extensions agree on the punctured disc.  Equivalently, one may choose an upper triangular logarithmic model of the monodromy and transport its standard flag.  This step uses neither distinctness of the eigenvalues nor exclusion of the eigenvalue~$1$.
Inside the prescribed lattice define $E_i=(E\cap F_i)^{\mathrm{sat}}\subset E$, where saturation is taken in $E$.  The intersection is coherent: locally, after multiplying $F_i$ by a sufficiently large power of $z$, both $E$ and $z^NF_i$ are coherent submodules of the same free meromorphic module, and their intersection is coherent over the Noetherian ring $\cO_{\Delta,0}$.  On the smooth curve $\Delta$, every saturated coherent subsheaf of a vector bundle is locally free, and each quotient $E_i/E_{i-1}$ is a line bundle.  Saturation changes nothing on $\Delta^\times$, so $E_i(1)=F_i(1)$.  Each exact sequence
\[
 0\longrightarrow E_{i-1}\longrightarrow E_i
 \longrightarrow E_i/E_{i-1}\longrightarrow0
\]
splits holomorphically because
\[
 \operatorname{Ext}^1(E_i/E_{i-1},E_{i-1})
 \cong H^1\!\left(\Delta,
 E_{i-1}\otimes(E_i/E_{i-1})^{-1}\right)=0.
\]
Since $\Delta$ is Stein and contractible, $\operatorname{Pic}(\Delta)=0$; hence each quotient $E_i/E_{i-1}$ is holomorphically trivial.  Choose the splittings inductively.  At the $i$-th stage, choose a nowhere-vanishing section of $E_i/E_{i-1}$ whose value at~$1$ is the image of $v_i(1)$, lift it through the chosen splitting, and modify the lift by a global section of $E_{i-1}$ so that its value at~$1$ is $v_i(1)$.  Such a correction exists because the preceding splittings identify $E_{i-1}$ with a direct sum of trivial line bundles, so evaluation at~$1$ is surjective.  This produces a holomorphic frame $w_1,\ldots,w_r$ of $E$, adapted to $E_\bullet$ and equal to the prescribed ordered frame at~$1$.  The resulting representative is therefore standard-framed.

The matrix of $\nabla$ in the frame $w_1,\ldots,w_r$ is upper triangular on $\Delta^\times$, because every $F_i$ is meromorphically $\nabla$-stable.  Its entries are meromorphic on the disc, so the strictly lower triangular entries, which vanish on the punctured disc, vanish identically.  Because the change of frame is holomorphic in the prescribed lattice, the pole order remains at most two.

It remains to examine the diagonal.  The $i$-th diagonal entry is the connection form induced on the rank-one meromorphic quotient $F_i/F_{i-1}$, expressed in the lattice $E_i/E_{i-1}$.  Regular singularity is stable under subquotients~\cite[Chap.~II]{Sabbah}.  A rank-one regular-singular connection has at most a simple pole in every lattice, since any two rank-one lattices differ by multiplication by an integral power of~$z$, which changes the connection form only by an integral multiple of $dz/z$.  Therefore the diagonal coefficient of $z^{-2}$ vanishes.  Consequently $D_{-2}\in\mathfrak n$, while $D_n\in\mathfrak b$ for every $n\geq-1$, as claimed.
\end{proof}

\begin{lemma}\label{lem:Borel-rigidity}
Let
\[
 \nabla=d+A(z)\,dz,
 \qquad
 \widetilde\nabla=d+\widetilde A(z)\,dz
\]
be meromorphic connections whose matrices take values in $\mathfrak b$ on the punctured disc.  If $g\in\GL_r(\cO)$ satisfies
$
 \widetilde\nabla=\nabla\cdot g, $ and $  g(1)=I_r,
$
then $g$ takes values in~$B$.
\end{lemma}

\begin{proof}
The gauge identity is equivalent to
\begin{equation}\label{eq:gauge-ode}
 g'=g\widetilde A-Ag.
\end{equation}
For each $z\neq0$, the linear operator $X\longmapsto X\widetilde A(z)-A(z)X$ preserves $\mathfrak b$, because products of upper triangular matrices are upper triangular.  It therefore induces a meromorphic linear operator $\overline{\mathcal L}(z)$ on $\mathfrak{gl}_r/\mathfrak b$.  If $\bar g$ denotes the class of $g$, projection of~\eqref{eq:gauge-ode} gives
\[
 \bar g'=\overline{\mathcal L}(z)\bar g,
 \qquad
 \bar g(1)=0.
\]
Uniqueness for this homogeneous linear initial-value problem implies $\bar g=0$ near~$1$.  Analytic continuation on the connected punctured disc gives $\bar g=0$ everywhere on $\Delta^\times$.  Equivalently, every strictly lower triangular entry of $g$ vanishes on $\Delta^\times$; these entries are holomorphic on $\Delta$, hence vanish identically.  Since $g$ is holomorphic and invertible, its diagonal entries are holomorphic units.  Thus $g:\Delta\to B$.
\end{proof}

\begin{corollary}\label{cor:triangular-subgroupoid}
The full subgroupoid of standard-framed upper triangular connections of the form \eqref{eq:triangular-expansion} is essentially surjective in $\NSQC^{\mathrm{fr}}_{\mathbf c}$.  Every morphism between two of its objects is a holomorphic $B$-valued gauge equal to $I_r$ at $z=1$.
\end{corollary}

\begin{proof}
Essential surjectivity is \Cref{prop:triangular-reduction}.  A morphism between standard-framed representatives satisfies $g(1)=I_r$ and is $B$-valued by \Cref{lem:Borel-rigidity}.
\end{proof}

At this point the exclusion of the eigenvalue~$1$ becomes unnecessary in the fully framed problem.  The proof of~\cite[Lemma~4.2]{KS} uses a residue obstruction to show triangularity of the transported gauge.  The framed initial condition already forces the same conclusion without any restriction on the spectrum.

\begin{proposition}\label{prop:residue-map}
There is a canonical essentially surjective functor
\[
 \mathbf{res}_{\Delta}:
 \NSQC^{\mathrm{fr}}_{\mathbf c}
 \longrightarrow
 \left(\prod_{i=1}^r\Logm(c_i)\right)_{\mathrm{disc}},
\]
defined on objects by the ordered diagonal residues of any standard-framed upper triangular representative and on every morphism by the identity of that residue vector.
\end{proposition}

\begin{proof}
Existence of a standard-framed upper triangular representative follows from \Cref{cor:triangular-subgroupoid}.  Suppose that two such representatives are isomorphic in the original groupoid.  The induced morphism is $B$-valued by \Cref{lem:Borel-rigidity}.  If its diagonal entries are $g_{ii}$, then the $i$-th diagonal connection coefficient changes by $d\log g_{ii}$.  Since $g_{ii}$ is a holomorphic unit on the disc, this logarithmic derivative has residue zero.  The ordered diagonal residue vector is therefore invariant under every morphism and independent of all choices made in the triangular reduction.  This defines the stated functor to the discrete groupoid.

For an upper triangular regular-singular system, the diagonal scalar equation shows that a diagonal residue $\mu_i$ has monodromy $e^{-2\pi\ii\mu_i}$ in our convention.  Hence the image lies in $\prod_i\Logm(c_i)$.  Conversely, every vector in that product is realised by the standard-framed diagonal connection $d+M_{\boldsymbol\mu}z^{-1}\,dz$. Thus the functor is essentially surjective.
\end{proof}

The functor $\mathbf{res}_{\Delta}$ is not obtained by taking the residue of the connection matrix in the original prescribed lattice, whose pole may have order two.  It records the ordered diagonal residue after framed triangular reduction.  \Cref{prop:residue-map} asserts that this datum is independent of every choice made in the reduction.

\section{Framed normal forms for arbitrary spectra}\label{sec:normalisation-enlarged}

Once framed Borel rigidity has been established, the remaining normalisation is the rootwise scalar quotient and height-filtered unipotent induction of Kontsevich and Szab\'o~\cite[Proposition~4.3 and Lemma~4.4]{KS}.  We give the analytic calculation in full because the resonant monomial, the framed condition at $z=1$, and base-change compatibility are all used later.  The scalar cokernel and the two coordinates per root originate in~\cite{KS}; here the same slice is shown to compute the original fully framed groupoid for every spectrum.

\subsection{Diagonal normalisation and residue strata}\label{sec:diagonal}

Let $\nabla=d+A(z)\,dz$ be a standard-framed upper triangular representative.  Denote the diagonal entries of $A$ by $a_i(z)$ and put
\[
 \mu_i=\res_0\bigl(a_i(z)\,dz\bigr),
 \qquad
 M_{\boldsymbol\mu}=\diag(\mu_1,\ldots,\mu_r).
\]

\begin{lemma}\label{lem:diagonal-normalisation}
There is a unique diagonal holomorphic gauge $h:\Delta\to(\C^\times)^r$ satisfying $h(1)=I_r$ for which the diagonal part of the transformed connection is $\frac{M_{\boldsymbol\mu}}{z}\,dz$. Moreover, every framed $B$-valued gauge between two upper triangular representatives with this fixed diagonal is $U$-valued.
\end{lemma}

\begin{proof}
By \Cref{prop:triangular-reduction}, each $a_i(z)\,dz$ has at most a simple pole, and $a_i(z)-\frac{\mu_i}{z}$ is holomorphic on~$\Delta$.  The scalar differential equation
\[
 \frac{h_i'}{h_i}=\frac{\mu_i}{z}-a_i(z),
 \qquad h_i(1)=1,
\]
has a unique nowhere-vanishing holomorphic solution, since the integrand below is holomorphic on the simply connected disc, so the integral is path-independent and
\[
 h_i(z)=\exp\!\left(\int_1^z
 \left(\frac{\mu_i}{\zeta}-a_i(\zeta)\right)d\zeta\right).
\]
Taking $h=\diag(h_1,\ldots,h_r)$ proves existence.  If $\widetilde h$ is another such gauge, then each diagonal entry of $\widetilde h h^{-1}$ has zero logarithmic derivative and equals~$1$ at $z=1$; hence $\widetilde h=h$.

Now let $g$ be a framed $B$-valued gauge between two representatives whose diagonal part is $M_{\boldsymbol\mu}dz/z$.  Comparing the $i$-th diagonal entries in the gauge equation gives $\frac{g_{ii}'}{g_{ii}}=0$. Thus each $g_{ii}$ is constant, and $g(1)=I_r$ forces $g_{ii}=1$.  Hence $g$ is $U$-valued.
\end{proof}

The diagonal gauge need not be polynomial, even when the initial matrix has a convergent Laurent expansion of the form~\eqref{eq:triangular-expansion}.  The natural analytic coefficient space after diagonal normalisation is therefore
\[
 \cA_{\boldsymbol\mu}
 =\left\{
 d+\left(\frac{M_{\boldsymbol\mu}}{z}+Q(z)\right)dz:
 Q\in\mathfrak n\otimes z^{-2}\cO
 \right\}.
\]
The analytic unipotent group
\[
 \cU_1=\{g\in U(\cO):g(1)=I_r\}
\]
acts on $\cA_{\boldsymbol\mu}$.  By \Cref{lem:Borel-rigidity,lem:diagonal-normalisation}, it is the full framed gauge group relating elements of this fixed-diagonal space.

\begin{lemma}\label{lem:analytic-regular}
Every connection in $\cA_{\boldsymbol\mu}$ is regular singular at the origin.
\end{lemma}

\begin{proof}
Fix a closed subsector $S$ of a sufficiently small sector at the origin and choose branches of the functions $z^{-\mu_i}$ on it.  We solve the horizontal system from the last component upwards.  The last component is a constant multiple of $z^{-\mu_r}$ and therefore has at most polynomial growth on~$S$.

Assume inductively that the components with index greater than $j$ have at most polynomial growth.  Write the $j$-th equation as
\[
 y_j'+\mu_j z^{-1}y_j
 =-\sum_{i>j}q_{ji}(z)y_i,
 \qquad q_{ji}\in z^{-2}\cO.
\]
After setting $y_j=z^{-\mu_j}u_j$, one obtains
\[
 u_j'
 =-\sum_{i>j}z^{\mu_j}q_{ji}(z)y_i.
\]
On the fixed subsector, the right-hand side is bounded by $C|z|^{-N}$ for some $C,N>0$.  Its integral from a fixed non-zero point to $z$ is bounded by a power of $|z|^{-1}$, with the usual logarithmic bound when $N=1$.  Multiplication by $z^{-\mu_j}$ preserves polynomial growth on~$S$.  This proves the inductive step.  Choosing the diagonal initial constants successively produces an upper triangular fundamental matrix whose diagonal entries are the non-zero functions $z^{-\mu_i}$ and whose off-diagonal entries have moderate growth.  Its inverse also has moderate growth: this follows recursively from triangular inversion, since the diagonal entries are powers of~$z$ on the chosen sector.  The sectorial moderate-growth criterion for regular singularity, in the form recalled in~\cite[(1.3)]{Malgrange}, therefore applies.
\end{proof}

\begin{proposition}\label{prop:admissibility}
Let $\boldsymbol\mu\in\prod_i\Logm(c_i)$.  Every connection in $\cA_{\boldsymbol\mu}$, equipped with the standard framing at $z=1$, is an object of $\NSQC^{\mathrm{fr}}_{\mathbf c}$.  In particular, every connection of the normal form \eqref{eq:main-normal-form} is admissible.
\end{proposition}

\begin{proof}
The matrix has a pole of order at most two by definition.  Regular singularity is \Cref{lem:analytic-regular}.  Since the matrix is upper triangular, the standard flag is preserved by parallel transport.  The $i$-th diagonal horizontal equation is
\[
 y_i'+\mu_i z^{-1}y_i=0,
\]
so the corresponding diagonal monodromy entry is $e^{-2\pi\ii\mu_i}=c_i$.  Thus all defining conditions of \Cref{def:groupoid} are satisfied.
\end{proof}

\begin{proposition}\label{prop:diagonal-reduction}
Fix $\boldsymbol\mu\in\prod_i\Logm(c_i)$.  The fibre groupoid $\mathbf{res}_{\Delta}^{-1}(\boldsymbol\mu)$ is equivalent to the action groupoid $[\cA_{\boldsymbol\mu}(\C)/\cU_1(\C)]$.  Every object of the fibre has a standard-framed representative in $\cA_{\boldsymbol\mu}$, and the morphisms between two such representatives are the elements of $\cU_1$ relating them.
\end{proposition}

\begin{proof}
Apply the unique framed diagonal gauge of \Cref{lem:diagonal-normalisation} to any standard-framed triangular representative in the residue fibre.  This gives an element of $\cA_{\boldsymbol\mu}$.
Conversely, let two elements of $\cA_{\boldsymbol\mu}$ be related in the original groupoid.  By \Cref{lem:Borel-rigidity}, the relating gauge is $B$-valued and equals $I_r$ at $z=1$.  On the diagonal, the gauge equation reads $\frac{g_{ii}'}{g_{ii}}=0$, because both diagonal connection coefficients are $\mu_i z^{-1}$.  Thus $g_{ii}$ is constant and the framed condition gives $g_{ii}=1$.  The gauge therefore belongs to $\cU_1$.  Conversely, every element of $\cU_1$ is a morphism in the original groupoid.
\end{proof}

The residue fibres are all non-empty by \Cref{prop:residue-map}, although \Cref{rem:log-constraint} shows that a fixed non-semisimple monodromy matrix need not occur in every fibre.

\subsection{The scalar analytic quotient}\label{sec:scalar}

Put $V=\{P\in\cO:P(1)=0\}$ and $W=z^{-2}\cO$.
For $\lambda\in\C$, define
\[
 L_\lambda:V\longrightarrow W,
 \qquad
 L_\lambda(P)=P'-\lambda z^{-1}P.
\]

\begin{lemma}\label{lem:scalar-splitting}
For every $\lambda\in\C$, the map $L_\lambda$ is injective and
\[
 W=L_\lambda(V)\oplus\C z^{-2}\oplus\C\rho_\lambda,
\]
where
\[
 \rho_\lambda(z)=
 \begin{cases}
  z^{\lambda-1},&\lambda\in\mathbb Z_{\geq0},\\
  z^{-1},&\lambda\notin\mathbb Z_{\geq0}.
 \end{cases}
\]
In particular, $\dim_\C\operatorname{coker}L_\lambda=2$.
\end{lemma}

\begin{proof}
Write $P(z)=\sum_{n\geq0}p_nz^n$ and
$Q=q_{-2}z^{-2}+\sum_{k\geq-1}q_kz^k$.  Since
\begin{equation}\label{eq:L-monomials}
 L_\lambda(P)=\sum_{n\geq0}(n-\lambda)p_nz^{n-1},
\end{equation}
the coefficient of $z^{-2}$ gives one independent obstruction.
We construct the complementary coefficient and the framed primitive
without assuming that $\Delta$ contains a disc centred at the origin
of radius greater than one.

Suppose first that $\lambda\notin\mathbb Z_{\geq0}$.  Near $z=0$ the
series
\[
 P_0(z)=\sum_{n\geq0}\frac{q_{n-1}}{n-\lambda}z^n
\]
converges and satisfies $L_\lambda(P_0)=Q-q_{-2}z^{-2}$.
Indeed, on a sufficiently small closed disc about $0$, Cauchy estimates
for $z(Q-q_{-2}z^{-2})$, together with
$|n-\lambda|^{-1}=O(n^{-1})$, give normal convergence.
The germ $P_0$ continues along every path in $\Delta^\times$ by the
linear equation $P'-\lambda z^{-1}P=Q-q_{-2}z^{-2}$.
Continuation around a small loop about $0$ returns the original germ,
because that germ is already holomorphic across $0$.  Such a loop
generates $\pi_1(\Delta^\times)$; hence analytic continuation is
single-valued and produces $P_0\in\cO(\Delta)$.
Set $P=P_0-P_0(1),\qquad \beta=-\lambda P_0(1)$. Then $P(1)=0$ and
$Q=q_{-2}z^{-2}+L_\lambda(P)+\beta z^{-1}$.

If $\lambda=m\in\mathbb Z_{\geq0}$, the coefficient of $z^{m-1}$
in \eqref{eq:L-monomials} vanishes.  Put $\beta=q_{m-1}$ and
construct near $0$ the holomorphic solution $P_*$ of
\[
 L_m(P_*)=Q-q_{-2}z^{-2}-\beta z^{m-1}
\]
whose coefficient of $z^m$ is zero.  Its remaining Taylor
coefficients are $p_n=q_{n-1}/(n-m)$ for $n\ne m$.
The same Cauchy estimate gives local convergence, and the same
single-valued continuation argument extends $P_*$ to $\Delta$.
The framed primitive is $P=P_*-P_*(1)z^m$. It has $P(1)=0$ and the required image.  Notice that evaluation
of $P_*$ at $1$ takes place after analytic continuation; the Taylor
series at $0$ need not converge at $1$ for a disc not centred at $0$.

In either case the decomposition is direct. Suppose that
$L_\lambda(P)=\gamma\rho_\lambda$ with $P(1)=0$.
If $\lambda\notin\mathbb Z_{\geq0}$, comparison of the
coefficients of $z^{n-1}$ for $n\geq1$ shows that $P$ is constant;
thus $P=0$ and $\gamma=0$. If $\lambda=m\in\mathbb Z_{\geq0}$,
the coefficient of $z^{m-1}$ gives $\gamma=0$, and the homogeneous
equation then gives $P=cz^m$, again zero by the framed condition.
This also proves injectivity, and hence the direct sum, for every
disc $\Delta$ containing $0$ and $1$.
\end{proof}

\begin{lemma}\label{lem:rootwise-gauge}
Let $E_{ji}$ denote the matrix unit in row $j$ and column $i$, with $j<i$; define its height to be $i-j$, and set $\lambda_{ji}=\mu_i-\mu_j$.  For a holomorphic function $p$ with $p(1)=0$, the gauge $g=I_r+pE_{ji}$ changes the $(j,i)$-coefficient, modulo terms of height strictly greater than $i-j$, by
\[
 q_{ji}\longmapsto q_{ji}
 +p'+(\mu_j-\mu_i)z^{-1}p
 =q_{ji}+L_{\lambda_{ji}}(p).
\]
\end{lemma}

\begin{proof}
Since $E_{ji}^2=0$, one has $g^{-1}=I_r-pE_{ji}$.  The diagonal contribution to $g^{-1}(M_{\boldsymbol\mu}/z)g$ is
\[
 \frac{1}{z}[M_{\boldsymbol\mu},pE_{ji}]
 =\frac{\mu_j-\mu_i}{z}pE_{ji},
\]
while $g^{-1}g'$ contributes $p'E_{ji}$.  Commutators with the strictly upper triangular part of the connection have height strictly greater than $i-j$.  The stated formula follows.
\end{proof}

For the root $E_{ji}$, the resonance parameter is therefore $\mu_i-\mu_j$.  Reversing this sign does not change the dimension of the affine slice, but it changes the resonant monomial.  For example, when $(\mu_1,\mu_2)=(0,1)$, a constant upper-right coefficient is resonant and cannot be removed by a holomorphic framed gauge.

\subsection{The non-abelian unipotent normal form}\label{sec:unipotent}

For $1\leq h\leq r$, let
\[
 \mathfrak n^{(h)}
 =\operatorname{span}\{E_{ji}:j<i,\ i-j\geq h\}.
\]
Then
\[
 \mathfrak n=\mathfrak n^{(1)}\supset\mathfrak n^{(2)}\supset\cdots\supset\mathfrak n^{(r)}=0
\]
and
\begin{equation}\label{eq:height-bracket}
 [\mathfrak n^{(a)},\mathfrak n^{(b)}]
 \subseteq\mathfrak n^{(a+b)}.
\end{equation}

Define the affine slice
\[
 \cS_{\boldsymbol\mu}
 =\left\{
 d+\left(\frac{M_{\boldsymbol\mu}}{z}
 +\sum_{j<i}
 (\alpha_{ji}z^{-2}+\beta_{ji}\rho_{\lambda_{ji}})E_{ji}
 \right)dz
 \right\}.
\]

\begin{proposition}\label{prop:existence}
Every element of $\cA_{\boldsymbol\mu}$ is gauge equivalent under $\cU_1$ to an element of $\cS_{\boldsymbol\mu}$.
\end{proposition}

\begin{proof}
We construct the normalising gauge by induction on the height of a positive root.  Let $A^{(0)}$ be the given matrix.  Suppose that, after applying gauges of heights $1,\ldots,h-1$, the resulting matrix $A^{(h-1)}$ has all entries of height smaller than~$h$ in the chosen complementary spaces.

For each pair $j<i$ with $i-j=h$, write the height-$h$ coefficient of $A^{(h-1)}$ as $q_{ji}\in W$.  By \Cref{lem:scalar-splitting}, there are unique numbers $\alpha_{ji},\beta_{ji}\in\C$ and a unique $p_{ji}\in V$ such that
\[
 q_{ji}+L_{\lambda_{ji}}(p_{ji})
 =\alpha_{ji}z^{-2}+\beta_{ji}\rho_{\lambda_{ji}}.
\]
Set
\[
 P_h=\sum_{i-j=h}p_{ji}E_{ji},
 \qquad g_h=I_r+P_h.
\]
Because $P_h$ is strictly upper triangular, $g_h$ is invertible with a finite polynomial inverse; moreover $P_h(1)=0$, so $g_h(1)=I_r$.

Reduce the gauge action modulo $\mathfrak n^{(h+1)}$.  Every product containing two factors of height at least~$h$ has height at least $2h\geq h+1$, and \eqref{eq:height-bracket} shows that commutators of $P_h$ with the already present strictly upper triangular coefficients have height greater than~$h$.  Consequently the only contribution at height~$h$ is the linear rootwise term of \Cref{lem:rootwise-gauge}.  Thus $g_h$ puts every height-$h$ coefficient into the required complement and leaves all lower heights unchanged.

Define recursively $A^{(h)}=A^{(h-1)}\cdot g_h$.  The process stops after height $r-1$, and the finite product
\[
 g=g_1g_2\cdots g_{r-1}\in\cU_1
\]
normalises every positive-root coefficient.  No convergence issue occurs: the number of heights is finite, and each factor is a holomorphic unipotent matrix on~$\Delta$.  The resulting connection is an element of $\cS_{\boldsymbol\mu}$.
\end{proof}

\begin{proposition}\label{prop:uniqueness}
If two elements of $\cS_{\boldsymbol\mu}$ are related by a gauge in $\cU_1$, then the connections coincide and the gauge is the identity.
\end{proposition}

\begin{proof}
Let $A$ and $\widetilde A$ be two matrices in the slice and let $g\in\cU_1$ satisfy $A\cdot g=\widetilde A$.  We prove simultaneously, by increasing height, that the two matrices agree and that $g$ is the identity.
Assume that $g\neq I_r$, and let $h$ be the least integer for which the image of $g-I_r$ in $\mathfrak n^{(h)}/\mathfrak n^{(h+1)}$ is non-zero.  Write this image as $P_h=\sum_{i-j=h}p_{ji}E_{ji}$. Since $g(1)=I_r$, every $p_{ji}$ belongs to~$V$.  Reducing the gauge equation modulo $\mathfrak n^{(h)}$ shows first that $A$ and $\widetilde A$ agree in all smaller heights.  Reducing one step further, modulo $\mathfrak n^{(h+1)}$, eliminates every nonlinear commutator by \eqref{eq:height-bracket} and yields
\[
 \widetilde q_{ji}-q_{ji}=L_{\lambda_{ji}}(p_{ji})
 \qquad(i-j=h).
\]
The left-hand side lies in
$\C z^{-2}\oplus\C\rho_{\lambda_{ji}}$, because both matrices belong to the slice; the right-hand side lies in $L_{\lambda_{ji}}(V)$.  The direct decomposition of \Cref{lem:scalar-splitting} forces both sides to vanish.  Injectivity of $L_{\lambda_{ji}}$ then gives $p_{ji}=0$ for every root of height~$h$, contradicting the definition of~$h$.
Therefore $g=I_r$, and the gauge equation gives $A=\widetilde A$.  Hence the slice representative is unique and its fully framed stabiliser is trivial.
\end{proof}

\begin{theorem}\label{thm:slice}
For every residue vector $\boldsymbol\mu\in\C^r$, the map
\[
 \Theta_{\boldsymbol\mu}:
 \cS_{\boldsymbol\mu}(\C)\times\cU_1(\C)
 \longrightarrow\cA_{\boldsymbol\mu}(\C),
 \qquad
 (\sigma,g)\longmapsto\sigma\cdot g^{-1},
\]
is a bijection.  Equivalently, the action groupoid is equivalent to a discrete affine groupoid:
\[
 [\cA_{\boldsymbol\mu}(\C)/\cU_1(\C)]
 \simeq
 \cS_{\boldsymbol\mu}(\C)_{\mathrm{disc}}
 \simeq
 (\mathfrak n\oplus\mathfrak n)_{\mathrm{disc}}.
\]
Thus every orbit meets the slice in a single point and every stabiliser is trivial.
\end{theorem}

\begin{proof}
Existence and uniqueness are \Cref{prop:existence,prop:uniqueness}.  If $A\cdot g=\sigma$ is the unique normalisation of $A$, then $A=\sigma\cdot g^{-1}$; conversely, uniqueness of the normalising gauge makes this factorisation unique.  The coefficients $(\alpha_{ji})_{j<i}$ and $(\beta_{ji})_{j<i}$ identify the slice with two copies of~$\mathfrak n$.
\end{proof}

\subsection{Proof of the orbit classification}\label{sec:orbits}

\begin{proof}[Proof of \Cref{thm:main}]
The residue functor and its essential surjectivity are established in \Cref{prop:residue-map}.  For a fixed residue vector $\boldsymbol\mu$, \Cref{prop:diagonal-reduction} identifies the fibre groupoid with the action groupoid $[\cA_{\boldsymbol\mu}(\C)/\cU_1(\C)]$.  The unique-slice theorem \Cref{thm:slice} identifies this groupoid with the discrete affine groupoid $\mathfrak n\oplus\mathfrak n$ and supplies the unique representative \eqref{eq:main-normal-form}.  Conversely, \Cref{prop:admissibility} shows that every point of the slice is an admissible object with the required diagonal monodromy.  The equivalence on each residue fibre is induced by the unique normalisation, and all automorphism groups are trivial.
\end{proof}

\begin{corollary}\label{cor:special}
Every fully framed residue fibre is a discrete groupoid equivalent to the complex points of $\mathbb A^{r(r-1)}_{\C}$.  In particular, unit eigenvalues and repeated diagonal monodromy eigenvalues change neither the affine parameter count nor the stabiliser groups: the latter are always trivial.
\end{corollary}

\begin{proof}
This follows from \Cref{thm:main}, since $2\dim\mathfrak n=r(r-1)$.
\end{proof}

Although all fixed-residue strata have the same affine type, their universal normal-form connections and monodromy maps depend on the residue differences, as the rank-two calculation below illustrates.

\section{The rank-two resonant case}\label{sec:rank-two}

Assume $r=2$ and the two prescribed diagonal monodromy eigenvalues are equal to $c$.  Fix $\mu_1,\mu_2\in\Logm(c)$ and set $\lambda=\mu_2-\mu_1\in\mathbb Z$. The normal form is
\begin{equation}\label{eq:rank-two-form}
 \nabla_{\alpha,\beta}
 =d+
 \begin{pmatrix}
  \mu_1z^{-1}&q_{\alpha,\beta}(z)\\
  0&\mu_2z^{-1}
 \end{pmatrix}dz,
 \qquad
 q_{\alpha,\beta}=\alpha z^{-2}+\beta\rho_\lambda.
\end{equation}

\begin{proposition}\label{prop:rank-two-monodromy}
With the horizontal convention fixed in the introduction, the monodromy of \eqref{eq:rank-two-form} with respect to the standard frame at $z=1$ is
\[
 T_{\alpha,\beta}
 =c\begin{pmatrix}
  1&-2\pi\ii\,\nu_\lambda(\alpha,\beta)\\
  0&1
 \end{pmatrix},
\]
where
\begin{equation}\label{eq:nu}
 \nu_\lambda(\alpha,\beta)
 =\res_0\bigl(q_{\alpha,\beta}(z)z^{-\lambda}\,dz\bigr)
 =
 \begin{cases}
  \beta,&\lambda\geq0,\\
  \alpha,&\lambda=-1,\\
  0,&\lambda\leq-2.
 \end{cases}
\end{equation}
\end{proposition}

\begin{proof}
On a simply connected sector containing the base point $z=1$, choose branches of $z^{-\mu_1}$ and $z^{-\mu_2}$ normalised to equal $1$ at $z=1$, and choose the integration constant below so that $f(1)=0$.  Then the resulting fundamental solution satisfies $Y(1)=I_2$ and may be written as
\[
 Y(z)=
 \begin{pmatrix}
  z^{-\mu_1}&z^{-\mu_1}f(z)\\
  0&z^{-\mu_2}
 \end{pmatrix}.
\]
The equation $dY+A Y\,dz=0$ gives $f'(z)=-q_{\alpha,\beta}(z)z^{-\lambda}$. Since $\lambda\in\mathbb Z$, the integrand is single-valued on the punctured disc.  After one positive circuit, $f$ changes by
\[
 -\oint q_{\alpha,\beta}(z)z^{-\lambda}\,dz
 =-2\pi\ii\,\nu_\lambda(\alpha,\beta).
\]
Both diagonal factors acquire the same multiplier $c$.  Right multiplication of $Y$ by the displayed matrix therefore gives its analytic continuation.
If $\lambda=m\geq0$, then
\[
 q_{\alpha,\beta}z^{-\lambda}
 =\alpha z^{-m-2}+\beta z^{-1},
\]
so the residue is $\beta$.  If $\lambda=-1$, then
\[
 q_{\alpha,\beta}z^{-\lambda}=\alpha z^{-1}+\beta,
\]
and the residue is $\alpha$.  If $\lambda\leq-2$, neither term has exponent $-1$, and the residue vanishes.
\end{proof}

\begin{corollary}\label{cor:rank-two-strata}
Every repeated-eigenvalue residue stratum is the affine plane $\mathbb A^2_{\alpha,\beta}$, but its Jordan stratification depends on $\lambda$:
\begin{enumerate}[(i)]
 \item if $\lambda\geq0$, the non-scalar locus is $\{\beta\neq0\}$;
 \item if $\lambda=-1$, the non-scalar locus is $\{\alpha\neq0\}$;
 \item if $\lambda\leq-2$, the monodromy is scalar throughout the stratum.
\end{enumerate}
Thus isomorphic affine orbit spaces may carry different universal connections and different monodromy maps.
\end{corollary}

\begin{proof}
The affine-plane statement is \Cref{thm:main} in rank two.  By \Cref{prop:rank-two-monodromy}, the monodromy is non-scalar if and only if $\nu_\lambda(\alpha,\beta)\neq0$.  Formula~\eqref{eq:nu} identifies this condition with $\beta\neq0$ for $\lambda\geq0$ and with $\alpha\neq0$ for $\lambda=-1$; for $\lambda\leq-2$, the non-scalar locus is empty.
\end{proof}

The last case does not contradict the existence of Jordan monodromy with eigenvalue~$c$.  It says only that such monodromy cannot occur in residue strata with $\lambda\leq-2$.  This is the normal-form manifestation of the constraint in \Cref{rem:log-constraint}.

\section{Changes of framing and unframed quotients}\label{sec:forgetting-frame}

The geometry of invariant flags, Borel quotients and Springer fibres used below is classical; see Springer~\cite{Springer}, and compare the parahoric Riemann--Hilbert description of Boalch~\cite{Boalch}.  This section identifies what is forgotten in passing from the fully framed Kontsevich--Szab\'o slice to flagged and unframed objects, especially for repeated spectra and spectra containing the eigenvalue~$1$.

\subsection{Re-framing and the obstruction to intrinsic affine coordinates}\label{sec:reframing}

The affine slice of \Cref{thm:main} classifies fully framed objects.  We next examine the effect of changing the ordered frame while keeping the underlying meromorphic connection fixed.  This is the local analogue of passing from framed to parabolic or parahoric Riemann--Hilbert data~\cite{Boalch,Inaba}; it explains the dependence of normal-form coordinates on a chosen frame.

Let $\nabla_{\boldsymbol\alpha,\boldsymbol\beta}$ be the standard-framed normal form \eqref{eq:main-normal-form}.  Replace the standard row frame $\mathbf e(1)$ by $\mathbf e(1)C$, where $C\in B$.  Returning the new framed object to the standard frame uses the constant gauge $C^{-1}$; hence the connection matrix becomes
\begin{equation}\label{eq:reframing-conjugation}
 A(z)\longmapsto C A(z) C^{-1}.
\end{equation}
The resulting matrix is again upper triangular, and it may then be returned to the distinguished slice by the unique framed unipotent normalisation of \Cref{thm:slice}.

\begin{proposition}\label{prop:torus-reframing}
Let $D=\diag(d_1,\ldots,d_r)\in B$ be diagonal.  The equivalence relation on the affine slice induced by replacing the ordered frame by $\mathbf e(1)D$ is
\begin{equation}\label{eq:torus-reframing}
 \alpha_{ji}\longmapsto \frac{d_j}{d_i}\alpha_{ji},
 \qquad
 \beta_{ji}\longmapsto \frac{d_j}{d_i}\beta_{ji}
 \qquad (j<i).
\end{equation}
Consequently, none of the individual affine coordinates descends to the object obtained by forgetting the ordered framing.
\end{proposition}

\begin{proof}
Since $D$ commutes with $M_{\boldsymbol\mu}$ and
\[
 D E_{ji}D^{-1}=\frac{d_j}{d_i}E_{ji},
\]
formula \eqref{eq:reframing-conjugation} sends the normal form directly to another point of the same distinguished slice, with coordinates given by \eqref{eq:torus-reframing}.  No further unipotent normalisation is required.  Choosing $D$ with $d_j/d_i\neq1$ proves that a non-zero coordinate cannot be invariant under change of framing.
\end{proof}

Let $H\cong(\Gm)^r\subset B$ denote the diagonal torus.

\begin{corollary}\label{cor:torus-invariants}
The ring of polynomial invariants of the diagonal re-framing torus on the affine slice is trivial:
\[
 \C[\mathfrak n\oplus\mathfrak n]^H=\C,
\]
where the central scalar subgroup acts trivially and $H/\Gm$ acts effectively.  Hence the affine categorical quotient of the slice by diagonal re-framing is a point.
\end{corollary}

\begin{proof}
Consider the one-parameter subgroup $D(t)=\diag(t^{-1},t^{-2},\ldots,t^{-r})$. For every $j<i$, both $\alpha_{ji}$ and $\beta_{ji}$ have strictly positive weight $i-j$ under \eqref{eq:torus-reframing}.  Every non-constant monomial therefore has strictly positive weight and cannot be invariant.  Thus only constants are invariant under this one-parameter subgroup, and hence under the full diagonal torus.
\end{proof}

The affine categorical quotient loses the orbit parameters, whereas the quotient stack $[(\mathfrak n\oplus\mathfrak n)/H]$ retains their stabilisers and root weights. Forgetting the full frame introduces a larger re-framing group than $H$.

The additional unipotent ambiguity can be described in rank two.  Put
\[
 A_{\alpha,\beta}(z)
 =\frac{1}{z}\begin{pmatrix}\mu_1&0\\0&\mu_2\end{pmatrix}
 +\bigl(\alpha z^{-2}+\beta\rho_\lambda(z)\bigr)E_{12},
 \qquad \lambda=\mu_2-\mu_1.
\]
For $s\in\C$, set $N_s=I_2+sE_{12}$.

\begin{proposition}\label{prop:rank-two-unipotent-reframing}
The equivalence relation induced by replacing the standard frame by $\mathbf e(1)N_s$ and returning to the distinguished slice is
\[
 (\alpha,\beta)\longmapsto
 \begin{cases}
  (\alpha,\beta),&\lambda\in\mathbb Z_{\geq0},\\[1mm]
  (\alpha,\beta+s\lambda),&\lambda\notin\mathbb Z_{\geq0}.
 \end{cases}
\]
In particular, when $\lambda\notin\mathbb Z_{\geq0}$, the parameter $\beta$ can be removed after the ordered framing is forgotten.
\end{proposition}

\begin{proof}
A direct calculation gives
\[
 N_s A_{\alpha,\beta}(z)N_s^{-1}
 =A_{\alpha,\beta}(z)+s\lambda z^{-1}E_{12}.
\]
It remains to project $s\lambda z^{-1}$ to the scalar normal-form complement of \Cref{lem:scalar-splitting}.  If $\lambda\notin\mathbb Z_{\geq0}$, then $\rho_\lambda=z^{-1}$, so this changes $\beta$ by $s\lambda$.  If $\lambda=0$, the added term vanishes.  If $\lambda=m\geq1$, then
\[
 z^{-1}=L_m\!\left(\frac{z^m-1}{m}\right),
\]
with $(z^m-1)/m$ vanishing at $z=1$; hence the added term is framed-gauge trivial and the distinguished coordinates do not change.
\end{proof}

\begin{corollary}\label{cor:rank-two-unframed-orbits}
In rank two, after allowing upper triangular re-framing, the affine coordinates are never canonical.  If $\lambda\notin\mathbb Z_{\geq0}$, every re-framing class contains a representative with $\beta=0$, and the remaining diagonal scaling has only the two set-theoretic orbits $\alpha=0$ and $\alpha\neq0$.  If $\lambda\in\mathbb Z_{\geq0}$, the unipotent re-framing acts trivially on the distinguished slice, while diagonal scaling identifies non-zero scalar multiples of $(\alpha,\beta)$.
\end{corollary}

\begin{proof}
Combine the unipotent translation of \Cref{prop:rank-two-unipotent-reframing} with the common weight-one diagonal action of \Cref{prop:torus-reframing}.  In the non-resonant range the additive parameter sets $\beta$ equal to zero, after which diagonal scaling has the two stated set-theoretic orbits in the remaining coordinate.  In the resonant range the unipotent action is trivial and the diagonal torus identifies the non-zero scalar multiples.
\end{proof}

The results above show that the first descent problem has no coordinatewise solution.  The coordinates $(\boldsymbol\alpha,\boldsymbol\beta)$ are invariants of a fully framed residual constituent, not of a constituent after the ordered basis has been forgotten.  There are, however, two distinct forgetting operations.  Quotienting by the standard Borel forgets the ordered basis while retaining the full invariant flag that it spans.  Forgetting the flag itself is a further operation and is not represented by the same quotient in the presence of repeated eigenvalues.  The next section describes this distinction.

\subsection{The flagged quotient and its structural interpretation}\label{sec:intrinsic-quotient}

Fix a residue vector $\boldsymbol\mu$ and write $\cS_{\boldsymbol\mu}=\mathfrak n\oplus\mathfrak n$ for the affine normal-form slice of \Cref{thm:main}.  The passage from a basis to the full flag it spans is classical, as is the appearance of parabolic or Borel data in Riemann--Hilbert and moduli-stack descriptions; compare~\cite{IIS,Inaba,Boalch,Bischoff}.  What is specific here is the induced action on the affine slice and its interaction with resonance.  The construction of \Cref{sec:reframing} assigns to each $C\in B$ and $x\in\cS_{\boldsymbol\mu}$ a new point of the same slice: change the standard frame by $C$, conjugate the connection by $C$, and apply the unique framed unipotent normalisation.  We denote this action by
\[
 B\times\cS_{\boldsymbol\mu}\longrightarrow\cS_{\boldsymbol\mu},
 \qquad
 (C,x)\longmapsto C\star x.
\]
The central scalar subgroup $\Gm\subset B$ acts trivially.

\begin{proposition}\label{prop:algebraic-borel-action}
For every fixed residue vector $\boldsymbol\mu$, the operation $\star$ is an algebraic action of $B$ on the affine space $\cS_{\boldsymbol\mu}$.  With respect to the height filtration of $\mathfrak n$, the action is triangular: the transformation of a coordinate of height $h$ depends only on the entries of $C$ and on normal-form coordinates of height at most $h$.  Its restriction to the diagonal torus is the root-weight action of \eqref{eq:torus-reframing}.
\end{proposition}

\begin{proof}
Conjugation by a constant element of $B$ is regular in the entries of $C$, in the inverses of its diagonal entries, and in the affine coordinates of the slice.  It preserves upper triangularity and the diagonal residue vector.  Starting from a point of $\cS_{\boldsymbol\mu}$, constant conjugation produces only a finite linear combination of the Laurent monomials already present together with the diagonal $z^{-1}$-term.  On this finite-dimensional coefficient space, the projection furnished by \Cref{lem:scalar-splitting} is a fixed complex-linear map depending only on the constants $\mu_i-\mu_j$.  The height-by-height normalisation therefore uses only finitely many additions, multiplications, and inversions of the diagonal entries of $C$.  Inductively, every normalised coordinate is a regular function on $B\times\cS_{\boldsymbol\mu}$.  Uniqueness of the distinguished normal form shows that the transformation associated with $C_1C_2$ agrees with the composite of those associated with $C_2$ and $C_1$.  Thus $\star$ is an algebraic group action.  The filtration statement follows from
\[
 [\mathfrak n^{(a)},\mathfrak n^{(b)}]\subseteq\mathfrak n^{(a+b)},
\]
and the torus formula is \Cref{prop:torus-reframing}.
\end{proof}

\begin{definition}\label{def:flagged-stack}
The intrinsic \emph{flagged} orbit stack in the residue stratum $\boldsymbol\mu$ is
\[
 \mathscr A^{\mathrm{flag}}_{\boldsymbol\mu}
 \coloneqq[\cS_{\boldsymbol\mu}/B].
\]
Its effective rigidification with respect to the trivially acting central scalar subgroup is $[\cS_{\boldsymbol\mu}/(B/\Gm)]$.
\end{definition}

\begin{proposition}\label{prop:meaning-flagged-stack}
The stack $\mathscr A^{\mathrm{flag}}_{\boldsymbol\mu}$ classifies, on complex points, prescribed-lattice connections in the residue stratum $\boldsymbol\mu$ together with a marked invariant full flag, but without a chosen ordered basis inducing that flag.  It is thus the quotient obtained from the fully framed Kontsevich--Szab\'o problem by forgetting the ordered basis while retaining its flag.
\end{proposition}

\begin{proof}
Two ordered bases induce the same full flag if and only if they differ by right multiplication by an element of $B$.  The fully framed classification identifies every such object with a unique point of $\cS_{\boldsymbol\mu}$, and changing the basis within its flag induces the action $\star$.  Passing to the action groupoid therefore forgets the basis and remembers the flag.  Stabiliser groups are retained by the quotient stack, as required when the connection admits non-trivial flag-preserving automorphisms after the framing has been removed.
\end{proof}

The quotient by $B$ retains an invariant full flag. An unflagged constituent requires quotienting further by arbitrary changes of basis; the possible flags form the monodromy-invariant flag scheme described in \Cref{sec:unflagged-descent}, with the connection automorphism group acting on it. In particular, repeated eigenvalues may give positive-dimensional Springer fibres.

We can compute the flagged quotient in rank two.  Put $B_{\mathrm{eff}}=B/\Gm\cong\Gm\ltimes\mathbb G_a$. Write an element as $D(t)N_s$, where $t=d_1/d_2$, and retain the notation $\lambda=\mu_2-\mu_1$.

\begin{theorem}\label{thm:rank-two-flagged-quotient}
On the rank-two affine slice $\mathbb A^2_{\alpha,\beta}$, the effective Borel action is
\[
 (t,s)\star(\alpha,\beta)=
 \begin{cases}
  (t\alpha,t\beta),&\lambda\in\mathbb Z_{\geq0},\\[1mm]
  (t\alpha,t(\beta+s\lambda)),&\lambda\notin\mathbb Z_{\geq0}.
 \end{cases}
\]
Consequently:
\begin{enumerate}[(i)]
 \item if $\lambda\notin\mathbb Z_{\geq0}$, then
 $
  [\mathbb A^2/B_{\mathrm{eff}}]\simeq[\mathbb A^1_{\alpha}/\Gm];
 $
 \item if $\lambda\in\mathbb Z_{\geq0}$, the unipotent radical acts trivially, and the rigidification of the quotient by that trivially acting subgroup is
 \[
  [\mathbb A^2_{\alpha,\beta}/\Gm],
 \]
 with both coordinates of weight one.  Its non-zero locus has geometric quotient $\mathbb P^1$.
\end{enumerate}
In both cases the affine categorical quotient is a point, but the quotient stack retains non-trivial orbit and stabiliser data.
\end{theorem}

\begin{proof}
The displayed action combines \Cref{prop:torus-reframing,prop:rank-two-unipotent-reframing}; the order $D(t)N_s$ gives the common factor $t$.  If $\lambda\notin\mathbb Z_{\geq0}$, the additive subgroup acts freely and transitively on every affine line with fixed $\alpha$, and every orbit meets $\{\beta=0\}$ once.  The residual action on that section is $\alpha\mapsto t\alpha$, which gives the first equivalence of quotient stacks.  If $\lambda\in\mathbb Z_{\geq0}$, the additive subgroup acts trivially by \Cref{prop:rank-two-unipotent-reframing}; rigidifying it leaves the simultaneous weight-one action of $\Gm$.  The quotient of $\mathbb A^2\setminus\{0\}$ by this action is $\mathbb P^1$.  Finally, in either case all non-constant polynomial functions have positive diagonal weight after the free additive direction is removed, so the affine invariant ring consists only of constants.
\end{proof}

\begin{corollary}\label{cor:precise-recipient}
For the local moduli problem defined by Kontsevich and Szab\'o, the quotient stack $[\cS_{\boldsymbol\mu}/B]$ is the natural recipient for intrinsic data after the ordered basis is forgotten but the invariant full flag is retained.  When eigenvalues repeat, this stack is not, in general, equivalent to the unflagged moduli problem.  The latter requires, in addition, the descent along the forgetful morphism from marked invariant flags to the underlying residual connection.
\end{corollary}

\begin{proof}
The first assertion is \Cref{prop:meaning-flagged-stack}.  The second follows from the description of the further forgetful fibre in \Cref{prop:fibre-forgetful-map}: when repeated eigenvalues occur, the scheme of invariant full flags need not be a point, so quotienting only by $B$ cannot forget the flag itself.
\end{proof}

\subsection{Invariant monodromy flags and the unflagged descent}\label{sec:unflagged-descent}

The full flag occurring in the original groupoid is a flag in the fibre at $z=1$ preserved by the monodromy.  It is not a constant flag in a chosen trivialisation which must be preserved by every Laurent coefficient of the connection matrix.  Accordingly, the scheme of possible invariant flags is governed by the monodromy operator itself; the corresponding 2-fibre also records the action of the automorphism group of the underlying connection.

Let $V$ be a complex vector space of dimension $r$, let $T\in\GL(V)$, and let $\mathbf c=(c_1,\ldots,c_r)\in(\C^\times)^r$ be an ordered list whose multiset is the spectrum of $T$, counted with algebraic multiplicity.

\begin{definition}\label{def:invariant-flag-scheme}
The ordered invariant-flag functor $\operatorname{Flag}_{\mathbf c}(T)$ assigns to a complex scheme $S$ the set of full flags of subbundles
\[
 0=\cF_0\subset\cF_1\subset\cdots\subset\cF_r=V\otimes_\C\cO_S
\]
such that $(T\otimes1)(\cF_i)\subseteq\cF_i$ for every $i$, and the endomorphism induced by $T\otimes1$ on the line bundle $\cF_i/\cF_{i-1}$ is multiplication by $c_i$.
\end{definition}

\begin{proposition}\label{prop:invariant-flags-projective}
The functor $\operatorname{Flag}_{\mathbf c}(T)$ is represented by a closed projective subscheme of the full flag variety $\operatorname{Flag}(V)$.
\end{proposition}

\begin{proof}
Let
\[
 0=\cF_0^{\mathrm{univ}}\subset\cF_1^{\mathrm{univ}}\subset\cdots\subset
 \cF_r^{\mathrm{univ}}=V\otimes\cO_{\operatorname{Flag}(V)}
\]
be the universal flag.  The condition that $T$ preserve $\cF_i^{\mathrm{univ}}$ is the vanishing of the bundle morphism
\[
 \cF_i^{\mathrm{univ}}
 \xrightarrow{\ T\ }
 V\otimes\cO_{\operatorname{Flag}(V)}
 \longrightarrow
 (V\otimes\cO_{\operatorname{Flag}(V)})/\cF_i^{\mathrm{univ}}.
\]
It therefore cuts out a closed subscheme.  On that closed subscheme, the requirement that the induced endomorphism of the universal line quotient $\cL_i=\cF_i^{\mathrm{univ}}/\cF_{i-1}^{\mathrm{univ}}$ be $c_i\operatorname{Id}_{\cL_i}$ is again a closed condition, namely the vanishing of an endomorphism of a line bundle.  Intersecting these finitely many zero schemes gives $\operatorname{Flag}_{\mathbf c}(T)$.  Since the full flag variety is projective, so is this closed subscheme.
\end{proof}

Let the distinct eigenvalues of $T$ be $a\in\Sigma$.  Write
\[
 V=\bigoplus_{a\in\Sigma}V_a,
 \qquad
 T|_{V_a}=aI_{V_a}+N_a,
\]
where $V_a$ is the generalised $a$-eigenspace and $N_a$ is nilpotent.  Denote by $\mathcal B(N_a)$ the Springer fibre of $N_a$, namely the scheme of full flags in $V_a$ preserved by $N_a$; see~\cite{Springer}.

\begin{theorem}\label{thm:product-springer-fibres}
There is an isomorphism of projective schemes
\[
 \operatorname{Flag}_{\mathbf c}(T)
 \cong
 \prod_{a\in\Sigma}\mathcal B(N_a).
\]
The isomorphism depends on the prescribed order $\mathbf c$ only through the positions at which each eigenvalue occurs.
\end{theorem}

\begin{proof}
For each $a\in\Sigma$, let $e_a=p_a(T)$ be the spectral idempotent projecting $V$ onto $V_a$.  If $W\subseteq V\otimes\cO_S$ is a $T$-stable subbundle, then every polynomial in $T$ preserves $W$.  The restriction $e_a|_W$ is an idempotent $\cO_S$-endomorphism, so its image $e_a(W)$ is a locally free direct summand of $W$, and
\[
 W=\bigoplus_{a\in\Sigma} e_a(W).
\]
Apply this to every member $\cF_i$ of a flag in $\operatorname{Flag}_{\mathbf c}(T)(S)$.  Let $m_a(i)=\#\{k\leq i:c_k=a\}$. On the line quotient $\cF_k/\cF_{k-1}$, the idempotent $e_a=p_a(T)$ acts as multiplication by $p_a(c_k)$, which is $1$ if $c_k=a$ and $0$ otherwise.  It follows inductively that $e_a\cF_i$ has rank $m_a(i)$.  Removing repetitions as $i$ varies produces a full flag
\[
 0\subset\cG_{a,1}\subset\cdots\subset\cG_{a,\dim V_a}=V_a\otimes\cO_S
\]
which is preserved by $T|_{V_a}$, equivalently by $N_a$.

Conversely, given an $N_a$-stable full flag $(\cG_{a,k})_k$ in every $V_a\otimes\cO_S$, define
\[
 \cF_i=\bigoplus_{a\in\Sigma}\cG_{a,m_a(i)},
 \qquad
 \cG_{a,0}=0.
\]
This is a full flag, it is preserved by $T$, and $T$ acts on $\cF_i/\cF_{i-1}$ as multiplication by $c_i$.  The two constructions are inverse and are functorial in $S$, hence define the asserted scheme isomorphism.
\end{proof}

\begin{corollary}\label{cor:regular-semisimple-flag}
If $T$ is regular semisimple and the ordering of its distinct eigenvalues is fixed, then $\operatorname{Flag}_{\mathbf c}(T)$ is a single reduced point.  If $T=cI_V$, then
\[
 \operatorname{Flag}_{(c,\ldots,c)}(T)=\operatorname{Flag}(V).
\]
If $T$ has a single eigenvalue and one Jordan block, then the invariant-flag scheme is a single reduced point.
\end{corollary}

\begin{proof}
In the regular semisimple case every generalised eigenspace is one-dimensional, so every factor in \Cref{thm:product-springer-fibres} is a point.  For a scalar operator every flag is invariant.  A regular nilpotent operator has a unique stable full flag, namely the flag by kernels of its successive powers; its Springer fibre is a reduced point.
\end{proof}

We now describe the fibre of the forgetting morphism.

\begin{definition}\label{def:flagged-unflagged-groupoids}
Let $\mathscr A_{\mathbf c}^{\mathrm{unflag}}$ be the groupoid whose objects are pairs $(E,\nabla)$ satisfying the connection-theoretic conditions of \Cref{def:groupoid} and whose monodromy has characteristic polynomial $\prod_{i=1}^r(t-c_i)$, with morphisms given by holomorphic gauges.  Let $\mathscr A_{\mathbf c}^{\mathrm{flag}}$ be the groupoid of triples $(E,\nabla,F_\bullet)$, where $(E,\nabla)$ is such an object and $F_\bullet$ is a full flag in $E|_{z=1}$ represented by a point of $\operatorname{Flag}_{\mathbf c}(T_\nabla)$.  Morphisms are gauges whose value at $z=1$ carries the marked flag to the marked flag.  There is a forgetful morphism
\[
 \pi_{\mathrm{forg}}:\mathscr A_{\mathbf c}^{\mathrm{flag}}
 \longrightarrow
 \mathscr A_{\mathbf c}^{\mathrm{unflag}}.
\]
\end{definition}

Let $x=(E,\nabla)$ be an unflagged residual constituent and let $T_x$ be its monodromy on $E|_{z=1}$.  Let $\Aut(x)$ be the group of holomorphic gauge automorphisms of $(E,\nabla)$.  Evaluation at $z=1$ gives an injective homomorphism
\[
 \operatorname{ev}_1:\Aut(x)\hookrightarrow Z_{\GL(E|_1)}(T_x),
\]
because an automorphism is horizontal, is determined by its value at one point, and its value commutes with monodromy.

\begin{proposition}\label{prop:fibre-forgetful-map}
The 2-fibre of $\pi_{\mathrm{forg}}$ over the unflagged object $x=(E,\nabla)$ is equivalent to the action groupoid
\[
 \left[
 \operatorname{Flag}_{\mathbf c}(T_x)/\Aut(x)
 \right],
\]
where $\Aut(x)$ acts through evaluation at $z=1$.
\end{proposition}

\begin{proof}
An object of the 2-fibre is a full flag in $E|_{z=1}$ preserved by $T_x$, with the prescribed ordered action $c_i$ on the successive quotients.  Thus its objects are the points of $\operatorname{Flag}_{\mathbf c}(T_x)$.  A morphism between two such lifts is an automorphism of $(E,\nabla)$ carrying one flag to the other.  Since a horizontal automorphism is determined by its value at $z=1$, this is the action groupoid displayed above.
\end{proof}

\begin{corollary}\label{cor:rank-two-unflagged-fibre}
Let $r=2$ and suppose that the characteristic polynomial of $T$ is $(t-c_1)(t-c_2)$.
\begin{enumerate}[(i)]
 \item If $c_1\neq c_2$, then $\operatorname{Flag}_{(c_1,c_2)}(T)$ is the single reduced point $\mathbb P(\ker(T-c_1I))$.
 \item If $c_1=c_2=c$ and $T\neq cI$, then $T=cI+N$ with $N\neq0$ and $N^2=0$, and
 \[
  \operatorname{Flag}_{(c,c)}(T)=\mathbb P(\ker N)
 \]
 is a single reduced point.
 \item If $T=cI$, then
 \[
  \operatorname{Flag}_{(c,c)}(T)=\mathbb P^1.
 \]
\end{enumerate}
Consequently, in rank two the choice of invariant flag is non-trivial only on the scalar-monodromy locus; there the parameter scheme before dividing by automorphisms is $\mathbb P^1$.
\end{corollary}

\begin{proof}
A full flag in dimension two is a line $L\subset V$.  The ordered condition requires $T|_L=c_1I_L$.  If $c_1\neq c_2$, the $c_1$-eigenspace is one-dimensional and gives the unique line.  If $c_1=c_2=c$ and $T\neq cI$, the nilpotent operator $N=T-cI$ has rank one and its kernel is the unique line on which $T$ acts as $c$.  If $T=cI$, every line is admissible.  The equations are linear in homogeneous coordinates, so the unique-point schemes in the first two cases are reduced.
\end{proof}

The invariant-flag scheme is determined by the monodromy on the framed fibre.  Requiring a constant flag in a chosen trivialisation to be preserved by every Laurent coefficient of a connection matrix would be a strictly stronger, gauge-dependent condition and is not the forgetting problem attached to the groupoid of Kontsevich and Szab\'o.

\subsection{The fibre of the framing map}\label{sec:rigidification-torsor}

Let $T\in\GL(V)$ and let $\mathbf c$ order its eigenvalues. The frames in which $T$ is upper triangular with diagonal $\mathbf c$ form
$\Fr_{\mathbf c}(T)=\{\phi\in\Isom(\C^r,V):\phi^{-1}T\phi\in B_{\mathbf c}\}$,
where $B_{\mathbf c}=\{b\in B:b_{ii}=c_i\}$. If $x=(E,\nabla)$ is an unflagged connection, write $T_x$ for its monodromy and $\Aut(x)$ for its holomorphic gauge automorphism group.

\begin{proposition}\label{prop:fully-framed-fibre}
The projection $\Fr_{\mathbf c}(T)\to\operatorname{Flag}_{\mathbf c}(T)$ is a principal right $B$-bundle. In particular it is a $B$-torsor over a point precisely when the ordered invariant flag is unique. The $2$-fibre over $x$ of the morphism forgetting the ordered frame is
\[
 \left[\Fr_{\mathbf c}(T_x)/\Aut(x)\right].
\]
Forgetting the basis while retaining its flag induces the morphism to $[\operatorname{Flag}_{\mathbf c}(T_x)/\Aut(x)]$. By \Cref{thm:product-springer-fibres} the base of the principal bundle is a product of Springer fibres; if $T=cI$, it is the full flag variety of $V$.
\end{proposition}

\begin{proof}
A frame $\phi:\C^r\to V$ belongs to $\Fr_{\mathbf c}(T)$ precisely when the flag $\phi(\C^i)$ is $T$-invariant and $T$ acts by $c_i$ on its $i$-th line quotient. Thus $\Fr_{\mathbf c}(T)$ is the inverse image of $\operatorname{Flag}_{\mathbf c}(T)$ under the principal $B$-bundle $\Isom(\C^r,V)\to\operatorname{Flag}(V)$. A framed lift of $x$ is a frame $\phi\in\Fr_{\mathbf c}(T_x)$; its arrows to $\phi'$ are the automorphisms $a\in\Aut(x)$ satisfying $\operatorname{ev}_1(a)\phi=\phi'$. This gives the stated action groupoid.

The right action $\phi\cdot b=\phi b$ of $B$ changes a frame without changing its flag. It commutes with the left action $a\cdot\phi=\operatorname{ev}_1(a)\phi$ of $\Aut(x)$. Hence $\phi\mapsto(\phi(\C^i))_i$ is $\Aut(x)$-equivariant and induces the asserted morphism between action groupoids. The description of its base follows from \Cref{thm:product-springer-fibres}.
\end{proof}

\section{Isomonodromic families of lattices}\label{sec:isomonodromic-lattices}

The order-two part of the affine slice records lattice moduli that are invisible to the punctured connection and its monodromy.  The rank-two family is the simplest instance of the following general construction.

Let $S$ be a complex analytic space, let $V$ be a finite-dimensional complex vector space, and let $R\in\End(V)$.  Let $\mathfrak a\subset\End(V)$ be a finite-dimensional commutative associative subalgebra, not necessarily unital, consisting of nilpotent endomorphisms, and assume that
\[
 [R,\mathfrak a]\subset\mathfrak a.
\]
The algebra $\mathfrak a$ is nilpotent.  To see this, choose a basis $A_1,\ldots,A_m$ and integers $N_i\geq1$ such that $A_i^{N_i}=0$.  By commutativity, every product of more than $\sum_i(N_i-1)$ basis elements contains some factor $A_i^{N_i}$ and therefore vanishes.  Consequently, for every holomorphic map $F:S\to\mathfrak a$, the exponential $G(s,u)=\exp\!\left(\frac{F(s)}{u}\right)$ is a finite sum, with a truncation order independent of $s$, and hence a meromorphic gauge on $S\times\Delta$.

\begin{theorem}\label{thm:isomonodromic-lattice-family}
The gauge transform of the constant logarithmic connection $\nabla_0=d+\frac{R}{u}\,du$ by $G$ is
\begin{equation}\label{eq:isomonodromic-lattice-family}
 \nabla_{R,F}
 =d+\frac{d_S F}{u}+\frac{R}{u}\,du
 +\frac{[R,F]-F}{u^2}\,du.
\end{equation}
It is flat.  Moreover, for every vector field $\xi$ on $S$,
\[
 (\nabla_{R,F})_{u\xi}=u\xi+\xi(F),
 \qquad
 (\nabla_{R,F})_{u^2\partial_u}
 =u^2\partial_u+uR+[R,F]-F,
\]
so the standard lattice satisfies the usual $F$-bundle pole conditions.
For every $s\in S$, the fibre connection
\[
 \nabla_s=d+\left(\frac{R}{u}+
 \frac{(\operatorname{ad}_R-I)F(s)}{u^2}\right)du
\]
is meromorphically gauge equivalent to $d+R\,du/u$.  In particular, all punctured fibres are meromorphically isomorphic to the fixed connection $d+R\,du/u$, and their monodromy conjugacy classes are constant, whereas the prescribed-lattice coefficient $K_{-2}(s)=(\operatorname{ad}_R-I)F(s)$ may vary.
\end{theorem}

\begin{proof}
Put $X=F/u$, so $G=e^X$.  Since $\mathfrak a$ is commutative and stable under $\operatorname{ad}_R$, the elements $F$, $d_S F$ and $[R,F]$ all lie in the same commutative algebra.  Hence
\[
 G^{-1}dG=dX=\frac{d_S F}{u}-\frac{F}{u^2}\,du.
\]
Moreover, the adjoint exponential terminates after the first commutator:
\[
 G^{-1}RG=e^{-\operatorname{ad}_X}(R)
 =R+[R,X]
 =R+\frac{[R,F]}{u},
\]
because $[X,[R,X]]=0$.  Applying the gauge convention \eqref{eq:gauge-action} to $\nabla_0=d+R\,du/u$ gives
\[
 G^{-1}\nabla_0G
 =d+G^{-1}dG+\frac{G^{-1}RG}{u}\,du,
\]
which is \eqref{eq:isomonodromic-lattice-family}.

Gauge invariance gives zero curvature on the punctured product.  Since the curvature of~\eqref{eq:isomonodromic-lattice-family} is a meromorphic two-form, it vanishes identically.  Equivalently, one may verify the mixed component directly from $d_S^2F=0$, $[F,d_S F]=0$, and $[R,d_S F]=d_S[R,F]$.
For a vector field $\xi$ on~$S$, contraction with $u\xi$ gives $(\nabla_{R,F})_{u\xi}=u\xi+\xi(F)$, which preserves the standard lattice.  Contraction with $u^2\partial_u$ gives
\[
 (\nabla_{R,F})_{u^2\partial_u}
 =u^2\partial_u+uR+[R,F]-F,
\]
which also preserves it.  These are the pole conditions in the definition of an $F$-bundle used in~\cite[Definition~2.2]{HYZZ}.

Finally, for fixed $s$, the meromorphic gauge $G(s,u)$ identifies the fibre with $d+R\,du/u$ on the punctured disc.  Hence the punctured meromorphic gauge class and the monodromy conjugacy class are independent of~$s$.  The coefficient of $u^{-2}du$ in the fixed lattice is nevertheless
\[
 [R,F(s)]-F(s)=(\operatorname{ad}_R-I)F(s),
\]
which may vary holomorphically, as claimed.
\end{proof}

\begin{corollary}\label{cor:isomonodromic-affine-map}
Assume that $R=M_{\boldsymbol\mu}$ is diagonal and $\mathfrak a\subset\mathfrak n$.  Write
\[
 F(s)=\sum_{j<i}f_{ji}(s)E_{ji}.
\]
Then every fibre of \eqref{eq:isomonodromic-lattice-family} is already in the distinguished normal form, with
\[
 \alpha_{ji}(s)=(\mu_j-\mu_i-1)f_{ji}(s)
 =-(\lambda_{ji}+1)f_{ji}(s),
 \qquad \beta_{ji}(s)=0.
\]
Equivalently, its fully framed classifying morphism is
\[
 S\longrightarrow\mathfrak n\oplus\mathfrak n,
 \qquad
 s\longmapsto\bigl((\operatorname{ad}_{M_{\boldsymbol\mu}}-I)F(s),0\bigr).
\]
\end{corollary}

\begin{proof}
For $j<i$ one has
\[
 [M_{\boldsymbol\mu},E_{ji}]=(\mu_j-\mu_i)E_{ji}.
\]
The fibre contains no off-diagonal term other than the displayed coefficient of $u^{-2}du$, so it belongs directly to the affine slice with the stated coordinates.
\end{proof}

\begin{corollary}\label{cor:scalar-isomonodromic-family}
Let $R=\mu I_V$.  Then
\[
 \nabla_{\mu,F}
 =d+\frac{d_S F}{u}+\frac{\mu I_V}{u}\,du-
 \frac{F}{u^2}\,du,
\]
and its classifying map is $s\mapsto(-F(s),0)$.  Thus every holomorphic family in a commutative nilpotent subalgebra is realised as a fully framed isomonodromic variation of the prescribed lattice, for arbitrary scalar monodromy $e^{-2\pi\ii\mu}$.
\end{corollary}

\begin{proof}
For $R=\mu I_V$ one has $[R,F]=0$.  Substitute this identity into \eqref{eq:isomonodromic-lattice-family} and into the classifying map of \Cref{cor:isomonodromic-affine-map}.
\end{proof}

In particular, if $N_1,\ldots,N_m$ are commuting nilpotent upper triangular matrices and
\[
 F(t_1,\ldots,t_m)=\sum_{a=1}^m t_a N_a,
\]
then
\[
 (t_1,\ldots,t_m)\longmapsto
 \left(-\sum_{a=1}^m t_a N_a,0\right)
\]
is the classifying map.  The original rank-two example is recovered by taking $R=0$, $m=1$, and $N_1=E_{12}$:
\[
 \nabla=d+\frac{E_{12}}{u}\,dt-\frac{tE_{12}}{u^2}\,du.
\]
Its monodromy is the identity for every $t$, while its affine coordinate is $(\alpha,\beta)=(-t,0)$.

The theorem concerns the prescribed lattice and, in the displayed coordinates, the fully framed normal-form slice.  After the frame or the flag is forgotten, distinct values of $F(s)$ may become conjugate.  What is invariant without further rigidification is the corresponding orbit of $(\operatorname{ad}_R-I)F(s)$, not its individual matrix entries.

The order-two coefficient on an $\operatorname{ad}_R$-weight-$\lambda$ vector is $(\lambda-1)F_\lambda$. Thus the weight-$1$ direction is invisible to this particular deformation, not to the moduli problem itself. For a diagonal $R$ this corresponds to $\lambda_{ji}=\mu_i-\mu_j=-1$, the regime detected by monodromy in \Cref{prop:rank-two-monodromy}.

\section{Normal forms in families}\label{sec:relative-slice}

Regular-singular connections with parameters admit a rich relative theory; see, for example, the algebraic Riemann--Hilbert equivalence of Phung Ho Hai, dos Santos and Pham Thanh Tam~\cite{HaiDosSantosTam}, the relative complex theory of~\cite{HaiDosSantos}, and the logarithmic decomposition results of~\cite{Tam}.  The result below is more elementary but also more rigid: the triangular lattice reduction, residue vector and ordered frame are fixed in advance, and we represent the remaining unipotent quotient.

The prescribed lattice may vary even in a meromorphically isotrivial family with constant monodromy, as \Cref{sec:isomonodromic-lattices} shows.  Once a relative triangular lattice reduction has been produced, however, the normalisation is finite, coefficientwise and compatible with arbitrary base change.  We work over an arbitrary complex analytic space, possibly non-reduced.

Let $S$ be a complex analytic space and let
\[
 p:S\times\Delta\longrightarrow S
\]
be the projection.  Write
\[
 \cH_S=p_*\cO_{S\times\Delta},
 \qquad
 \cV_S=\ker\bigl(\operatorname{ev}_1:\cH_S\to\cO_S\bigr),
 \qquad
 \cW_S=z^{-2}\cH_S.
\]
Thus a section of $\cV_S$ is a holomorphic function $P(s,z)$ satisfying $P(s,1)=0$, whereas a section of $\cW_S$ is a meromorphic function with a pole of order at most two along $S\times\{0\}$.  For $\lambda\in\C$, set
\[
 L_{\lambda,S}:\cV_S\longrightarrow\cW_S,
 \qquad
 L_{\lambda,S}(P)=\partial_z P-\lambda z^{-1}P.
\]

\begin{lemma}\label{lem:relative-scalar-splitting}
For every complex analytic space $S$ and every $\lambda\in\C$, the morphism $L_{\lambda,S}$ is injective and there is a direct-sum decomposition of sheaves of $\cO_S$-modules
\begin{equation}\label{eq:relative-scalar-splitting}
 \cW_S
 =L_{\lambda,S}(\cV_S)
 \oplus\cO_S z^{-2}
 \oplus\cO_S\rho_\lambda.
\end{equation}
The decomposition is functorial under arbitrary morphisms of complex analytic spaces.
\end{lemma}

\begin{proof}
The assertion is local on $S$.  For a section $Q$ over an open
$T\subseteq S$, write
\[
 Q=q_{-2}z^{-2}+\sum_{k\geq-1}q_kz^k,
 \qquad q_k\in\cO_S(T).
\]
The coefficient $q_{-2}$ belongs to the first complementary summand.
If $\lambda\notin\mathbb Z_{\geq0}$, construct near $z=0$
\[
 P_0=\sum_{n\geq0}\frac{q_{n-1}}{n-\lambda}z^n,
\qquad
 P=P_0-P_0|_{z=1},\qquad
 \beta=-\lambda P_0|_{z=1}.
\]
Here $P_0|_{z=1}$ means the value of the holomorphic continuation
described below, not necessarily the sum of the Taylor series at $1$.
If $\lambda=m\in\mathbb Z_{\geq0}$, first remove the resonant
coefficient $\beta=q_{m-1}$, construct the solution $P_*$ with
coefficient of $z^m$ equal to zero, and put
$P=P_*-P_*|_{z=1}z^m$.
Then in both cases
\[
 Q=q_{-2}z^{-2}+L_{\lambda,S}(P)+\beta\rho_\lambda,
 \qquad P|_{z=1}=0.
\]

We justify the constructions also for non-reduced $T$.  Near each
point of $T$, embed $T$ into a polydisc $\Omega$ and lift the
holomorphic coefficients of $Q$ locally.  Cauchy estimates on a
sufficiently small product of polydiscs around $(s,0)$ give normal
convergence of the displayed series: the denominators $n-\lambda$
(or $n-m$ for $n\ne m$) have at most linear inverse growth.
These series solve a linear differential equation in $z$ whose
coefficients are holomorphic on $\Delta^\times$.
Analytic continuation along paths in $\Delta^\times$, with
holomorphic dependence on the parameters in $\Omega$, extends
the solutions.  The local germ is holomorphic at $z=0$, so its
continuation has trivial monodromy around the generator of
$\pi_1(\Delta^\times)$, and hence is single-valued on $\Delta$.
This argument works locally along a finite chain of coordinate
neighbourhoods in the base and the disc; uniqueness makes the
continuations agree on overlaps.  Since the differential equation
is linear over $\cO_T$, changing a lift by an element of the ideal
of $T$ changes the continued solution by a section of that ideal.
The resulting $P$ therefore descends to $T$ and depends
holomorphically on the possibly non-reduced parameter.

The directness of the decomposition and injectivity follow from
coefficient comparison at $z=0$, exactly as in
\Cref{lem:scalar-splitting}.  The framed condition eliminates
$cz^m$ from the homogeneous kernel in the resonant case.
The construction uses linear differential operators in $z$ and
multiplication by constants depending only on $\lambda$.  In local
presentations of a non-reduced base, it therefore preserves the ideal
of relations: if a coefficient belongs to the ideal defining $S$ in an
ambient polydisc, the constructed normalising coefficient belongs to
the same ideal.  Uniqueness of the framed solution and its two
complementary coefficients then implies that the construction commutes
with every pullback $S'\to S$, including non-flat pullbacks to
non-reduced spaces.  This proves the sheaf decomposition and its
functoriality.
\end{proof}

Fix a residue vector $\boldsymbol\mu=(\mu_1,\ldots,\mu_r)\in\C^r$.  For a complex analytic space $S$, let $\cA_{\boldsymbol\mu}(S)$ be the set of relative connections
\[
 d_z+\left(\frac{M_{\boldsymbol\mu}}{z}+Q(s,z)\right)dz,
 \qquad
 Q\in\Gamma(S\times\Delta,z^{-2}\cO_{S\times\Delta}\otimes\mathfrak n),
\]
where $d_z$ denotes the relative differential in the disc direction.  Let
\[
 \cU_1(S)
 =\left\{
 g:S\times\Delta\to U\text{ holomorphic}:g|_{z=1}=I_r
 \right\}
\]
act by the relative form of \eqref{eq:gauge-action}.  Finally, let $\cS_{\boldsymbol\mu}(S)$ consist of the connections
\[
 d_z+\left[
 \frac{M_{\boldsymbol\mu}}{z}
 +\sum_{j<i}
 \bigl(\alpha_{ji}(s)z^{-2}+\beta_{ji}(s)\rho_{\lambda_{ji}}(z)\bigr)E_{ji}
 \right]dz,
\]
where $\alpha_{ji},\beta_{ji}\in\Gamma(S,\cO_S)$ and $\lambda_{ji}=\mu_i-\mu_j$.

\begin{theorem}\label{thm:relative-normal-form}
For every complex analytic space $S$ and every element $A\in\cA_{\boldsymbol\mu}(S)$, there is a unique gauge $g\in\cU_1(S)$ such that $A\cdot g$ belongs to $\cS_{\boldsymbol\mu}(S)$.  The resulting coefficient functions
\[
 \alpha_{ji},\beta_{ji}\in\Gamma(S,\cO_S)
\]
are functorial under arbitrary base change.
Namely, if $N=2\dim\mathfrak n=r(r-1)$, the normal-form coefficients define a natural transformation
\[
 \pi_{\boldsymbol\mu}:\cA_{\boldsymbol\mu}\longrightarrow\mathbb A^N_{\C},
\]
and the map
\[
 \Theta_{\boldsymbol\mu,S}:
 \mathbb A^N_{\C}(S)\times\cU_1(S)
 \longrightarrow\cA_{\boldsymbol\mu}(S),
 \qquad
 (\sigma,g)\longmapsto\sigma\cdot g^{-1},
\]
is bijective, functorially in~$S$.  Hence $\pi_{\boldsymbol\mu}$ is a functorially trivial torsor for the sheaf of groups $\cU_1$, the quotient stack $[\cA_{\boldsymbol\mu}/\cU_1]$ is represented by $\mathbb A^N_{\C}$, and all stabilisers are trivial.
\end{theorem}

\begin{proof}
The argument is the absolute height induction carried out sheafwise over the base.

\emph{Existence.}  Work locally on an open subset $T\subset S$.  Suppose that all coefficients of height smaller than $h$ have been normalised.  For every $j<i$ with $i-j=h$, apply the functorial direct sum \eqref{eq:relative-scalar-splitting} to the current coefficient $q_{ji}$.  There are unique sections
\[
 p_{ji}\in\cV_S(T),
 \qquad \alpha_{ji},\beta_{ji}\in\cO_S(T)
\]
such that
\[
 q_{ji}+L_{\lambda_{ji},S}(p_{ji})
 =\alpha_{ji}z^{-2}+\beta_{ji}\rho_{\lambda_{ji}}.
\]
Set
\[
 P_h=\sum_{i-j=h}p_{ji}E_{ji},
 \qquad g_h=I_r+P_h.
\]
The matrix $P_h$ is nilpotent, so $g_h^{-1}$ is a finite polynomial in $P_h$; it is therefore holomorphic on $T\times\Delta$.  The condition $p_{ji}(s,1)=0$ gives $g_h|_{z=1}=I_r$.  The bracket estimate \eqref{eq:height-bracket} shows that the height-$h$ part of the gauge action is the direct sum of the scalar operators $L_{\lambda_{ji},S}$, while every nonlinear term has greater height.  Thus $g_h$ normalises height~$h$ and preserves the preceding stages.  The finite product $g_1\cdots g_{r-1}$ gives the required gauge on~$T$.

The scalar splitting is unique, hence the gauges constructed on two overlapping open subsets coincide there.  They therefore glue to a global element of $\cU_1(S)$, and the coefficient functions $\alpha_{ji},\beta_{ji}$ glue to global sections of~$\cO_S$.

\emph{Uniqueness.}  Let $g\in\cU_1(S)$ carry one relative normal form to another.  We argue locally on~$S$ and induct on the height filtration, a formulation which is valid without any reducedness hypothesis.  Project the gauge equation first to $\bigl(\mathfrak n^{(1)}/\mathfrak n^{(2)}\bigr)\otimes\cO_S$.  For every root of height one it gives
\[
 \widetilde q_{ji}-q_{ji}
 =L_{\lambda_{ji},S}(p_{ji}).
\]
The left-hand side lies in the complementary summand of \eqref{eq:relative-scalar-splitting}, whereas the right-hand side lies in the image summand.  Directness and injectivity force both sides to vanish.  Suppose inductively that the image of $g-I_r$ vanishes through height $h-1$ and that the two normal forms agree in those heights.  Projection modulo $\mathfrak n^{(h+1)}$ gives the same equation for every root of height $h$, because all nonlinear contributions have larger height.  The direct splitting again forces the height-$h$ image of $g-I_r$ to vanish and the corresponding normal-form coefficients to agree.  The finite filtration ends at height $r-1$, so $g=I_r$ and the two normal forms agree.

\emph{Base change and representability.}  Every operation used above---Taylor coefficient extraction, evaluation at $z=1$, multiplication by fixed complex scalars, and finite matrix multiplication in $\mathfrak n$---commutes with pullback.  The classifying coefficients and the normalising gauge therefore commute with arbitrary base change, including pullback to non-reduced spaces.  They are the $N=r(r-1)$ coordinate functions of the point $\pi_{\boldsymbol\mu}(A)\in\mathbb A^N_{\C}(S)$.  The factorisation $
 A=\sigma\cdot g^{-1}
$
is unique and functorial.  Hence $\Theta_{\boldsymbol\mu}$ is an isomorphism of functors
\[
 \mathbb A^N_{\C}\times\cU_1\xrightarrow{\sim}\cA_{\boldsymbol\mu},
\]
equivariant for right translation on the second factor.  This is a functorial trivialisation of the $\cU_1$-torsor $\pi_{\boldsymbol\mu}$.  The quotient stack is therefore represented by $\mathbb A^N_{\C}$, and all stabilisers are trivial.
\end{proof}

\begin{corollary}\label{cor:relative-deformation-theory}
Fix a complex normal form $\sigma_0\in\cS_{\boldsymbol\mu}(\C)$, and let $A$ be a local Artinian $\C$-algebra with maximal ideal $\mathfrak m_A$.  Within the fixed triangular-reduction and residue stratum, the groupoid of deformations of $\sigma_0$ over $\operatorname{Spec}_{\mathrm{an}}A$ is discrete, and its set of objects is identified with
\[
 (\mathfrak n\oplus\mathfrak n)\otimes_{\C}\mathfrak m_A.
\]
Consequently, the framed deformation problem is unobstructed, has tangent space $\mathfrak n\oplus\mathfrak n$, and has no infinitesimal automorphisms.
\end{corollary}

\begin{proof}
By \Cref{thm:relative-normal-form}, the deformation groupoid is the fibre over $\sigma_0$ of
\[
 \mathbb A^{r(r-1)}_{\C}(\operatorname{Spec}_{\mathrm{an}}A)
 \longrightarrow
 \mathbb A^{r(r-1)}_{\C}(\operatorname{Spec}_{\mathrm{an}}\C).
\]
Using the affine coordinates $(\boldsymbol\alpha,\boldsymbol\beta)$, this fibre is $(\mathfrak n\oplus\mathfrak n)\otimes\mathfrak m_A$.  The quotient groupoid is discrete, so there are no automorphisms.  Affine space is smooth, hence every deformation lifts across a small extension and there are no obstruction classes.
\end{proof}

\begin{corollary}\label{cor:relative-classifying-map}
Let $(\cE,\nabla)$ be a family of residual constituents over a complex analytic space $S$.  Assume that on $S$ it is supplied with
\begin{enumerate}[(i)]
 \item a holomorphic ordered framing at $z=1$;
 \item a holomorphic triangular reduction in the prescribed lattice; and
 \item a fixed diagonal-residue vector $\boldsymbol\mu$.
\end{enumerate}
Then the family determines a holomorphic classifying morphism
\[
 \Phi^{\mathrm{fr}}_{S,\boldsymbol\mu}:
 S\longrightarrow\mathbb A^{r(r-1)}_{\C},
\]
whose value at every complex point is the affine normal-form coordinate of \Cref{thm:main}.  The construction commutes with arbitrary base change.
\end{corollary}

\begin{proof}
In a framed triangular trivialisation, write the diagonal coefficients as $\mu_i z^{-1}+b_i(s,z)$, where $b_i$ is holomorphic. The unique diagonal gauge $h=\diag(h_i)$ satisfying $\partial_z h_i=-b_i h_i$ and $h_i(s,1)=1$ gives $A\in\cA_{\boldsymbol\mu}(S)$. Set $\Phi^{\mathrm{fr}}_{S,\boldsymbol\mu}=\pi_{\boldsymbol\mu}(A)$. On overlaps, two such normalised connections differ by a framed unipotent gauge, so their images under $\pi_{\boldsymbol\mu}$ agree. The relative normalisation in \Cref{thm:relative-normal-form} makes the map holomorphic and compatible with base change.
\end{proof}

The classifying morphism requires the prescribed holomorphic triangular reduction, which need not exist globally when lattice positions jump. If the residue vector is only locally constant, the construction is applied on each connected component. The wall charts compare changes of reduction.

\section{Elementary modifications and connection-compatible wall charts}\label{sec:wall-crossing}

The analytic normal forms of \Cref{sec:relative-slice} depend on a chosen triangular reduction. An elementary lattice modification changes that reduction and must be accompanied by a compatible connection. In rank two, the resulting wall chart is governed by the residue difference: the distinguished normal slice meets the oriented wall outside the strictly positive integral resonance range. In rank three, adjacent roots interact through a composite-root correction, which is polynomial in adapted coordinates and contributes a central cocycle to the composition of walls. Relative saturation also detects an additional infinitesimal generator over a square-zero base.

Let $S$ be a small disc with coordinate $t$, and write
$\cO_{S\times\Delta}$ simply as $\cO$.  If $\cI\subset\cE$ is a coherent
submodule of a locally free module, its saturation along the polar divisor is
\[
 \cI^{\mathrm{sat}_z}
 =\{s\in\cE\mid z^N s\in\cI\text{ for some }N\geq0\}.
\]
This operation need not commute with restriction to a parameter fibre.

\subsection{The scalar projection needed at a wall}

The special fibres below contain holomorphic linear terms which must still be
projected to the framed normal complement of \Cref{lem:scalar-splitting}.  We
record that projection.

\begin{lemma}\label{lem:linear-wall-projection}
For $\nu\in\C$ and $a,b\in\C$, define
\begin{equation}\label{eq:wall-sigma}
 \sigma_\nu(a,b)=
 \begin{cases}
 -\nu\left(\dfrac{a}{1-\nu}+\dfrac{b}{2-\nu}\right),
      &\nu\notin\mathbb Z_{\geq0},\\[3mm]
 0,&\nu=0,\\
 a,&\nu=1,\\
 b,&\nu=2,\\
 0,&\nu\in\mathbb Z_{\geq3}.
 \end{cases}
\end{equation}
Then the unique framed normal representative of the scalar coefficient
$a+bz$ for residue difference $\nu$ is
\[
 \sigma_\nu(a,b)\rho_\nu(z).
\]
Equivalently, there is a unique $P\in\cO(\Delta)$ with $P(1)=0$ such that
 $
 a+bz+L_\nu(P)=\sigma_\nu(a,b)\rho_\nu(z).
$
\end{lemma}

\begin{proof}
Assume first that $\nu\notin\mathbb Z_{\geq0}$.  Put
\[
 P_0(z)=\frac{a}{1-\nu}z+\frac{b}{2-\nu}z^2,
 \qquad
 P(z)=-P_0(z)+P_0(1).
\]
Then $P(1)=0$ and a direct calculation gives
\[
 a+bz+L_\nu(P)
 =-\nu P_0(1)z^{-1}
 =\sigma_\nu(a,b)z^{-1}.
\]
If $\nu=m\in\mathbb Z_{\geq0}$, the coefficient of $z^{m-1}$ is the
resonant coefficient omitted by $L_m$.  For $m=0$ the polynomial $a+bz$ is
a framed derivative; for $m=1$ its constant coefficient survives; for $m=2$
its linear coefficient survives; and for $m\geq3$ it has no
$z^{m-1}$-coefficient.  The free coefficient of $z^m$ in the primitive is
then chosen uniquely to impose $P(1)=0$.  This gives
\eqref{eq:wall-sigma}.
\end{proof}

In particular, for $\nu=-1$ one has
\begin{equation}\label{eq:constant-projection-minus-one}
 K+L_{-1}\left(-\frac K2(z-1)\right)=\frac K2z^{-1}.
\end{equation}
Thus a constant term is not removed without changing the normal
$z^{-1}$-coefficient.  Consequently, removing a constant term changes the resonant normal coordinate; this contribution is essential in the rank-three calculation below.

\subsection{The elementary rank-two lattice jump}

Set $\cE=\cO e_1\oplus\cO e_2$ and consider the relative connection
in the $z$-direction
\begin{equation}\label{eq:rank-two-wall-fixed}
 \nabla^{\mathrm{rel}}=d_z+\frac{t}{z^2}E_{21}\,dz.
\end{equation}
At $z=1$ choose the ordered frame
$
 v_1=e_1+te_2,\ v_2=e_2,
$
and, on the punctured disc, put
\[
 \cF^\circ=\cO(*0)\left(e_1+\frac tz e_2\right).
\]

\begin{proposition}\label{prop:rank-two-lattice-jump}
The line $\cF^\circ$ is horizontal in the $z$-direction for
\eqref{eq:rank-two-wall-fixed}, every punctured fibre has trivial monodromy, and $\cF^\circ|_{z=1}$ is the first
member of the flag determined by $(v_1,v_2)$.  Moreover,
\begin{equation}\label{eq:rank-two-relative-intersection}
 \cE\cap\cF^\circ=\cO(ze_1+te_2).
\end{equation}
This module is $z$-saturated over the reduced base $S$, but
\begin{equation}\label{eq:rank-two-base-change-failure}
 (\cE\cap\cF^\circ)\otimes_{\cO_S}\C_0
 =\cO_\Delta ze_1
 \neq
 \cE_0\cap\cF_0^\circ
 =\cO_\Delta e_1.
\end{equation}
\end{proposition}

\begin{proof}
Put $f=e_1+tz^{-1}e_2$.  Since
\[
 d_zf=-tz^{-2}e_2\,dz,
 \qquad
 \frac{t}{z^2}E_{21}f\,dz=tz^{-2}e_2\,dz,
\]
one has $\nabla^{\mathrm{rel}}f=0$.  The meromorphic frame $(f,e_2)$ is horizontal, and
$f(1)=v_1$.
Write $R=\C\{t,z\}$.  An element $hf$ belongs to $Re_1\oplus Re_2$ if and
only if $h\in R$ and $ht/z\in R$.  Since $t$ and $z$ are relatively prime,
this is equivalent to $z\mid h$, which proves
\eqref{eq:rank-two-relative-intersection}.  If
\[
 z(x_1e_1+x_2e_2)=a(ze_1+te_2),
\]
then $x_1=a$ and $zx_2=at$.  Relative primality gives $a=zb$, and hence
$x_1e_1+x_2e_2=b(ze_1+te_2)$.  Thus the module is already $z$-saturated.
Restriction to $t=0$ gives \eqref{eq:rank-two-base-change-failure}.
\end{proof}

Here and in the wall computations below, horizontality refers to the
relative connection in the $z$-direction.  The matrix in
\eqref{eq:rank-two-wall-fixed}, with no $dt$-term, does not define a flat
total connection on $S\times\Delta$.  It admits the flat extension
\[
 d-\frac{E_{21}}{z}\,dt+\frac{tE_{21}}{z^2}\,dz
 =d-d\left(\frac{t}{z}\right)E_{21},
\]
for which both $e_1+(t/z)e_2$ and $e_2$ are horizontal in all directions.
The elementary wall classification itself concerns the relative meromorphic
connections and does not assume a choice of this total extension.

On $S^\times=S\setminus\{0\}$ use the standard-framed adapted basis
\[
 w_1=ze_1+te_2,
 \qquad
 w_2=\frac{z-1}{t}e_1+e_2,
\]
whose matrix is
\begin{equation}\label{eq:rank-two-H}
 H_t=
 \begin{pmatrix}
 z&(z-1)/t\\
 t&1
 \end{pmatrix},
 \qquad \det H_t=1.
\end{equation}

\begin{theorem}\label{thm:rank-two-wall}
Let $\mu_1,\mu_2\in\C$, put $\lambda=\mu_2-\mu_1$, and consider on
$S^\times$ the standard normal form
\begin{equation}\label{eq:rank-two-general-open-normal}
 A_{\mathrm{open}}
 =\begin{pmatrix}
 \mu_1z^{-1}&q(z)\\
 0&\mu_2z^{-1}
 \end{pmatrix},
 \qquad
 q(z)=\alpha z^{-2}+\beta\rho_\lambda(z),
\end{equation}
where $\alpha$ and $\beta$ are meromorphic in $t$.  Transport it to the fixed
lattice by $
 A_{\mathrm{fix}}
 =H_tA_{\mathrm{open}}H_t^{-1}-H_t'H_t^{-1}.
$
Set
\[
 N_\lambda(t,z)
 =\lambda z-\lambda-1+t z^2q(z).
\]
Then
\begin{equation}\label{eq:rank-two-general-fixed-matrix}
 A_{\mathrm{fix}}=
 \begin{pmatrix}
 -\lambda+\mu_2z^{-1}-tzq&N_\lambda/t\\[1mm]
 -t(\lambda+tzq)/z&\lambda+\mu_1z^{-1}+tzq
 \end{pmatrix}.
\end{equation}
The following alternatives hold.
\begin{enumerate}[(i)]
 \item If $\lambda\in\mathbb Z_{\geq1}$, no meromorphic arc in the
 distinguished normal slice \eqref{eq:rank-two-general-open-normal} extends
 through this oriented elementary wall.
 \item If $\lambda\notin\mathbb Z_{\geq1}$, then
 $\rho_\lambda=z^{-1}$ and the arc extends holomorphically across $t=0$, with
 pole order at most two in the fixed lattice, if and only if there are unique
 holomorphic functions $a,b\in\cO_S$ such that
 \begin{equation}\label{eq:rank-two-all-spectrum-renormalisation}
  \alpha=\frac{\lambda+1}{t}+a,
  \qquad
  \beta=-\frac{\lambda}{t}+b.
 \end{equation}
 Its special fibre has residue vector
 \begin{equation}\label{eq:rank-two-shifted-residue}
  (\mu_1-1,\mu_2+1)
 \end{equation}
 and, before framed normalisation, upper-right coefficient $a(0)+b(0)z$.
 Its unique special normal coordinates are
 \begin{equation}\label{eq:rank-two-special-all-spectrum}
  (\widetilde\alpha,\widetilde\beta)
  =\bigl(0,\sigma_{\lambda+2}(a(0),b(0))\bigr).
 \end{equation}
\end{enumerate}
\end{theorem}

\begin{proof}
A direct multiplication gives \eqref{eq:rank-two-general-fixed-matrix}.  If
$\lambda=m\in\mathbb Z_{\geq1}$, then
$z^2\rho_\lambda=z^{m+1}$ and
\[
 \frac{N_\lambda}{t}
 =\frac{t\alpha-m-1}{t}+\frac{m}{t}z+\beta z^{m+1}.
\]
The monomials $1,z,z^{m+1}$ are distinct.  Hence holomorphicity in~$t$ of
the upper-right entry would force each of their coefficients to be
holomorphic, but the coefficient of $z$ is $m/t$.  This is impossible.

Assume now that $\lambda\notin\mathbb Z_{\geq1}$.  Then
$\rho_\lambda=z^{-1}$, including the resonant case $\lambda=0$, and
\[
 N_\lambda
 =(t\alpha-\lambda-1)+(t\beta+\lambda)z.
\]
Thus $N_\lambda/t$ is holomorphic if and only if
\eqref{eq:rank-two-all-spectrum-renormalisation} holds with $a,b$
holomorphic.  Substitution into \eqref{eq:rank-two-general-fixed-matrix}
shows that every entry then extends and that the diagonal special residues
are \eqref{eq:rank-two-shifted-residue}.  The special upper-right coefficient
is $a(0)+b(0)z$.  Its residue difference is $\lambda+2$, so
\Cref{lem:linear-wall-projection} gives
\eqref{eq:rank-two-special-all-spectrum}.
\end{proof}

\begin{corollary}\label{cor:rank-two-wall-contraction}
For $(\mu_1,\mu_2)=(1,-1)$ every extending arc is uniquely of the form
\[
 \alpha=-t^{-1}+a,
 \qquad
 \beta=2t^{-1}+b,
 \qquad a,b\in\cO_S,
\]
and every special normal representative is the origin of the residue
stratum $(0,0)$.  The distinguished choice $a=b=0$ is the family
\eqref{eq:rank-two-wall-fixed}.  The finite wall parameters are
\[
 a=\frac{1+t\alpha}{t},
 \qquad
 b=\frac{-2+t\beta}{t};
\]
the numerators alone merely express the required divisibility by $t$.
\end{corollary}

\begin{proof}
Here $\lambda+2=0$ and $\sigma_0=0$.  Substituting $a=b=0$ into
\eqref{eq:rank-two-general-fixed-matrix} gives
$t z^{-2}E_{21}$.
\end{proof}

The theorem exhibits a resonance barrier which is invisible in the
fixed-stratum dimension count: the oriented wall meets every distinguished
normal slice except those with strictly positive integral residue difference.
The reverse orientation is a different chart and must be analysed
separately.

\subsection{Polynomial rank-three wall coordinates}

We first cross the wall in the first adjacent pair.  Embed $H_t$ into the
first two coordinates and put
\[
 \widehat H_t^{12}=
 \begin{pmatrix}
 z&(z-1)/t&0\\
 t&1&0\\
 0&0&1
 \end{pmatrix}.
\]
On $S^\times$ consider
\begin{equation}\label{eq:rank-three-open-normal}
 A_{\mathrm{open}}^{12}=
 \begin{pmatrix}
 z^{-1}&\alpha z^{-2}+\beta z^{-1}&\eta z^{-2}+\theta z^{-1}\\
 0&-z^{-1}&\gamma z^{-2}+\delta z^{-1}\\
 0&0&-z^{-1}
 \end{pmatrix}.
\end{equation}

\begin{theorem}\label{thm:rank-three-wall}
Transport \eqref{eq:rank-three-open-normal} by $\widehat H_t^{12}$.  The
result extends across $t=0$ if and only if there are unique holomorphic
functions $a,b,C,D,U,V$ satisfying
\begin{equation}\label{eq:rank-three-12-parameters}
 \begin{gathered}
 \alpha=-t^{-1}+a,
 \qquad
 \beta=2t^{-1}+b,
 \qquad
 \gamma=tC,
 \qquad
 \delta=D,\\
 \eta=U-C+D/t,
 \qquad
 \theta=V-D/t.
 \end{gathered}
\end{equation}
In these parameters the transported matrix is
\begin{equation}\label{eq:rank-three-12-fixed}
 \begin{pmatrix}
 -t(a+bz)/z&a+bz&-Cz^{-2}+Uz^{-1}+V\\
 t(1-t(a+bz))/z^2&t(a+bz)/z&(D+tU)z^{-2}+tVz^{-1}\\
 0&0&-z^{-1}
 \end{pmatrix}.
\end{equation}
The special residue vector is $(0,0,-1)$.  Define
\begin{equation}\label{eq:rank-three-12-adapted}
 A=-C-D\left(a+\frac b2\right),
 \qquad
 B=U+Da+\frac V2+\frac{Db}{4}.
\end{equation}
After unique framed normalisation of the special fibre one has
\begin{align}
 (\widetilde\alpha_{12},\widetilde\beta_{12})&=(0,0),\label{eq:special-root12}\\
 (\widetilde\alpha_{23},\widetilde\beta_{23})&=(D(0),0),\notag\\
 (\widetilde\alpha_{13},\widetilde\beta_{13})&=(A(0),B(0)).\label{eq:special-root13}
\end{align}
Moreover,
\begin{equation}\label{eq:rank-three-polynomial-coordinate-change}
 (a,b,C,D,U,V)\longmapsto(a,b,D,V,A,B)
\end{equation}
is a polynomial automorphism of the six-dimensional extension parameter
space.  Thus the correction in the natural root coordinates is triangular
and admits polynomially adapted wall coordinates.
\end{theorem}

\begin{proof}
Direct multiplication gives
\[
 \widehat H_t^{12}A_{\mathrm{open}}^{12}(\widehat H_t^{12})^{-1}
 -(\widehat H_t^{12})'(\widehat H_t^{12})^{-1}
 =
 \begin{pmatrix}
 -\Xi/z&\Xi/t&-Cz^{-2}+Uz^{-1}+V\\
 t(1-\Xi)/z^2&\Xi/z&(D+tU)z^{-2}+tVz^{-1}\\
 0&0&-z^{-1}
 \end{pmatrix},
\]
where
\[
 \Xi=1-2z+t\alpha+t\beta z,
 \quad
 C=\gamma/t,
 \quad
 U=(\gamma-\delta)/t+\eta,
 \quad
 V=\delta/t+\theta,
 \quad
 D=\delta.
\]
Coefficientwise holomorphicity is therefore equivalent to
\eqref{eq:rank-three-12-parameters}, and then $\Xi=t(a+bz)$, giving
\eqref{eq:rank-three-12-fixed}.
At $t=0$ the height-one coefficients are
\[
 q_{12}=a+bz,
 \qquad
 q_{23}=Dz^{-2}.
\]
The unique framed gauge removing $q_{12}$ is
\[
 g=I_3+p(z)E_{12},
 \qquad
 p(z)=-a(z-1)-\frac b2(z^2-1).
\]
It changes the composite coefficient to
\[
 \left[-C-D\left(a+\frac b2\right)\right]z^{-2}
 +(U+Da)z^{-1}+V+\frac{Db}{2}.
\]
The special residue difference for the root $13$ is $-1$.  By
\eqref{eq:constant-projection-minus-one}, the constant term contributes half
of itself to the normal $z^{-1}$-coefficient.  This yields
\eqref{eq:rank-three-12-adapted} and
\eqref{eq:special-root12}--\eqref{eq:special-root13}.

Finally, \eqref{eq:rank-three-polynomial-coordinate-change} is triangular,
with inverse
\[
 C=-A-D\left(a+\frac b2\right),
 \qquad
 U=B-Da-\frac V2-\frac{Db}{4}.
\]
Hence it is a polynomial automorphism.
\end{proof}

The opposite simple-root chart is described in the same way.  Embed $H_t$ into the
last two coordinates and put
\[
 \widehat H_t^{23}=
 \begin{pmatrix}
 1&0&0\\
 0&z&(z-1)/t\\
 0&t&1
 \end{pmatrix}.
\]
On $S^\times$ take
\begin{equation}\label{eq:rank-three-open-normal-23}
 A_{\mathrm{open}}^{23}=
 \begin{pmatrix}
 z^{-1}&\gamma z^{-2}+\delta z^{-1}&\eta z^{-2}+\theta z^{-1}\\
 0&z^{-1}&\alpha z^{-2}+\beta z^{-1}\\
 0&0&-z^{-1}
 \end{pmatrix}.
\end{equation}

\begin{proposition}\label{prop:rank-three-wall-23}
The transport of \eqref{eq:rank-three-open-normal-23} by
$\widehat H_t^{23}$ extends if and only if there are unique holomorphic
functions $a,b,C,D,U,V$ such that
\begin{equation}\label{eq:rank-three-23-parameters}
 \begin{gathered}
 \alpha=-t^{-1}+a,
 \qquad
 \beta=2t^{-1}+b,
 \qquad
 \gamma=tC,
 \qquad
 \delta=D,\\
 \eta=U+C-D/t,
 \qquad
 \theta=V+D/t.
 \end{gathered}
\end{equation}
The special residue vector is $(1,0,0)$.  If
\begin{equation}\label{eq:rank-three-23-adapted}
 A^+=C+D\left(a+\frac b2\right),
 \qquad
 B^+=U-Da+\frac V2-\frac{Db}{4},
\end{equation}
then its special normal coordinates are
\begin{equation*}
(\widetilde\alpha_{23},\widetilde\beta_{23})=(0,0),\qquad
(\widetilde\alpha_{12},\widetilde\beta_{12})=(D(0),0),\qquad
(\widetilde\alpha_{13},\widetilde\beta_{13})=(A^+(0),B^+(0)).
\end{equation*}
The map
$
 (a,b,C,D,U,V)\longmapsto(a,b,D,V,A^+,B^+)
$
is again a polynomial automorphism.
\end{proposition}

\begin{proof}
A direct multiplication yields the extended root coefficients
\begin{equation*}
q_{13}^{\mathrm{fix}}=Cz^{-2}+Uz^{-1}+V,\ q_{12}^{\mathrm{fix}}=(D-tU)z^{-2}-tVz^{-1}.
\end{equation*}
together with the rank-two wall matrix in the $(2,3)$ block.  This is
equivalent to \eqref{eq:rank-three-23-parameters}.  At $t=0$ one has
\[
 q_{23}=a+bz,
 \qquad
 q_{12}=Dz^{-2},
 \qquad
 q_{13}=Cz^{-2}+Uz^{-1}+V.
\]
Removing $q_{23}$ by the framed gauge
$I_3+p(z)E_{23}$ with
$p=-a(z-1)-\frac b2(z^2-1)$ changes the composite coefficient to
\[
 \left[C+D\left(a+\frac b2\right)\right]z^{-2}
 +(U-Da)z^{-1}+V-\frac{Db}{2}.
\]
Applying \eqref{eq:constant-projection-minus-one} gives
\eqref{eq:rank-three-23-adapted}.  Triangularity gives the final assertion.
\end{proof}

The two propositions show what is intrinsic in the first
rank-three calculation.  The natural rootwise coordinates are not
independent across the wall, but their mixing is polynomially triangular and
can be absorbed into adapted wall coordinates.  Thus this calculation does
not disprove every abstract product chart.  What survives coordinate changes
is the ordered compatibility at the intersection of two adjacent walls, to which we now
turn.

\subsection{Elementary wall torsors in arbitrary rank}\label{sec:all-rank-connection-wall}

We now pass from the rank-two and rank-three calculations to a single elementary wall in arbitrary rank.  The first statement is a block calculation.  The second packages all extension equations into a free lattice over the wall parameter and thereby yields a relative affine torsor parametrising connection lifts.
For $1\leq i<r$, let $H_i(t)$ be the identity outside the coordinates $i,i+1$ and have the block
\[
 H_t=
 \begin{pmatrix}
 z&(z-1)/t\\
 t&1
 \end{pmatrix}
\]
on those coordinates.  Write a meromorphic arc in the distinguished normal slice as
\[
 A_{\mathrm{open}}(t,z)
 =\frac{M_{\boldsymbol\mu}}{z}
 +\sum_{1\leq j<k\leq r}q_{jk}(t,z)E_{jk},
 \qquad
 q_{jk}=\alpha_{jk}z^{-2}+\beta_{jk}\rho_{\mu_k-\mu_j}(z),
\]
where $\alpha_{jk}$ and $\beta_{jk}$ are meromorphic germs in~$t$.

\begin{proposition}\label{prop:levi-local-raw-transport}
Fix $i$, put $\lambda_i=\mu_{i+1}-\mu_i$ and $q_i=q_{i,i+1}$, and set
\[
 A_{\mathrm{fix}}
 =H_i(t)A_{\mathrm{open}}H_i(t)^{-1}
 -(\partial_zH_i(t))H_i(t)^{-1}.
\]
Then the only entries changed by the raw transport are the direct $2\times2$ block and the two root strings meeting it.  If
\[
 N_i(t,z)=\lambda_i z-\lambda_i-1+t z^2q_i(t,z),
\]
then the direct block is
\begin{align}
 (A_{\mathrm{fix}})_{ii}
 &=-\lambda_i+\mu_{i+1}z^{-1}-tzq_i,\notag\\
 (A_{\mathrm{fix}})_{i,i+1}
 &=N_i/t,\label{eq:all-rank-wall-direct-12}\\
 (A_{\mathrm{fix}})_{i+1,i}
 &=-t(\lambda_i+tzq_i)/z,\label{eq:all-rank-wall-direct-21}\\
 (A_{\mathrm{fix}})_{i+1,i+1}
 &=\lambda_i+\mu_i z^{-1}+tzq_i.\notag
\end{align}
For $p<i$ one has
\begin{align}
 (A_{\mathrm{fix}})_{p,i}
 &=q_{p,i}-tq_{p,i+1},\label{eq:all-rank-wall-left-1}\\
 (A_{\mathrm{fix}})_{p,i+1}
 &=zq_{p,i+1}-\frac{z-1}{t}q_{p,i},\label{eq:all-rank-wall-left-2}
\end{align}
and, for $k>i+1$,
\begin{align}
 (A_{\mathrm{fix}})_{i,k}
 &=zq_{i,k}+\frac{z-1}{t}q_{i+1,k},\label{eq:all-rank-wall-right-1}\\
 (A_{\mathrm{fix}})_{i+1,k}
 &=tq_{i,k}+q_{i+1,k}.\label{eq:all-rank-wall-right-2}
\end{align}
Every other off-diagonal entry is unchanged.  Consequently, extension through $t=0$ is equivalent to the rank-two condition on $q_i$ together with holomorphicity in~$t$ of the finitely many expressions in \eqref{eq:all-rank-wall-left-1}--\eqref{eq:all-rank-wall-right-2}; there are no further raw extension equations.
\end{proposition}

\begin{proof}
On the active block,
\[
 H_t^{-1}=
 \begin{pmatrix}
 1&-(z-1)/t\\
 -t&z
 \end{pmatrix}.
\]
The direct block is the rank-two calculation in \Cref{thm:rank-two-wall}.  Multiplying the row $(q_{p,i},q_{p,i+1})$ by $H_t^{-1}$ gives \eqref{eq:all-rank-wall-left-1}--\eqref{eq:all-rank-wall-left-2}, while multiplying $H_t$ by the column $(q_{i,k},q_{i+1,k})^{\mathsf T}$ gives \eqref{eq:all-rank-wall-right-1}--\eqref{eq:all-rank-wall-right-2}.  Entries having neither row nor column in the active pair are unchanged.  Since the displayed formula is the full matrix transport, its coefficientwise holomorphicity is necessary and sufficient.
\end{proof}

Put $R=\C\{t\}$ and $K_t=R[t^{-1}]$.  For every positive root $j<k$, let
\[
 V_{jk}^{\boldsymbol\mu}
 =\C z^{-2}\oplus\C\rho_{\mu_k-\mu_j}(z),
 \qquad
 V_{\boldsymbol\mu}
 =\bigoplus_{j<k}V_{jk}^{\boldsymbol\mu}E_{jk}.
\]
Thus $\dim_{\C}V_{\boldsymbol\mu}=N=r(r-1)$.  For $\nu\in\C$, write
$
 \pi_\nu:z^{-2}\cO(\Delta)\longrightarrow\C^2
$
for the linear normal-coordinate projection defined by the direct sum in \Cref{lem:scalar-splitting}: if
$q=L_\nu(P)+\alpha z^{-2}+\beta\rho_\nu$, then
$\pi_\nu(q)=(\alpha,\beta)$.

Assume from now on that $\lambda_i\notin\mathbb Z_{\geq1}$, so $\rho_{\lambda_i}=z^{-1}$, and define the distinguished minimal crossing
\begin{equation}\label{eq:all-rank-minimal-crossing}
 q^0_{i,\boldsymbol\mu}(t,z)
 =\left(\frac{\lambda_i+1}{t}z^{-2}
       -\frac{\lambda_i}{t}z^{-1}\right)E_{i,i+1}.
\end{equation}

For $p<i$, define the left lateral lattice
\begin{align*}\mathscr L^{\mathrm L}_{p,i}=\bigl\{(x,y)\in (V_{p,i}^{\boldsymbol\mu}\oplus V_{p,i+1}^{\boldsymbol\mu})\otimes_{\C}K_t:\ x-ty\in z^{-2}R\{z\},\quad zy-(z-1)t^{-1}x\in z^{-2}R\{z\}\bigr\}.\end{align*}
and, for $k>i+1$, define the right lateral lattice
 \begin{align*}\mathscr L^{\mathrm R}_{i,k}=\bigl\{(x,y)\in (V_{i,k}^{\boldsymbol\mu}\oplus V_{i+1,k}^{\boldsymbol\mu})\otimes_{\C}K_t:\ zx+(z-1)t^{-1}y\in z^{-2}R\{z\},\quad tx+y\in z^{-2}R\{z\}\bigr\}.\end{align*}
Finally, let $V_{\boldsymbol\mu}^{\mathrm{out}}$ be the sum of the root spaces whose two indices avoid $\{i,i+1\}$, and put
\begin{equation}\label{eq:all-rank-wall-lattice}\mathscr L_{i,\boldsymbol\mu}=(V_{i,i+1}^{\boldsymbol\mu}\otimes R)\oplus\bigoplus_{p<i}\mathscr L^{\mathrm L}_{p,i}\oplus\bigoplus_{k>i+1}\mathscr L^{\mathrm R}_{i,k}\oplus(V_{\boldsymbol\mu}^{\mathrm{out}}\otimes R).\end{equation}

For a flat local analytic $R$-algebra $A$ with $A/tA\neq0$, let
$\mathbf{Wall}_{i,\boldsymbol\mu}(A)$ be the set of coefficients
$Q\in V_{\boldsymbol\mu}\otimes_{\C}A[t^{-1}]$ for which the transport of
$M_{\boldsymbol\mu}/z+Q$ by $H_i(t)$ has entries in $z^{-2}A\{z\}$.
Here $A\{z\}$ denotes the corresponding ring of convergent relative
germs.  The definition is deliberately restricted to flat $R$-algebras:
it uses the punctured parameter family, whose localisation loses data
when $t$ is a zero divisor.

\begin{theorem}\label{thm:all-rank-affine-wall}
\begin{enumerate}[(i)]
 \item If $\lambda_i\in\mathbb Z_{\geq1}$, then
 $\mathbf{Wall}_{i,\boldsymbol\mu}(A)=\varnothing$ for every flat local analytic $R$-algebra~$A$ with non-empty special fibre.
 \item If $\lambda_i\notin\mathbb Z_{\geq1}$, then, functorially in every such~$A$,
 \begin{equation}\label{eq:all-rank-wall-affine-translate}
  \mathbf{Wall}_{i,\boldsymbol\mu}(A)
  =q^0_{i,\boldsymbol\mu}
   +\mathscr L_{i,\boldsymbol\mu}\otimes_R A.
 \end{equation}
 Every lateral module in \eqref{eq:all-rank-wall-lattice} is a free $R$-module of rank~$4$, and $\mathscr L_{i,\boldsymbol\mu}$ is free of rank
 \[
  2+4(r-2)+(r-2)(r-3)=r(r-1)=N.
 \]
 Thus, on the category of flat local analytic $R$-algebras, the elementary
connection-lift functor is the restriction of the functor represented by the
affine $S$-space
$
 q^0_{i,\boldsymbol\mu}+\mathbb V_S(\mathscr L_{i,\boldsymbol\mu}),
$
a torsor under the analytic vector bundle
$\mathbb V_S(\mathscr L_{i,\boldsymbol\mu})$.  Subtraction of the distinguished
minimal crossing identifies it with that vector bundle; after choosing an
$R$-basis, it is non-canonically isomorphic to
$S\times\mathbb A^N_{\C}$.
 \item Put
 $
  \boldsymbol\mu'
  =\boldsymbol\mu-\mathbf e_i+\mathbf e_{i+1}.
 $
 Evaluation of the transported universal connection on the scheme-theoretic
special fibre of the affine torsor, followed by the unique framed
normalisation, defines a polynomial morphism
 \begin{equation}\label{eq:all-rank-specialisation-map}
  \operatorname{sp}_{i,\boldsymbol\mu}:
  \mathbb V_S(\mathscr L_{i,\boldsymbol\mu})_0
  \longrightarrow
  \cS_{\boldsymbol\mu'}\simeq\mathbb A^N_{\C}.
 \end{equation}
 If $q^{\mathrm{raw}}_{jk}$ denotes the transported special coefficient and $h=k-j$, then
 \begin{equation}\label{eq:all-rank-height-triangular-map}
  (\widetilde\alpha_{jk},\widetilde\beta_{jk})
  =\pi_{\mu'_k-\mu'_j}(q^{\mathrm{raw}}_{jk})
   +P_{jk}\bigl(q^{\mathrm{raw}}_{ab}:b-a<h\bigr),
 \end{equation}
 where $P_{jk}$ is a $\C^2$-valued polynomial of total degree at most~$h$.  For $h=1$ one has $P_{jk}=0$; for $h\geq2$, the polynomial $P_{jk}$ has no terms of total degree less than~$2$.  In particular, the specialisation map is triangular by root height, and its degree-one part is the scalar projection of the transported raw coefficient.  On the active root, writing
 \[
  \alpha_{i,i+1}=\frac{\lambda_i+1}{t}+a_i,
  \qquad
  \beta_{i,i+1}=-\frac{\lambda_i}{t}+b_i,
 \]
 it is
$
  (a_i,b_i)\longmapsto
  \bigl(0,\sigma_{\lambda_i+2}(a_i(0),b_i(0))\bigr).
 $
 \item The wall functor on flat families and the polynomial specialisation
map commute with flat local analytic base change, including flat
non-reduced bases.  The special fibre of the representing affine space
records limits of flat families; it is not the value of the punctured-arc
functor on the non-flat base $R/(t)$.  No analogous identification with
meromorphic arcs is asserted for non-flat base changes.
\end{enumerate}
\end{theorem}

\begin{proof}
Part~\textup{(i)} is the rank-two resonance barrier.  The direct entry \eqref{eq:all-rank-wall-direct-12} is independent of every lateral coordinate, so no off-block term can cancel its $z$-coefficient obstruction.  The same coefficient argument remains valid after flat base change because $t$ is a non-zero divisor.  In particular, it remains valid over non-reduced but flat $R$-algebras.

Assume $\lambda_i\notin\mathbb Z_{\geq1}$.  By \Cref{thm:rank-two-wall}, the direct root of every extending arc is uniquely the minimal crossing \eqref{eq:all-rank-minimal-crossing} plus an element of $V_{i,i+1}^{\boldsymbol\mu}\otimes R$.  The formulas of \Cref{prop:levi-local-raw-transport} show that the remaining extension conditions split independently into the left lateral pairs, the right lateral pairs and the unaffected roots.  This gives \eqref{eq:all-rank-wall-affine-translate} over~$R$ and the direct-sum decomposition \eqref{eq:all-rank-wall-lattice}.  All root coefficients belong to fixed two-dimensional spaces of Laurent polynomials, so each extension condition is a finite system of $R$-linear coefficient equations.  Flat pullback is exact and therefore preserves their kernels; this yields the same formula over every flat~$A$, including those with nilpotents.
For a left pair, let $T^{\mathrm L}(x,y)$ denote the two expressions in its definition.  If $x$ and $y$ are divisible by~$t$, then $T^{\mathrm L}(x,y)$ is holomorphic in~$t$; hence
\[
 t(V_{p,i}^{\boldsymbol\mu}\oplus V_{p,i+1}^{\boldsymbol\mu})\otimes R
 \subseteq\mathscr L^{\mathrm L}_{p,i}.
\]
Conversely, writing $T^{\mathrm L}(x,y)=(u,v)$ and multiplying by $H_t$ gives
$
 x=zu+tv, $ and $  y=(z-1)t^{-1}u+v.
$
Thus the $t$-adic valuations of the coefficients of $x$ and $y$ are at least $-1$.  Therefore
\[
 \mathscr L^{\mathrm L}_{p,i}
 \subseteq
 t^{-1}(V_{p,i}^{\boldsymbol\mu}\oplus V_{p,i+1}^{\boldsymbol\mu})\otimes R.
\]
The same argument, using $H_t^{-1}$, gives the identical two-sided bound for every right lateral lattice.  The quotient
\[
 t^{-1}(V_{p,i}^{\boldsymbol\mu}\oplus V_{p,i+1}^{\boldsymbol\mu})\otimes R
 \big/
 t(V_{p,i}^{\boldsymbol\mu}\oplus V_{p,i+1}^{\boldsymbol\mu})\otimes R
\]
has finite length over the discrete valuation ring~$R$.  The two-sided bound therefore
makes $\mathscr L^{\mathrm L}_{p,i}$ a finitely generated torsion-free
$R$-module.  It has rank~$4$, since it becomes the full four-dimensional
space after inverting~$t$, and is consequently free over the discrete
valuation ring~$R$.  The same argument applies to the right lateral modules.
The displayed rank count proves the assertion for
$\mathscr L_{i,\boldsymbol\mu}$.

For part~\textup{(iii)}, the special transported matrix is upper triangular: the only potentially lower entry is \eqref{eq:all-rank-wall-direct-21}, which vanishes at $t=0$.  Substitution of the distinguished direct crossing shows that its diagonal is $M_{\boldsymbol\mu'}/z$.  Its off-diagonal coefficients lie in $z^{-2}\cO$, so \Cref{prop:admissibility} makes it an admissible regular-singular family.  Apply the relative normal-form theorem to the universal special connection over the affine special fibre of the wall torsor.  All raw special coefficients are linear functions of the wall coordinates.

At height one the normalising gauge and the scalar projection are linear.  Suppose inductively that every normalising gauge coordinate of height less than~$h$ is polynomial of total degree at most its height.  In the height-$h$ gauge equation, each correction involving lower roots contains at least two lower-height off-diagonal factors, and their root heights add to~$h$.  Such a correction therefore has total degree between~$2$ and~$h$.  The scalar splitting is linear, so both the height-$h$ gauge coordinate and the resulting normal coordinate have total degree at most~$h$; moreover, the terms not coming from the raw height-$h$ coefficient have degree at least~$2$.  This proves \eqref{eq:all-rank-height-triangular-map} and the assertion about its degree-one part.  The active-root formula is \Cref{thm:rank-two-wall}.

Finally, for each lateral pair only the normal monomials and their products with $z$ and $z-1$ occur.  They span a finite-dimensional space of Laurent polynomials in~$z$, so the extension conditions are kernels of coefficient maps between finite free $R$-modules.  Flat base change preserves these kernels and the direct-sum decomposition.  The universal transport is algebraic in the coefficients, and the relative normalisation commutes with pullback by \Cref{thm:relative-normal-form}.  This proves part~\textup{(iv)}.
\end{proof}

\begin{corollary}\label{cor:resonance-filtered-root-lift}
Fix the elementary wall frame $H_i(t)$ and the embedded rank-two root-subgroup arc
\[
 t\longmapsto G_i(t)I,
 \qquad
 G_i(t)=I_r+t z^{-1}E_{i+1,i}.
\]
The corresponding distinguished fully framed connection-wall problem with residue vector $\boldsymbol\mu$ is non-empty if and only if
\[
 \mu_{i+1}-\mu_i\notin\mathbb Z_{\geq1}.
\]
When this condition holds, the extension functor is a relative affine torsor of relative dimension $r(r-1)$.  The limiting connection is governed by the polynomial specialisation map \eqref{eq:all-rank-specialisation-map}, which is triangular by root height.  This is a statement for the chosen elementary local-model datum; it does not assert that the original unsliced moduli functor carries a global position morphism to an affine flag variety.
\end{corollary}

\begin{proof}
The active Levi block is the rank-two curve of \Cref{prop:rank-two-affine-position}, while the chosen wall frame fixes the lateral comparison data.  The existence criterion and the description of all extending connection arcs are given by \Cref{thm:all-rank-affine-wall}.
\end{proof}

\subsection{The composite-root correction at a two-wall intersection}

\begin{theorem}\label{thm:rank-three-central-cocycle}
Put
\[
 G_{12}(t)=I_3+\frac tzE_{21},
 \qquad
 G_{23}(s)=I_3+\frac szE_{32},
 \qquad
 K_{13}(u)=I_3+\frac{u}{z^2}E_{31}.
\]
Then
\begin{equation}\label{eq:rank-three-central-cocycle}
 G_{23}(s)G_{12}(t)
 =K_{13}(st)G_{12}(t)G_{23}(s).
\end{equation}
The factor $K_{13}(st)$ is central in the lower unitriangular subgroup
generated by $E_{21},E_{32},E_{31}$.  In coordinates $(x,y,w)$ on this
subgroup, multiplication is
$
 (x,y,w)(x',y',w')=(x+x',y+y',w+w'+yx'),
$
so $c((x,y),(x',y'))=yx'$ is a central group $2$-cocycle.  Modifications with
disjoint supports commute.  Consequently, the first codimension-two
compatibility datum is a central composite-root cocycle, rather than an
equality of the two ordered simple-wall charts.
\end{theorem}

\begin{proof}
One has $E_{32}E_{21}=E_{31}$ and $E_{21}E_{32}=0$.  Moreover, $E_{31}$
commutes with $E_{21}$ and $E_{32}$.  Expanding both sides gives
\[
 G_{23}(s)G_{12}(t)
 =I_3+\frac tzE_{21}+\frac szE_{32}+\frac{st}{z^2}E_{31},
\]
whereas
\[
 G_{12}(t)G_{23}(s)
 =I_3+\frac tzE_{21}+\frac szE_{32}.
\]
This proves \eqref{eq:rank-three-central-cocycle}.  The displayed
multiplication law follows from the same calculation, and the cocycle identity
\[
 c(u,v)+c(u+v,w)=c(v,w)+c(u,v+w)
\]
follows from bilinearity.  Matrix units with disjoint supports multiply
to zero in both orders.
\end{proof}

\section{Root subgroups and affine flag positions}\label{sec:wall-loop-root}

The rank-three identity is a special case of the root-subgroup relations in arbitrary rank. Put
\[
 D_r=\diag(1,z,\ldots,z^{r-1})
\]
and let $U_r^-$ denote the lower unitriangular algebraic group.

\begin{theorem}\label{thm:wall-loop-root-group}
Conjugation by $D_r$ gives an algebraic-group isomorphism
\[
 U_r^-\xrightarrow{\sim}\mathcal U_r^{\mathrm{wall}}
 \coloneqq D_r^{-1}U_r^-D_r
 \subset GL_r(\C((z))),
\]
where
\begin{equation}\label{eq:wall-loop-group}
 \mathcal U_r^{\mathrm{wall}}
 =\left\{
 I_r+\sum_{j>i}u_{ji}z^{-(j-i)}E_{ji}
 \;\middle|\;
 u_{ji}\in\C
 \right\}.
\end{equation}
The construction is functorial after replacing $\C$ by any commutative complex algebra.  For $j>i$, set
\[
 X_{ji}(a)=I_r+a z^{-(j-i)}E_{ji}.
\]
Then, for $k>j>i$,
\begin{equation}\label{eq:wall-chevalley-relation}
 X_{kj}(b)X_{ji}(a)
 =X_{ki}(ab)X_{ji}(a)X_{kj}(b),
\end{equation}
and two such root subgroups commute whenever both ordered products of the corresponding matrix units vanish.  The simple elements
$
 G_i(a)=X_{i+1,i}(a)
 \ (1\leq i<r)
$
generate $\mathcal U_r^{\mathrm{wall}}$.  Hence every higher-root correction in the lattice part of an ordered wall product is the usual type-$A$ Chevalley correction, weighted by the pole order $z^{-(j-i)}$.
\end{theorem}

\begin{proof}
For $j>i$ one has
\[
 D_r^{-1}E_{ji}D_r=z^{-(j-i)}E_{ji},
\]
which proves \eqref{eq:wall-loop-group} and the asserted functorial isomorphism.  The relation \eqref{eq:wall-chevalley-relation} follows from
\[
 E_{kj}E_{ji}=E_{ki},
 \qquad
 E_{ji}E_{kj}=0.
\]
If both ordered products of the relevant matrix units vanish, the root subgroups commute.  Finally, the ordinary lower unitriangular group is generated by its simple-root subgroups; conjugation by $D_r$ transports this statement to $\mathcal U_r^{\mathrm{wall}}$.
\end{proof}

\begin{corollary}\label{cor:pairwise-wall-compatibility}
For the elementary matrices $G_i(a)=I_r+a z^{-1}E_{i+1,i}$ one has
\[
 G_i(a)G_j(b)=G_j(b)G_i(a)
 \qquad\text{if }|i-j|>1,
\]
whereas
\begin{equation}\label{eq:all-rank-adjacent-wall}
 G_{i+1}(b)G_i(a)
 =X_{i+2,i}(ab)G_i(a)G_{i+1}(b).
\end{equation}
Consequently, the pairwise compatibility of the lattice-modification matrices in every rank is determined by the rank-two and rank-three Levi subgroups.  This statement concerns the lattice part only: the scalar normal projections and resonance barriers of the connection still have to be analysed in the relevant residue blocks.
\end{corollary}

\begin{proposition}\label{prop:rank-two-affine-position}
Put $O=\C[[z]]$, $G(t)=I_2+t z^{-1}E_{21}$ and $\Lambda_t=G(t)O^2$
and let $I\subset GL_2(O)$ be the upper Iwahori subgroup.  Then:
\begin{enumerate}[(i)]
 \item the map $t\mapsto G(t)I$ is the standard opposite root-subgroup chart of the minimal affine Schubert curve in $\operatorname{Fl}_2$; it sends $t=0$ to the base point and $t\ne0$ to the Bruhat cell $I\dot s_0I/I$, where
 \[
  \dot s_0=\begin{pmatrix}0&-z\\ z^{-1}&0\end{pmatrix};
 \]
 \item for $t\ne0$, the relative elementary divisors of $\Lambda_t$ and $O^2$ are $z^{-1}$ and $z$, while $\Lambda_0=O^2$;
 \item the matrix $H_t$ of \eqref{eq:rank-two-H} lies in $GL_2(O)$ for $t\ne0$ and satisfies $H_tO^2=O^2$.  It is therefore an adapted holomorphic frame of the prescribed lattice, not the matrix that records the position in the affine Grassmannian.
\end{enumerate}
\end{proposition}

\begin{proof}
For $t\ne0$, put $b=t^{-1}z$ and $x_+(b)=I_2+bE_{12}$.  The rank-one factorisation
\[
 G(t)
 =x_+(b)
 \begin{pmatrix}0&-b\\ t z^{-1}&0\end{pmatrix}
 x_+(b)
 =x_+(b)\diag(t^{-1},t)\dot s_0x_+(b)
\]
shows that $G(t)\in I\dot s_0I$, because $b\in zO$ and $\diag(t^{-1},t)\in I$.  Together with $G(0)=I_2$, this is the usual opposite root-subgroup chart of the rank-one affine Schubert curve~\cite{PappasRapoport}.
For the Grassmannian position, multiply by $z$:
\[
 zG(t)=\begin{pmatrix}z&0\\ t&z\end{pmatrix}.
\]
When $t\ne0$, elementary row and column operations over $O$ put this matrix into Smith normal form $\diag(1,z^2)$.  Hence $G(t)$ has elementary divisors $z^{-1}$ and $z$.  At $t=0$ it is the identity.  Finally, \eqref{eq:rank-two-H} has entries in $O$ and determinant~$1$ for $t\ne0$, proving the last assertion.
\end{proof}

The same rank-two matrix embedded in the $(i,i+1)$-block gives a real-root curve of the corresponding rank-two Levi in every rank; the remaining elementary divisors are $1$.

\subsection{Loop-group interpretation of the wall atlas}\label{sec:wall-affine-flag}

Passing to the completion at the pole, put $K=\C((z))$ and $O=\C[[z]]$.  The standard affine Grassmannian and affine flag variety are
\[
 \operatorname{Gr}_r=GL_r(K)/GL_r(O),
 \qquad
 \operatorname{Fl}_r=GL_r(K)/I,
\]
with $I$ the upper Iwahori subgroup~\cite{PappasRapoport}.  A comparison lattice gives a point of $\operatorname{Gr}_r$, and an Iwahori-level lattice chain gives a point of $\operatorname{Fl}_r$.  Such a comparison lattice is an additional local-model datum: it is not supplied by the original unsliced prescribed-lattice functor.

\Cref{prop:rank-two-affine-position} identifies the relevant matrix.  The meromorphic horizontal frame $G(t)=I_2+t z^{-1}E_{21}$ carries the non-trivial affine position and traces the opposite chart of the minimal affine Schubert curve in the rank-two Levi.  The matrix $H_t$, by contrast, lies in $GL_2(O)$ and merely chooses an adapted holomorphic basis of the fixed prescribed lattice.  Thus the wall theorem is a lifting criterion from a rank-two Levi root-subgroup arc to the distinguished normal slice of connections; the forbidden range $\lambda\in\mathbb Z_{\geq1}$ is the extra resonance obstruction imposed by the connection.

In arbitrary rank, \Cref{thm:wall-loop-root-group} gives the lattice-theoretic structure.  The wall matrices form the conjugated finite unipotent group $D_r^{-1}U_r^-D_r$, and their higher-root corrections are the ordinary type-$A$ Chevalley corrections with the prescribed pole weights.  Non-adjacent simple factors commute, while adjacent factors produce the composite-root term of \eqref{eq:all-rank-adjacent-wall}.  The rank-three correction is central only inside its minimal $3\times3$ Levi block; in higher rank the same root subgroup may interact with further roots.  Accordingly, the $A_1$, $A_1\times A_1$ and $A_2$ terminology refers to Levi root-subgroup blocks.  The corresponding affine flag positions are those determined by the indicated Levi blocks.

Thus the affine-flag description records relative lattice positions, whereas the connection-compatible charts impose the additional resonance conditions established above.

\section{Trace centring and geometric atomic classes}\label{sec:residual-atoms}

The preceding normal forms classify complex regular-singular connections with a prescribed lattice and frame. Geometric atoms arise instead from $F$-bundles with a horizontal connection and generally irregular vertical behaviour. Scalar residualisation will be applied directly to these $F$-bundles, not by transporting the complex normal forms.

\subsection{Analytic categories}\label{sec:analytic-scope}

We use three categories, following~\cite[\S3.1]{KKPYY}: holomorphic germs over $\C$, formal $F$-bundles over a characteristic-zero field, and non-Archimedean analytic $F$-bundles over a complete non-Archimedean field.  The pole conditions in each category are that $\nabla_{u^2\partial_u}$ and $\nabla_{u\xi}$ preserve the lattice, for every vector field $\xi$ on the base (or a logarithmic vector field in the logarithmic case).  In the non-Archimedean setting the base and the $u$-disc are analytic germs in the sense of~\cite[Definition~3.1]{KKPYY}; the required convergence of the quantum connection is supplied by its construction there.  The complex normal forms of \Cref{thm:main} do not define formal or non-Archimedean $F$-bundle normal forms. The regular logarithmic Riemann--Hilbert comparison over $\C$~\cite{Achinger} does not identify these categories.

The proof of \Cref{lem:scalar-splitting} illustrates the distinction. A formal $u$-disc has no evaluation at $u=1$, so the framed condition $g(1)=I$ in \Cref{lem:scalar-splitting} has no direct formal counterpart. That lemma also uses division by $n-\lambda$ and complex Cauchy estimates; it supplies no non-Archimedean convergence estimate. The spectral decomposition theorem over non-Archimedean fields relies on the analysis of~\cite{HYZZ}, while its general complex-analytic analogue fails because of Stokes phenomena~\cite[Remark~3.2]{KKPYY}.

The scalar operation uses the residual endomorphism at $u=0$ and the central differential $d(a/u)$. Its proof is algebraic in the coefficients of a given $F$-bundle, separately in each category. We apply it to the convergent non-Archimedean $A$-model of~\cite{KKPYY}, without changing analytic categories. Neither the complex gauges nor formal power-series identities alone imply rigid-analytic convergence or uniformly bounded projector poles.

\subsection{Scalar residualisation}

For an $F$-bundle $(\cH,\nabla)/S$, put
\[
 \kappa_{\nabla}
 \coloneqq\left.\nabla_{u^2\partial_u}\right|_{\cH/u\cH}.
\]
This is the residual endomorphism used in the spectral decomposition theory~\cite[\S3]{HYZZ} and in the definition of atoms~\cite[\S\S3--5]{KKPYY}.

\begin{definition}\label{def:one-branch-category}
All characteristic polynomials below refer to the underlying vector bundle; in the super setting the connections and isomorphisms are even and the parity is retained.  Let $(\cH,\nabla)/S$ have locally constant rank $n$ (interpreted componentwise on $S$).  For $a\in\Gamma(S,\cO_S)$, let $\Fbun_a(S)$ be the full subcategory of $F$-bundles for which the characteristic polynomial of the residual endomorphism satisfies
\[
 \det(TI-\kappa_{\nabla})=(T+a)^n
 \qquad\text{in }\Gamma(S,\cO_S)[T].
\]
Let $\Fbun_{\mathrm{1br}}(S)$ be the full subcategory consisting of the objects which belong to $\Fbun_a(S)$ for some $a$, and let $\Fbun_0(S)$ denote the zero-branch category.  Over a reduced base, this condition is equivalent to requiring the single fibrewise eigenvalue $-a$; the characteristic-polynomial formulation remains unambiguous over non-reduced bases.

The variable exponential line of branch $a$ is the rank-one $F$-bundle
\[
 \ExpShift_a
 \coloneqq\left(\cO_{S\times\mathbb D_u},\ d+d(a/u)\right).
\]
Its residual endomorphism is multiplication by $-a$.
\end{definition}

\begin{lemma}\label{lem:intrinsic-branch}
For an object of $\Fbun_{\mathrm{1br}}(S)$ of locally constant rank $n$, the branch is unique and is recovered intrinsically, componentwise, by
\[
 a=-\frac{1}{n}\operatorname{tr}(\kappa_{\nabla}).
\]
Moreover, $N_{\nabla}\coloneqq\kappa_{\nabla}+aI$ satisfies $N_{\nabla}^{n}=0$.  Branch extraction commutes with arbitrary base change.
\end{lemma}

\begin{proof}
If $\det(TI-\kappa_{\nabla})=(T+a)^n$, comparison of the coefficient of $T^{n-1}$ gives $-\operatorname{tr}(\kappa_{\nabla})=na$.  Since the bases are defined over a characteristic-zero field, the integer $n$ is invertible in $\cO_S$ on each rank component; hence this identity determines $a$ uniquely.  The characteristic polynomial of $N_{\nabla}=\kappa_{\nabla}+aI$ is $T^n$, so Cayley--Hamilton gives $N_{\nabla}^{n}=0$.  Trace, characteristic polynomial and scalar translation commute with pullback, proving base-change compatibility.
\end{proof}

For constant $a$, this is the exponential-shift symmetry considered in~\cite[\S3.3]{KKPYY}; allowing $a$ to vary over the base is essential for atomic spectral components.

\begin{theorem}\label{thm:valued-residualisation-equivalence}
For each base $S$ in any of the three categories of \Cref{sec:analytic-scope}, scalar residualisation defines an equivalence
\begin{equation}\label{eq:branch-splitting-functor}
 \Fbun_{\mathrm{1br}}(S)
 \xrightarrow{\sim}
 \Gamma(S,\cO_S)_{\mathrm{disc}}\times\Fbun_0(S),
 \qquad
 (\cH,\nabla)
 \longmapsto
 \left(a,\cH,\nabla-d(a/u)I\right),
\end{equation}
where, on each connected rank component, $a=-(1/\operatorname{rk}\cH)\operatorname{tr}(\kappa_{\nabla})$, and we write $\Res_a(\cH,\nabla)=(\cH,\nabla-d(a/u)I)$.  These equivalences commute with pullbacks in the chosen category and hence give an equivalence of the corresponding fibre categories over bases.  They are symmetric monoidal with respect to addition of branches and compatible with duality:
\begin{align}
 \Res_{a+b}(\cH_1\otimes\cH_2)
 &\cong
 \Res_a(\cH_1)\otimes
 \Res_b(\cH_2),
 \label{eq:res-tensor}\\
 \Res_{-a}(\cH^{\vee})
 &\cong
 \Res_a(\cH)^{\vee}.
 \label{eq:res-dual}
\end{align}
The same statement holds for external products over different bases.
\end{theorem}

\begin{proof}
The branch function is intrinsic and base-change compatible by \Cref{lem:intrinsic-branch}.  Since $d(a/u)$ is a closed scalar one-form, subtracting it from the connection does not alter the curvature.  The total differential is
\[
 d(a/u)=\frac{d_S a}{u}-\frac{a}{u^2}\,du.
\]
For a vector field $\xi$ on~$S$ one obtains
\begin{equation}\label{eq:valued-horizontal-residual}
(\nabla^{\mathrm{res}})_{u\xi}=\nabla_{u\xi}-\xi(a)I,\qquad (\nabla^{\mathrm{res}})_{u^2\partial_u}=\nabla_{u^2\partial_u}+aI.
\end{equation}
Both correction terms are regular endomorphisms of the intrinsic lattice.  Hence the $F$-bundle pole conditions are preserved.  The residual endomorphism becomes
\[
 \kappa_{\nabla^{\mathrm{res}}}=\kappa_{\nabla}+aI,
\]
whose characteristic polynomial is $T^n$; thus the result lies in $\Fbun_0(S)$.  Subtracting a central one-form does not alter the intertwining equation for morphisms.  Conversely, from a pair $(a,\cH,\nabla^0)$ with zero branch, adding $d(a/u)I$ produces characteristic polynomial $(T+a)^n$.  The two operations are mutually inverse and commute with pullback, hence define an equivalence of stacks. 
For tensor products, the residual endomorphism is
\[
 \kappa_{\nabla_1\otimes\nabla_2}
 =\kappa_{\nabla_1}\otimes I+I\otimes\kappa_{\nabla_2}.
\]
Writing $\kappa_{\nabla_i}=-a_iI+N_i$ with $N_i$ nilpotent, the commuting sum $N_1\otimes I+I\otimes N_2$ is nilpotent, so the branch is $a_1+a_2$.  Since
\[
 d((a_1+a_2)/u)=d(a_1/u)+d(a_2/u),
\]
this gives~\eqref{eq:res-tensor}.  The dual residual endomorphism is $-\kappa_{\nabla}^{\mathsf t}$, hence its branch is $-a$, and the dual connection changes a scalar one-form by its negative; this proves~\eqref{eq:res-dual}.  Symmetry and associativity are inherited from the tensor category of $F$-bundles~\cite[\S3.2]{KKPYY}.
\end{proof}

\begin{proposition}\label{prop:trace-centred-splitting}
Let $S$ be a base as above, and let $\Fbun^{\mathrm{cent}}(S)$ denote the full subcategory of $F$-bundles for which $\operatorname{tr}\kappa_\nabla=0$.  Every $F$-bundle of locally constant positive rank $n$ admits an intrinsic \emph{trace-centred} presentation
\[
 a=-\frac{1}{n}\operatorname{tr}(\kappa_\nabla),\qquad
 (\cH,\nabla^{\mathrm{cent}})=(\cH,\nabla-d(a/u)I),
\]
and this gives a base-change-compatible equivalence
\[
 \Fbun(S)\simeq\Gamma(S,\cO_S)_{\mathrm{disc}}\times
 \Fbun^{\mathrm{cent}}(S).
\]
It preserves specified horizontal pairings and horizontal automorphisms in the sense of \Cref{prop:enhanced-residualisation}.  If the original residual endomorphism has a single eigenvalue along a specified closed subspace $Z\subset S$, the trace-centred residual endomorphism is nilpotent at every geometric point of $Z$ (and need not be nilpotent away from $Z$).
\end{proposition}

\begin{proof}
Trace and formation of $\kappa_\nabla$ commute with pullback, and $n$ is invertible in characteristic zero.  The scalar one-form $d(a/u)$ is closed, has the permitted pole orders and changes $\kappa_\nabla$ to $\kappa_\nabla+aI$, whose trace vanishes.  Adding $d(a/u)I$ reconstructs the original connection on the same lattice.  Since an intertwining map preserves the trace, this is an equivalence on morphisms as well as objects.  The proof of \Cref{prop:enhanced-residualisation} uses only the centrality of this correction and hence applies without the one-branch assumption.  Finally, at each geometric point of $Z$ all the eigenvalues of $\kappa_\nabla$ agree, so Cayley--Hamilton makes its trace-centred translate nilpotent on that fibre.  Nilpotence over a non-reduced scheme-theoretic fixed locus would require the stronger characteristic-polynomial condition there; nothing in the trace calculation implies nilpotence away from $Z$.
\end{proof}

The operation~\eqref{eq:branch-splitting-functor} changes the connection on the same lattice; it is not a lattice-preserving gauge transformation.  On a complex punctured disc the scalar exponential $e^{-a/u}$ relates the two connections and is single-valued, so their topological monodromy is the same.  This observation is restricted to the complex analytic category: neither the exponential nor a Riemann--Hilbert correspondence is invoked to prove the formal or non-Archimedean equivalence.

\subsection{The horizontal action and maximality under scalar splitting}\label{sec:branch-maximality}

A geometric atomic $F$-bundle is maximal in the sense of~\cite[Definition~3.8 and Definition~5.20]{KKPYY}.  The equivalence above is an equivalence of $F$-bundles, not an assertion that the zero-branch factor is maximal.  The distinction is already visible in the horizontal action.

\begin{proposition}\label{prop:scalar-shift-maximality}
Let $(\cH,\nabla)/S$ be a one-branch $F$-bundle and put $V=\cH/u\cH$.  For a vector field $\xi$ on $S$, write $\mu_{\nabla}(\xi)=\left.\nabla_{u\xi}\right|_V$.  Then
\[
 \mu_{\nabla^{\mathrm{res}}}(\xi)
 =\mu_{\nabla}(\xi)-\xi(a)I_V.
\]
Consequently, the induced map $T_S\to\End(V)/\cO_S I_V$ is unchanged by residualisation.  At a point $s$, the original $F$-bundle is maximal if and only if there is $v\in V_s$ for which
\[
 T_{S,s}\longrightarrow V_s,\qquad
 \xi\longmapsto
 \bigl(\mu_{\nabla^{\mathrm{res}}}(\xi)+\xi(a)I\bigr)v
\]
is an isomorphism.  The residual factor need not itself be maximal, even when the original $F$-bundle is maximal.
\end{proposition}

\begin{proof}
The formula for the horizontal action is the restriction of~\eqref{eq:valued-horizontal-residual} to $u=0$.  The difference between the two endomorphisms is scalar, which proves the assertion about the quotient by scalars.  The criterion for maximality is precisely the cyclic-vector definition in~\cite[Definition~3.8]{KKPYY}, after substituting $\mu_{\nabla}=\mu_{\nabla^{\mathrm{res}}}+da\,I$.  To see that maximality need not survive, take $S=\operatorname{Spf}k[[t]]$ and the rank-one $F$-bundle $(\cO,d+d(t/u))$. Here $\mu_{\nabla}(\partial_t)=I$, so $1\in V_s$ is cyclic and the bundle is maximal.  Its branch is $a=t$, however, and its zero-branch residualisation is $(\cO,d)$, for which $\mu_{\nabla^{\mathrm{res}}}=0$.  The latter is not maximal over the one-dimensional base.
\end{proof}

Thus maximality belongs to the reconstructed atom, not necessarily to its centred factor. The unit-direction descent of \Cref{thm:unit-direction-residual-descent} reflects the disappearance of the cohomological unit after quotienting the horizontal action by scalars.

\subsection{Pairings, half-turn transport and enhanced atoms}

Let $\nu(s,u)=(s,-u)$. A horizontal pairing in the sense of~\cite[\S6.4]{KKPYY} is a morphism $\psi:\cH\otimes\nu^*\cH\to\cO$, even, non-degenerate and symmetric under interchange with $\nu$. We allow arbitrary non-degenerate horizontal pairings below. The two central corrections cancel on $\cH\otimes\nu^*\cH$.

We distinguish the specified pairing $\psi$ from the Euler form $\chi$ obtained by half-turn parallel transport between $u=1$ and $u=-1$, and from a Serre or integral enhancement. Half-turn transport requires a complex analytic domain and a path; integral or Serre data require an independent definition. Preserving $\psi$ alone supplies neither.

\begin{proposition}\label{prop:enhanced-residualisation}
Let $(\cH,\nabla)/S$ be an $F$-bundle and put
$a=-\operatorname{tr}(\kappa_\nabla)/\operatorname{rk}\cH$.  Suppose that it is equipped with a non-degenerate horizontal pairing
\[
 \psi:\cH\otimes\nu^*\cH\longrightarrow
 \cO_{S\times\mathbb D_u}.
\]
Then the same morphism $\psi$ is horizontal for
$\nabla^{\mathrm{res}}=\nabla-d(a/u)I$.  Any symmetry of $\psi$ under
interchange of the two factors and pullback by $\nu$ is unchanged.  If
$T$ is a specified horizontal bundle automorphism, it remains
horizontal after residualisation; algebraic identities relating $T$
and $\psi$ are preserved.  These statements hold in the complex
analytic, formal and non-Archimedean categories, without a
Riemann--Hilbert correspondence.  If the $F$-bundle is
$\mathcal G$-equivariant and the group acts trivially on~$u$, the trace
scalar is $\mathcal G$-invariant and the trace-centred connection
retains the action and the parity grading.  This assertion about $T$
concerns an automorphism of the whole $F$-bundle; it does not identify
it with a Serre automorphism defined only on the fibre at $u=1$.
\end{proposition}

\begin{proof}
The connection on $\cH\otimes\nu^*\cH$ changes by the central form
\[
 -d(a/u)\otimes I-I\otimes\nu^*d(a/u).
\]
Since $a$ is pulled back from the base and $\nu(u)=-u$, we have
$\nu^*d(a/u)=d(-a/u)=-d(a/u)$.  The two terms therefore cancel,
including their components in the directions tangent to $S$.  Thus
the covariant derivative of $\psi$ is identical before and after
residualisation.  Non-degeneracy and the stated symmetry depend only
on the underlying morphism $\psi$, which has not changed.  On
$\operatorname{End}(\cH)$ the induced connection changes by the
commutator with $-d(a/u)I$, which is zero.  Hence horizontality of $T$
and all algebraic relations involving $T$ and $\psi$ are preserved.
For an equivariant $F$-bundle the action on $\cH/u\cH$
intertwines $\kappa_\nabla$, since the group acts trivially on~$u$.
Taking the ordinary trace gives an invariant scalar $a$.  Thus the
central correction is equivariant and even, so the given $\mathcal G$
action and the action of $\epsilon_G$ are unchanged.  All these
identities take place in the sheaves of meromorphic differential forms
in the category under consideration.
\end{proof}

\begin{theorem}\label{thm:paired-residualisation-equivalence}
For an admissible base $S$, let
$\Fbun^{\mathrm{pair}}_{\mathrm{1br}}(S)$ be the groupoid of one-branch
$F$-bundles endowed with a specified non-degenerate horizontal pairing
as in \Cref{prop:enhanced-residualisation}, and let
$\Fbun^{\mathrm{pair}}_0(S)$ be its zero-branch counterpart.  Morphisms
are isomorphisms of $F$-bundles preserving the pairings.  Scalar
residualisation induces an equivalence
\[
 \Fbun^{\mathrm{pair}}_{\mathrm{1br}}(S)
 \xrightarrow{\ \sim\ }
 \Gamma(S,\cO_S)_{\mathrm{disc}}
 \times\Fbun^{\mathrm{pair}}_0(S),
 \qquad
 (\cH,\nabla,\psi)\longmapsto
 (a,\cH,\nabla-d(a/u)I,\psi).
\]
Its inverse adds $d(a/u)I$ to the connection and leaves the lattice and
$\psi$ unchanged.  The equivalence is natural under pullback within
the chosen category and restricts to any additional algebraic
symmetry imposed on $\psi$.  It also applies to objects equipped with
specified horizontal automorphisms and algebraic compatibility
relations.  On maximal one-branch $F$-bundles the corresponding
subcategory on the right is defined by the cyclicity condition of
\Cref{prop:scalar-shift-maximality}; the zero-branch factor is not
required to be maximal.
\end{theorem}

\begin{proof}
The unpaired equivalence, including its inverse and its naturality,
is \Cref{thm:valued-residualisation-equivalence}.  By
\Cref{prop:enhanced-residualisation}, the same pairing is horizontal
before and after either scalar correction, and the specified
additional identities remain true.  An isomorphism between two paired
one-branch objects preserves the branches by
\Cref{lem:intrinsic-branch}; it is therefore precisely an isomorphism
between their residual factors preserving the pairing.  Conversely,
such an isomorphism commutes with the scalar corrections and hence
intertwines the reconstructed connections.  The two constructions
are mutually inverse on objects and morphisms, and they commute with
pullback.  The final assertion is the maximality criterion of
\Cref{prop:scalar-shift-maximality}, applied to the reconstructed
connection rather than to its zero-branch factor.
\end{proof}

A horizontal automorphism of the full $F$-bundle survives centring,
since the connection on $\End(\cH)$ is unchanged. In contrast, the
duality operator on the fibre $u=1$ of~\cite[\S6.4]{KKPYY}
requires further data; the Euler form is compared here only by
complex half-turn transport, not by non-Archimedean Riemann--Hilbert.

\begin{lemma}\label{lem:scalar-half-turn}
Let $S$ be complex analytic, fix $s\in S$, and let $\gamma$ be a path
in the punctured $u$-disc from $u_0$ to $u_1$.  For the parallel
transports of $\nabla$ and $\nabla^{\mathrm{res}}$, one has
\[
 P^{\mathrm{res}}_{\gamma}
 =\exp\!\left(\frac{a(s)}{u_1}-\frac{a(s)}{u_0}\right)
 P^{\nabla}_{\gamma}.
\]
In particular, the monodromies around $u=0$ agree.  If $1,-1$ belong
to the chosen domain, define the non-degenerate bilinear forms on
$\cH_{(s,1)}$ by
\[
 \chi_{\nabla}(v,w)
 =\psi_{(s,1)}\bigl(v,P^{\nabla}_{1\to-1}w\bigr),
 \qquad
 \chi_{\mathrm{res}}(v,w)
 =\psi_{(s,1)}\bigl(v,P^{\mathrm{res}}_{1\to-1}w\bigr).
\]
For the same path and the same pairing $\psi$, these satisfy
\begin{equation}\label{eq:euler-half-turn-rescaling}
 \chi_{\mathrm{res}}=e^{-2a(s)}\chi_{\nabla}.
\end{equation}
Consequently the uniquely determined asymmetry automorphism $S_\chi$
characterised by
$\chi(v,w)=\chi(w,S_\chi v)$ is the same for the two
forms.  Both the original half-turn pairing and its asymmetry
automorphism are therefore recovered from the branch-labelled
residual presentation.  Where the complex analytic Serre automorphism
is defined by the monodromy of the $u$-connection, as in
\cite[\S6.4]{KKPYY}, it is also unchanged by scalar residualisation.
\end{lemma}

\begin{proof}
Put $h(u)=a(s)/u$ on the punctured fibre.  Since
$\nabla^{\mathrm{res}}=\nabla-dh\,I$, the section $e^{h}v$ is
$\nabla^{\mathrm{res}}$-horizontal whenever $v$ is
$\nabla$-horizontal.  Taking its values at the endpoints yields the
first formula.  The function $e^{h}$ is single-valued on the
punctured complex disc, so its endpoint ratio is one along a closed
loop.  For a path from $1$ to $-1$ the ratio is $e^{-2a(s)}$.
Substitution into the definition of $\chi$ gives
\eqref{eq:euler-half-turn-rescaling}.

For completeness, non-degeneracy of $\psi$ and invertibility of
parallel transport imply non-degeneracy of each $\chi$.  Hence the
relation $\chi(v,w)=\chi(w,S_\chi v)$ determines a unique invertible
endomorphism $S_\chi$.  Multiplying $\chi$ by a non-zero scalar
multiplies both sides by the same scalar and does not change
$S_\chi$.  The final assertion follows by applying the inverse
scalar shift to the residual connection.  The equality of monodromies
established above gives the additional statement about the
monodromy-defined Serre automorphism.
\end{proof}

\begin{remark}\label{rem:half-turn-not-invariant}
The rescaling is already non-trivial in rank one.  With
$\nabla=d+d(a/u)$ and the standard pairing $\psi=1$, the transport
from $1$ to $-1$ is multiplication by $e^{2a}$; after residualisation
it is the identity.  Thus $\chi_{\nabla}=e^{2a}\chi_{\mathrm{res}}$,
whereas the asymmetry automorphism is the identity in both cases.
The scalar part must not be discarded if the actual Euler pairing,
rather than only its asymmetry, is to be reconstructed.
\end{remark}

\begin{corollary}\label{cor:calabi-yau-atom-no-loss}
For a Calabi--Yau $d$-fold, trace-labelled centring reconstructs
any specified horizontal pairing and leaves the scalar cohomological
Serre operator $(-1)^dI$ unchanged. When defined complex analytically,
the Euler form is recovered by \eqref{eq:euler-half-turn-rescaling}.
\end{corollary}
\begin{proof}
Use \Cref{prop:trace-centred-splitting,prop:enhanced-residualisation,lem:scalar-half-turn}
and the identity $\mathsf S_X\simeq[d]$. The operator alone does not
determine the numerical Euler form.
\end{proof}

For $a\ne0$ the scalar exponential is not a formal Laurent gauge, and non-Archimedean convergence is not automatic. Our constructions use $d(a/u)$, not Riemann--Hilbert or integral data~\cite[\S6.4]{KKPYY}.

\subsection{Dependence on the trace scalar}\label{sec:forgetting-branch}

The trace scalar is part of the equivalences above. Omitting it changes the fixed-base classification, even when the complex Euler form and its Serre asymmetry are retained.

\begin{proposition}\label{prop:forgetting-branch-loses-information}
Let $S$ be a connected admissible base over a characteristic-zero field, and let $(\cH,\nabla^0,\psi)/S$ be a positive-rank $F$-bundle equipped with a non-degenerate horizontal pairing
$\psi:\cH\otimes\nu^*\cH\to\cO$, where $\operatorname{tr}(\kappa_{\nabla^0})=0$.
For $a\in\Gamma(S,\cO_S)$ put
\[
 \nabla^a=\nabla^0+d(a/u)I_{\cH}.
\]
Then the following assertions hold.

\begin{enumerate}[label=\textup{(\roman*)}]
\item Every $(\cH,\nabla^a,\psi)$ is a paired $F$-bundle, its trace-labelled residual presentation is $(a,\cH,\nabla^0,\psi)$, and its unlabelled trace-centred presentation is $(\cH,\nabla^0,\psi)$.  Conversely, every paired $F$-bundle with this unlabelled centred presentation is obtained in this way.  Over the fixed base $S$, distinct scalars give non-isomorphic original $F$-bundles with their prescribed lattices.  Thus the fibre on isomorphism classes of the operation forgetting the trace label is parametrised by $\Gamma(S,\cO_S)$.
\item Suppose that $S$ is complex analytic and that the connections and pairings are defined on a common punctured $u$-domain containing $1$ and $-1$, with a fixed half-turn path.  The associated Euler forms satisfy
\[
 \chi_a=e^{2a}\chi_0,
\]
whereas the monodromy and the Serre asymmetry automorphism are independent of $a$.  For $a,b\in\Gamma(S,\cO_S)$ one has $\chi_a=\chi_b$ if and only if $e^{2(a-b)}=1$.  On connected $S$ this is equivalent to $a-b\in\pi i\mathbb Z$ as a constant function.  Consequently, for a fixed centred paired object, the Euler form together with the value of $a$ at one specified point determines $a$ on $S$.
\end{enumerate}

In particular, on a complex disc $S$ with coordinate $t$, the paired rank-one families
\[
 (\cO,d+d(t/u),1),\qquad
 (\cO,d+d((t+\pi i)/u),1)
\]
are distinct maximal $F$-bundles over the fixed base.  They have the same unlabelled centred connection, horizontal pairing, Euler form and monodromy-defined Serre automorphism.
\end{proposition}

\begin{proof}
The scalar form $d(a/u)$ is closed, and its addition preserves the pole bounds defining an $F$-bundle.  Since $\nu^*d(a/u)=-d(a/u)$, its contributions to the connection on $\cH\otimes\nu^*\cH$ cancel.  Hence the given pairing remains horizontal.  The residual endomorphism changes by $-aI_{\cH/u\cH}$, and consequently
\[
 -\frac{\operatorname{tr}(\kappa_{\nabla^a})}{\operatorname{rk}\cH}=a.
\]
Trace-centring therefore recovers $\nabla^0$ for every $a$.  Conversely, the inverse formula in \Cref{prop:trace-centred-splitting} shows that any original connection with this centred presentation is $\nabla^a$ for its intrinsic trace scalar.  A lattice-preserving isomorphism intertwines the residual endomorphisms and hence preserves their traces; it follows that $a=b$ whenever the corresponding original objects are isomorphic over $S$.  This proves (i).  The assertion concerns all paired $F$-bundles: imposing maximality or geometric origin may restrict the allowable scalars.

For (ii), the pairing $\chi_0$ is non-degenerate, and the transport formula in \Cref{lem:scalar-half-turn} gives $\chi_a=e^{2a}\chi_0$.  The exponential scalar is single-valued on the punctured $u$-disc, so it does not change local monodromy.  Multiplying a non-degenerate bilinear form by a non-zero scalar leaves its asymmetry automorphism unchanged.  Non-degeneracy also shows that equality of the two Euler forms is equivalent to $e^{2(a-b)}=1$.  An analytic function on connected $S$ whose values lie in the discrete set $\pi i\mathbb Z$ is constant, proving the final equivalence.

For the displayed rank-one families, $\mu_{\nabla}(\partial_t)=I$ on $\cH/u\cH$; the section $1$ is therefore cyclic, so both families are maximal by \Cref{prop:scalar-shift-maximality}.  Their trace scalars differ by $\pi i$, hence they are not isomorphic over $S$, while $e^{2(t+\pi i)}=e^{2t}$.  All the claimed equalities now follow from (i) and (ii).
\end{proof}

These fixed-base rank-one examples are not geometric atoms. The Serre operator does not recover $a$; the marked Calabi--Yau $A$-model fixes $a=t^0$ under \Cref{cor:calabi-yau-unit-branch}. The unmarked assertion requires the unit suspension of \Cref{prop:CY-unmarked-atom-recovery}.

\subsection{Residual atoms}

Fix a motivic symmetry pair $(G,\epsilon_G)$ and a characteristic-zero base field~$K$.  Let $k$ be the coefficient field used for the corresponding non-Archimedean $A$-model $F$-bundle.  The notation $\Atoms_G^{K,F}$ and $\CF_G^F(X)$ is that of~\cite[\S5.3]{KKPYY}.
Choose a smooth projective variety $X$, a rigid point $b$ in the unramified spectral locus $U_X$, and a point $\tilde b$ of the reduced spectral cover above $b$.  The spectral decomposition of~\cite{HYZZ,KKPYY} supplies a maximal geometric $G$-atomic $F$-bundle germ $(\cA_{X,b,\tilde b},\nabla_{X,b,\tilde b})/(B_{\tilde b},b_{\tilde b})$, where the germ $B_{\tilde b}$ is the base of the corresponding factor and need not coincide with $U_X$.  Its trace scalar is $a_{X,b,\tilde b}$, and scalar splitting gives the trace-centred germ $(\cA_{X,b,\tilde b},\nabla^{\mathrm{res}}_{X,b,\tilde b})$.  The single-eigenvalue condition in~\cite[Definition~5.20]{KKPYY} is imposed along the $G$-fixed locus; the trace-centred residual endomorphism is therefore fibrewise nilpotent at its geometric points, but it need not have a single branch on the whole base germ.  The label $a$ denotes the trace scalar on the entire germ.  It agrees with the negative of the unique eigenvalue precisely where the one-branch condition holds; away from that locus it need not be an eigenvalue at all.  Thus ``trace-labelled'' does not assert the existence of a global spectral branch.  Restriction of its full connection to the $u$-direction gives a classifying morphism to $\mathscr M_{\leq2,k}^{\mathrm{unfr}}$.  This morphism records part of the existing data; it is neither a further choice nor a replacement for the horizontal connection.

\begin{definition}\label{def:residual-atom-category}
A \emph{local trace-labelled residual $G$-atomic presentation} is the germ
\[
 \mathfrak r_{X,b,\tilde b}
 =\bigl((B_{\tilde b},b_{\tilde b}),a_{X,b,\tilde b},
 (\cA_{X,b,\tilde b},\nabla^{\mathrm{res}}_{X,b,\tilde b}),
 \epsilon_{X,b,\tilde b},\rho_{X,b,\tilde b},\mathcal T_{X,b,\tilde b}\bigr),
\]
where the full horizontal and vertical connection, the prescribed lattice, parity and $G$-action are retained, together with whichever specified enhancements $\mathcal T$ are part of the atomic theory.  Isomorphisms are isomorphisms of such trace-labelled germs preserving these structures.  In the undecorated $G$-atomic theory, $\mathcal T$ is empty.  For an enhanced theory, $\mathcal T$ is retained only when its enhancement and corresponding notion of isomorphism are separately defined; the resulting enriched equivalence is distinct from the undecorated one.  In particular, maximality is required of the $F$-bundle reconstructed by adding $d(a/u)$; it is not imposed on the trace-centred factor.  The centred factor is zero-branch on a locus only where the original factor has one spectral branch.

Following~\cite[Definition~5.21]{KKPYY}, we define the equivalence relation generated by the following two relations: (i) isomorphism of trace-labelled residual presentations; (ii) presentations arising from the same variety $X$ at points $\tilde b_1,\tilde b_2$ on the same connected component of its unramified reduced spectral cover.  An equivalence class is called a \emph{residual $G$-atom}, and the set of these classes is denoted by $\RAtoms_G^K$.  The second relation does not assert that the two germs, or the vertical Jordan profiles at the chosen points, are isomorphic.  In particular, the entire spectral component and a classifying map from that component are not additional data of a residual atomic class.
\end{definition}

\begin{proposition}\label{prop:atomic-correspondence-residualises}
Scalar trace-centred splitting induces a bijection
\[
 \mathfrak{res}:\Atoms_G^{K,F}\xrightarrow{\sim}\RAtoms_G^K.
\]
Every atomic correspondence induces the corresponding equivalence of residual atoms and preserves the intrinsic lattice, parity, $G$-action and spectral multiplicity.  Compatibility with specified pairings or other enhancements requires the corresponding enriched equivalence, treated in \Cref{prop:paired-atomic-classes}.
\end{proposition}

\begin{proof}
By \Cref{prop:trace-centred-splitting}, scalar splitting and its inverse are functorial equivalences on all the underlying germs: the trace scalar is retained as a label, and the full connection is reconstructed by adding $d(a/u)I$ on the same lattice.  The trace-centred residual endomorphism has trace zero and does not, by itself, determine $a$.  This distinction is necessary because the atomic condition of~\cite[Definition~5.20]{KKPYY} only requires one spectral eigenvalue along the $G$-fixed locus.  The trace scalar is $\mathcal G$-invariant by \Cref{prop:enhanced-residualisation}, so both functors commute with $G$-equivariant isomorphisms and preserve parity and the specified horizontal pairings.  Additional enrichments are preserved only in the precise circumstances covered by that proposition or by a separately defined compatible enhancement.  The maximality condition is recovered, rather than assumed for the zero-branch factor, by \Cref{prop:scalar-shift-maximality}.  Thus these functors preserve the first generating relation of Definition~\ref{def:residual-atom-category} in both directions.

For the second generating relation, choose two points $\tilde b_1,\tilde b_2$ in the same connected component of the unramified spectral cover of one variety.  The original germs are equivalent by~\cite[Definition~5.21(2)]{KKPYY}, and the residual presentations are declared equivalent by exactly the corresponding relation in Definition~\ref{def:residual-atom-category}.  Conversely, reconstruction returns the original germs at those same two points.  The two functors therefore identify the equivalence relations generated by (i) and (ii), and induce mutually inverse maps on the resulting sets of classes.  The original prescribed lattice, parity, action and spectral multiplicity are unchanged by scalar splitting.  This proves the bijection for the undecorated geometric atomic equivalence relation of \cite[Definition~5.21]{KKPYY}.  It does not, without a separately established notion of enhanced equivalence along spectral components, identify the quotient sets of any proposed integral or Serre-enhanced theory.
\end{proof}

\begin{proposition}\label{prop:paired-atomic-classes}
Fix a collection of geometric $G$-atomic $F$-bundles obtained from
$G$-equivariant $A$-model $F$-bundles of smooth projective varieties,
each equipped with a specified even, non-degenerate horizontal pairing.
Assume that the pairing is defined on the domain of spectral
decomposition and induces non-degenerate horizontal pairings on its
spectral factors, as in~\cite[\S6.4]{KKPYY}.  We regard the variety
\emph{together with this global pairing} as the source of a paired
atomic germ.

Let $\Atoms_G^{K,F;\psi}$ denote the quotient of these paired geometric
germs by the equivalence relation generated by isomorphisms preserving
the pairing and by pairs of germs obtained from the \emph{same globally
paired $F$-bundle} at points of the same connected component of its
unramified reduced spectral cover.  Let $\RAtoms_G^{K;\psi}$ be defined
by the corresponding two relations on trace-labelled residual
presentations, with the pairing retained.  Then trace-centring induces
a bijection
\[
 \Atoms_G^{K,F;\psi}\xrightarrow{\ \sim\ }\RAtoms_G^{K;\psi}.
\]
The construction is compatible with restriction and admissible base
change.  It leaves the $G$-action, parity and pairing unchanged, and
recovers maximality from the reconstructed $F$-bundle, not from its
trace-centred factor.

In the complex analytic setting, suppose further that the source
pairings admit the half-turn construction on a domain containing
$u=1,-1$, with a corresponding monodromy-defined Serre automorphism
on the fibre at $u=1$.  The same conclusion holds for the equivalence
relations additionally preserving these data, provided the Euler
pairing in the residual presentation is taken to be
$\chi_{\mathrm{res}}=e^{-2a}\chi_{\nabla}$ and the Serre automorphism is
retained.  No assertion about a general non-Archimedean Serre or
integral enhancement is implicit in this statement.
\end{proposition}

\begin{proof}
For an isomorphism of paired atomic germs, the trace scalar is
intrinsic and is preserved by the isomorphism.  The inverse scalar
corrections consequently intertwine the connections on the same
prescribed lattices.  By \Cref{prop:enhanced-residualisation}, they
also preserve the specified pairings, the $G$-action and parity.  Conversely, an isomorphism $f$ of trace-labelled centred presentations
intertwines $\nabla^0$ and $\nabla^{0\prime}$ and identifies their trace
scalars $a$ and $a'$. Therefore $f$ also intertwines
$\nabla^0+d(a/u)I$ and $\nabla^{0\prime}+d(a'/u)I$, on the same
lattices and with the same pairings. Both constructions preserve the
first generator of the equivalence relation.

For the second generator, two germs are taken from a single
$G$-equivariant $F$-bundle with a fixed global pairing, at points of
one connected spectral component.  Trace-centring is performed on
each germ without altering its source, its component or its induced
pairing.  Conversely, adding $d(a/u)I$ recovers the original germs
with the same source and component labels.  Thus the two operations
respect the second generator in both directions.  They induce inverse
maps on the quotient sets.  Their compatibility with base change and
the maximality qualification follow from
\Cref{prop:trace-centred-splitting,prop:scalar-shift-maximality}.

Finally, in the complex analytic setting,
\Cref{lem:scalar-half-turn} gives
$\chi_{\mathrm{res}}=e^{-2a}\chi_{\nabla}$ and equality of the
monodromy operators.  As the factor $e^{-2a}$ is invertible, the
asymmetry automorphism of the Euler form is unchanged.  Hence both
generators also respect these analytic enhancements whenever they
are defined on the chosen global source.  The argument uses neither
an integral lattice nor a non-Archimedean Riemann--Hilbert comparison.
\end{proof}

A common spectral component does not identify unrelated pairings: the global source pairing is part of the hypothesis. The bijections transport the birational criteria of~\cite[\S\S5--6]{KKPYY}, not new invariants, and say nothing about integral enhancements under blowups.

\subsection{Descent along the unit direction}\label{sec:formal-rigid-a-model}

Scalar trace-centering is already non-trivial along the universal coordinate of the cohomological unit.  This direction occurs in the $A$-model base of every smooth projective variety and provides a test of the variable-branch formalism.

Let $t^0$ denote the coordinate corresponding to the unit $\mathbf 1\in H^0(X)$, let
\[
 B_X^0\coloneqq\{t^0=0\}\subset B_X,
 \qquad
 p:B_X\longrightarrow B_X^0
\]
be the projection which forgets the unit coordinate, and write $(\cH_X,\nabla_X)/B_X$ for the $A$-model $F$-bundle.  In the formal category we use a germ in the unit coordinate.  In the non-Archimedean analytic $A$-model of~\cite[\S3.5.2]{KKPYY}, the base has a full analytic affine-line factor in $t^0$.  We use this global factor only when comparing representatives of one connected spectral component; the connection identities below are valid already on product germs.

\begin{theorem}\label{thm:unit-direction-residual-descent}
On the formal $A$-model base, and on the non-Archimedean domain of the $A$-model $F$-bundle of~\cite{KKPYY}, the quantum product is independent of $t^0$.  In the standard quantum-cohomology trivialisation, $\cH_X\cong p^*\cH_X^0$ and the connection satisfies
\begin{equation}\label{eq:unit-direction-connection-splitting}
 \nabla_X=p^*\nabla_X^0+d(t^0/u)I,
\end{equation}
where $(\cH_X^0,\nabla_X^0)$ denotes the restriction to $t^0=0$.  Let $\alpha$ be a connected component of the unramified reduced spectral cover over $B_X^0$, with trace scalar $a_{\alpha,0}$.  Its extension along the unit line has trace scalar
\begin{equation}\label{eq:unit-direction-branch-translation}
 a_\alpha=p^*a_{\alpha,0}+t^0,
\end{equation}
and its residual factor descends:
\begin{equation}\label{eq:unit-direction-residual-pullback}
 \Res_{a_\alpha}(\cH_\alpha,\nabla_\alpha)
 \cong
 p^*\Res_{a_{\alpha,0}}
 (\cH_{\alpha,0},\nabla_{\alpha,0}).
\end{equation}
In particular, the classifying morphism of the trace-centred residual factor is constant along the unit direction.  Before scalar residualisation, the atomic $F$-bundle is not isomorphic, in the prescribed-lattice category, to the pullback of an $F$-bundle on $B_X^0$.
\end{theorem}

\begin{proof}
The genus-zero string equation and the unit axiom give, coefficient by coefficient in the formal Novikov and cohomology variables, independence of quantum multiplication from $t^0$ and $\mathbf 1\star(-)=I$.  As the non-Archimedean connection is defined by the convergent series of these same coefficients on its analytic domain~\cite[\S3.5.2]{KKPYY}, the same identity holds on the convergence domain.  No complex analytic continuation is used.  In our sign convention the unit-direction and vertical covariant derivatives are
\[
 \nabla_{\partial_{t^0}}=\partial_{t^0}+u^{-1}I,
 \qquad
 \nabla_{\partial_u}=\partial_u-u^{-2}(\mathsf{Eu}\star(-))+u^{-1}\mathsf{Gr}.
\]
The unit equation and the existence of the $A$-model connection are as in~\cite[\S3.5]{KKPYY}.  To determine the sign of the unit coordinate without relying on a separately written Euler-field convention, observe that the $u^{-2}$-coefficient of
$[\nabla_{\partial_{t^0}},\nabla_{\partial_u}]$ is
$I-\partial_{t^0}(\mathsf{Eu}\star(-))$.
Flatness and the unit axiom therefore give
\[
 \partial_{t^0}\mathsf{Eu}=\mathbf1,\qquad
 \mathsf{Eu}(t^0,b)=\mathsf{Eu}(0,b)+t^0\mathbf1.
\]
This computation fixes the sign relative to the displayed covariant derivatives; it is independent of conventions used to present the Euler vector field.  All connection components tangent to $B_X^0$ are therefore pulled back from the zero slice, whereas the two components involving $t^0$ differ from their pullbacks by
\[
 \frac{dt^0}{u}I-\frac{t^0}{u^2}\,du\,I
 =d(t^0/u)I.
\]
This proves~\eqref{eq:unit-direction-connection-splitting}.
The residual endomorphism is $\kappa_{\nabla_X}=-L_{\mathsf{Eu}}$.  Hence translating the Euler element by $t^0\mathbf 1$ translates every eigenvalue of $\kappa_{\nabla_X}$ by $-t^0$.  Taking traces of the residual endomorphism on each transported spectral factor gives~\eqref{eq:unit-direction-branch-translation}.  This argument does not assume that its residual endomorphism has a single eigenvalue away from the $G$-fixed locus.  The translation compatibility of the spectral factors is proved in \Cref{lem:unit-equivariant-spectral-factors}.  Finally,
\begin{align*}
 \nabla_\alpha-d(a_\alpha/u)I
 &=p^*\nabla_{\alpha,0}+d(t^0/u)I
   -d\bigl((p^*a_{\alpha,0}+t^0)/u\bigr)I\\
 &=p^*\bigl(\nabla_{\alpha,0}-d(a_{\alpha,0}/u)I\bigr),
\end{align*}
which gives~\eqref{eq:unit-direction-residual-pullback}.  After residualisation, the component in the unit direction is $\partial_{t^0}$, so the residual factor is a pullback.  Conversely, the function
\[
 -\frac{1}{\operatorname{rk}\cH_X}\operatorname{tr}(\kappa_{\nabla_X})
\]
is invariant under every isomorphism preserving the intrinsic lattice.  By the preceding calculation it differs from its restriction to $t^0=0$ by $t^0$, whereas the corresponding function for a pullback from $B_X^0$ is independent of $t^0$.  These functions cannot agree on a product neighbourhood with a non-trivial unit coordinate.  Thus no such lattice-preserving isomorphism exists.  This argument is algebraic in the coefficients of the connection and applies equally to the formal and non-Archimedean models.
\end{proof}

On a complex analytic punctured $u$-disc the scalar gauge $e^{-t^0/u}$ removes $d(t^0/u)$, but it does not extend holomorphically through $u=0$ on the prescribed lattice.  The same exponential need not define a gauge in the formal or non-Archimedean category.  Thus \Cref{thm:unit-direction-residual-descent} is not a lattice-preserving gauge equivalence: it is scalar residualisation, defined algebraically on the intrinsic lattice.

\begin{corollary}\label{cor:calabi-yau-unit-branch}
Let $X$ be a connected smooth projective complex variety with $c_1(T_X)=0$ and $H^1(X,\C)=0$, and let $(\cH_X,\nabla_X)$ be its full formal big $A$-model $F$-bundle, with its cohomological frame and its specified coordinate $t^0$ along the unit.  The same assertions hold on the non-Archimedean $A$-model domain of~\cite{KKPYY} wherever the formal quantum products converge.  Put $r=\dim_{\C}H^\bullet(X,\C)$, and denote by $\kappa_X$ the vertical residual endomorphism.  Then
\[
 \kappa_X=-t^0I+N_X,\qquad N_X^{2\dim X+1}=0,
 \qquad \det(TI-\kappa_X)=(T+t^0)^r.
\]
In particular, the maximal $A$-model has a single spectral branch on the entire indicated base, not merely along the unit line, and its intrinsic trace scalar is
\[
 a_X=-\frac{1}{r}\operatorname{tr}(\kappa_X)=t^0.
\]
Consequently, \emph{on the geometrically marked $A$-model base}, the unlabelled trace-centred connection together with the unit coordinate recovers the original connection by adding $d(t^0/u)I$.  The reconstruction also preserves any specified horizontal pairing of the type considered above.  Moreover, the Higgs map $\mu_{\mathrm{cent}}(\xi)=\nabla^{\mathrm{cent}}_{u\xi}|_{u=0}$ of the centred connection has kernel exactly $\cO_{B_X}\partial_{t^0}$: the residual presentation intrinsically retains the unparametrised unit foliation, although not its affine origin or normalised coordinate.  It does not imply that the unit coordinate is retained by the unmarked atomic equivalence relation of~\cite{KKPYY}.  In particular, the shifted rank-one counterexample in \Cref{prop:forgetting-branch-loses-information} cannot satisfy this fixed geometric normalisation for both shifts: its two scalars take different values at the marked origin $t^0=0$.
\end{corollary}

\begin{proof}
By the string equation, quantum multiplication is independent of $t^0$.  The coefficient of the unit coordinate in the vertical residual endomorphism is $-t^0I$ by \Cref{thm:unit-direction-residual-descent}; hence it suffices to prove that $\kappa_X$ is strictly degree-raising on the slice $t^0=0$.  In a homogeneous cohomological basis, the Euler element on this slice has no component in $H^0(X)$, since $t^0=0$ and $c_1(T_X)=0$, and no component in $H^2(X)$, since the Euler coefficient of a degree-two coordinate is zero.  The condition $H^1(X,\C)=0$ therefore makes each nonzero component of the Euler element have cohomological degree greater than two.

We recall the degree calculation, which also controls the big quantum corrections.  For a genus-zero coefficient of multiplication by a homogeneous class $\gamma$ of degree $p>2$, with input $v$ of degree $j$, curve class $\beta$, and additional insertions $\tau_1,\ldots,\tau_m$ of degrees $d_1,\ldots,d_m$, the virtual dimension constraint forces the output to have degree
\[
 j+p+\sum_{\ell=1}^{m}(d_\ell-2)-2\langle c_1(T_X),\beta\rangle.
\]
The string equation removes additional insertions of the cohomological unit.  Every other deformation insertion has degree at least two, and $c_1(T_X)=0$.  Thus every nonzero coefficient in multiplication by $\gamma$ increases cohomological degree strictly.  This also holds coefficientwise for odd deformation variables on the super base: their parity does not alter the cohomological degree constraint.  It follows that multiplication by the Euler element on $t^0=0$ is strictly upper triangular with respect to increasing cohomological degree.  Denoting the negative of this operator by $N_X$, we have $N_X^{2\dim X+1}=0$ and $\operatorname{tr}(N_X)=0$.

The displayed formula for $\kappa_X$, its characteristic polynomial and the trace identity follow.  All calculations take place coefficientwise in the formal quantum parameters.  On the non-Archimedean analytic domain where these same coefficients converge, the triangularity and the resulting finite nilpotence identity remain valid; no comparison with complex analytic normalising gauges is involved.  Finally \Cref{thm:unit-direction-residual-descent} gives the reconstruction, and \Cref{prop:enhanced-residualisation} proves compatibility with a specified horizontal pairing.  For any tangent field $\xi$, the Higgs map of the centred connection is $\mu_{\mathrm{cent}}(\xi)=L_{\xi}-\xi(t^0)I$, where the tangent-to-cohomology identification is given by the quantum frame.  Evaluating multiplication on the cohomological unit shows that $\xi\mapsto L_{\xi}$ is injective; hence $\mu_{\mathrm{cent}}(\xi)=0$ if and only if $\xi=\xi(t^0)\partial_{t^0}$.  This proves the kernel assertion and also shows why the residual connection alone cannot single out a preferred affine parameter along the foliation.  The lack of a preferred affine parameter prevents a reconstruction over an arbitrary fixed pointed base.  A reconstruction at the level of unmarked geometric atomic equivalence classes is proved separately in \Cref{prop:CY-unmarked-atom-recovery}; it uses the same-component relation and is not a reconstruction of the function $t^0$.
\end{proof}

\begin{lemma}\label{lem:unit-equivariant-spectral-factors}
Let $(\cH_X,\nabla_X)/B_X$ be the non-Archimedean maximal
$A$-model $F$-bundle of a smooth projective variety, and write
$B_X=(\mathbb A^1_{t^0})^{\mathrm{an}}\times C$ for its unit-line
product. At a rigid point of the reduced unramified spectral locus,
consider the local decomposition into maximal factors
$(\cA_i,\nabla_i)/B_i$ furnished by the spectral decomposition theorem.
Then the unit translation acts on the corresponding spectral factors.
More precisely, the factors can be represented on a common product
neighbourhood by subbundles whose projectors are independent of $t^0$.
The induced unit field $e_i$ on each factor base satisfies
\[
 \mu_i(e_i)=I_{\cA_i/u\cA_i},\qquad e_i(a_i)=1,
 \quad a_i=-\frac{\operatorname{tr}(\kappa_i)}{\operatorname{rk}\cA_i}.
\]
These assertions hold on germs; the unit orbits needed to compare
geometric representatives are supplied by the entire unit line of
$B_X$, not by a global product decomposition of every factor base.
\end{lemma}

\begin{proof}
In the standard quantum-cohomology trivialisation the unit equation
is $\nabla_{\partial_{t^0}}=\partial_{t^0}+u^{-1}I$ and, by
\Cref{thm:unit-direction-residual-descent}, the full connection is
$p^*\nabla_X^0+d(t^0/u)I$. Write $P_i$ for the projection onto a
factor, after pulling the spectral decomposition back to its common
product base. The decomposition is one of $F$-bundles, hence the
$P_i$ are horizontal sections of $\End(\cH_X)$. It follows that
\[
 0=\nabla^{\End}_{\partial_{t^0}}P_i
   =\partial_{t^0}P_i+[u^{-1}I,P_i]
   =\partial_{t^0}P_i.
\]
Thus the summands, including their intrinsic lattices, are constant
in the unit direction. Their continuation is unambiguous: two
horizontal spectral projectors with the same reduction at $u=0$
cannot differ, since the first nonzero coefficient of their
difference would be off-diagonal and would commute with the
residual endomorphism; disjoint spectra exclude such a coefficient.
The local translated summands therefore agree on overlaps along
the unit orbit. The argument concerns projectors onto horizontal
$F$-bundle summands, not merely generalized eigenprojectors of
$\kappa_X$ on $\cH_X/u\cH_X$.

The product decomposition of maximal $F$-bundles induces the product
decomposition of their tangent algebras, as in the spectral
decomposition theorem of~\cite{HYZZ}. In the $A$-model the
cohomological unit is a $t^0$-independent cyclic section, so the
identification of the tangent sheaf with $\cH_X/u\cH_X$ sends the
tangent spectral idempotents to the reductions of the projectors
$P_i$. In particular, the corresponding tangent distributions are
independent of $t^0$. Denote the local factor projections by
$q_i$. Since $\mu_X(\partial_{t^0})=I$, the
factorwise image $e_i=dq_i(\partial_{t^0})$ satisfies
$\mu_i(e_i)=I$. This also shows that the distributions defining the
factor bases are invariant under unit translation. Indeed, their
spectral idempotents are independent of $t^0$, so translation
preserves each distribution and induces the indicated local flows
on the quotient germs.

The vertical residual endomorphism of the full $A$-model changes by
$-t^0 I$ along the unit line. Since the factor projectors are
independent of $t^0$, the same is true on each summand. Taking the
normalised negative trace gives
$a_i(t^0,x)=t^0+a_i(0,x)$ on the pullback to the common base.
As $a_i$ is a function on its factor base and
$e_i=dq_i(\partial_{t^0})$, we obtain $e_i(a_i)=1$.
These identities are polynomial in the unit coordinate. They hold
on the convergent non-Archimedean $A$-model domain without using a
complex-analytic normalising gauge.
\end{proof}

\begin{lemma}\label{lem:unit-spectral-cover-product}
Let $X$ be a smooth projective variety and let $B_X^G$ be the fixed
locus used for the $G$-atomic decomposition in~\cite[\S5.2]{KKPYY}.
Write $B_X^G=(\mathbb A^1_{t^0})^{\mathrm{an}}\times C^G$ and set
$\kappa_0=\kappa_X|_{t^0=0}$. In the standard $A$-model frame,
\[
 \kappa_X(t^0,x)=\kappa_0(x)-t^0I,\qquad
 \det(TI-\kappa_X(t^0,x))
   =\det((T+t^0)I-\kappa_0(x)).
\]
The change of eigenvalue coordinate $\lambda_0=\lambda+t^0$
identifies the reduced spectral cover, and its unramified locus,
with their pullbacks from $C^G$. Under this identification a point
and any of its unit translates belong to the same connected
component of the reduced unramified cover.
\end{lemma}

\begin{proof}
The characteristic-polynomial identity follows from the unit
equation. The spectral algebras are identified by the polynomial
substitution $\lambda_0=\lambda+t^0$, with polynomial inverse.
The identification therefore preserves reduction and the unramified
locus. The unit belongs to the $G$-fixed tangent directions, so the
same calculation applies on $B_X^G$. Since the non-Archimedean
analytic affine line is connected, translation does not change the
connected-component label of a point on the product cover. The unit-line product of the geometric $A$-model supplies the comparison of components.
\end{proof}

\begin{lemma}\label{lem:rigid-poincare-horizontal}
For the convergent non-Archimedean $A$-model of a smooth projective
variety $X$, the extension of the Poincar\'e pairing to the quantum
cohomology frame defines a non-degenerate horizontal pairing
\[
 \psi_X:\cH_X\otimes\nu^*\cH_X\longrightarrow\cO,
 \qquad \nu(u)=-u.
\]
The same pairing is horizontal for the formal $A$-model.
\end{lemma}

\begin{proof}
The genus-zero quantum product is graded self-adjoint for the
Poincar\'e form, while the grading operator
$\mathsf{Gr}=(\mathsf{Deg}-\dim X)/2$ is graded skew-adjoint.
In the quantum frame the horizontal connection is the sum of
ordinary differentiation and $u^{-1}$ times quantum multiplication;
its $u$-component is
$\partial_u-u^{-2}L_{\mathsf{Eu}}+u^{-1}\mathsf{Gr}$.
Pullback by $\nu$ reverses the sign of the connection term in
horizontal directions and changes the $u$-component to
$\partial_u+u^{-2}L_{\mathsf{Eu}}+u^{-1}\mathsf{Gr}$.
The adjointness identities show that the covariant derivative of
the Poincar\'e pairing vanishes in all these directions, including
the logarithmic Novikov directions. The pairing is non-degenerate
by Poincar\'e duality. These are coefficientwise identities for
the Gromov--Witten series; they remain valid on the rigid-analytic
domain where the series converge, as in~\cite[\S\S3.5.2 and 6.4]{KKPYY}.
\end{proof}

\begin{lemma}\label{lem:paired-spectral-orthogonality}
Suppose a maximal $F$-bundle $(\cH,\nabla)$ is equipped with a
non-degenerate horizontal pairing
$\psi:\cH\otimes\nu^*\cH\to\cO$, where $\nu(u)=-u$.
On a neighbourhood where its spectral decomposition has factors
$(\cH_i,\nabla_i)$ whose vertical residual spectra are pairwise
disjoint, the decomposition is orthogonal:
\[
 \psi(\cH_i,\nu^*\cH_j)=0\quad(i\ne j).
\]
Consequently $\psi_i=\psi|_{\cH_i\otimes\nu^*\cH_i}$ is horizontal
and non-degenerate for every factor. If the original pairing and
the spectral decomposition are $\mathcal G$-equivariant, so are
these restrictions.
\end{lemma}

\begin{proof}
Work on the common base of the spectral splitting and choose local
frames of the summands. Let $G_{ij}(u)$ be the matrix of the cross
pairing and let $K_i$ be the restriction to $u=0$ of
$\nabla_{u^2\partial_u}$ on the $i$th summand.
The involution sends the operator $u^2\partial_u$ to its negative.
Horizontality of $\psi$, reduced modulo $u$, therefore gives
\[
 K_i^{\mathrm{st}}G_{ij}(0)-G_{ij}(0)K_j=0,
\]
where $\mathrm{st}$ denotes the graded transpose.
Because the spectra of $K_i$ and $K_j$ are disjoint, the Sylvester
operator $T\mapsto K_i^{\mathrm{st}}T-TK_j$ is invertible, so
$G_{ij}(0)=0$. Suppose inductively that $G_{ij}$ is divisible by
$u^m$. The derivative $u^2\partial_u(u^mG)$ is divisible by
$u^{m+1}$; the same horizontality equation at the coefficient of
$u^m$ again has the invertible Sylvester operator as its leading
term. Hence $G_{ij}$ is divisible by $u^{m+1}$. Separatedness of
the formal or analytic $u$-adic local ring yields $G_{ij}=0$.
Since $\psi$ is non-degenerate on the direct sum and has no cross
terms, each diagonal restriction is non-degenerate. Its
horizontality is inherited from that of $\psi$. Equivariance, when
specified, follows from invariance of the spectral factors.
\end{proof}

\begin{lemma}\label{lem:transverse-unit-suspension}
Let $(\cA,\nabla)/B$ be a maximal geometric factor of a
non-Archimedean $A$-model, and put
$a=-\operatorname{tr}(\kappa)/\operatorname{rk}(\cA)$ and
$R=\nabla-d(a/u)I$. Then the Higgs map $\mu_R$ has
$
 \ker\mu_R=\cO_B e,
$
where $e$ is the intrinsic unit field of the factor. Near every
rigid point there is a transverse slice $C=\{a=a(b)\}$ and a
regular trivialisation along the unit leaves in which $R$ is the
pullback of its restriction $R_C$ to $C$. After moving the
representative along the geometric unit line so that $a(b)=0$,
the original factor is isomorphic to the suspension
\[
 p^*R_C+d(s/u)I,\qquad s=a.
\]
An isomorphism of centred geometric factor germs induces an
isomorphism of their transverse centred connections. Their
suspensions using a common coordinate $s$ are therefore
isomorphic. The assertion is equivariant and preserves specified
horizontal pairings whenever the isomorphism preserves them.
\end{lemma}

\begin{proof}
By \Cref{lem:unit-equivariant-spectral-factors}, the factor has a
unit field $e$ with $\mu_\nabla(e)=I$ and $e(a)=1$. Maximality
makes $\mu_\nabla$ injective. Consequently
\[
 \mu_R(\xi)=0\ \Longleftrightarrow\
 \mu_\nabla(\xi)=\xi(a)I\ \Longleftrightarrow\
 \xi=\xi(a)e,
\]
which proves the kernel assertion. Since $e(a)=1$, the scalar
function $a$ is a local coordinate along the regular foliation
spanned by $e$. Use $s=a-a(b)$ and take the analytic slice $s=0$.
The equation $\mu_R(e)=0$ means that $R_e$ has no pole along
$u=0$. A frame identifies $R_e$ with $\partial_s+A(s,u)$, where $A$
is regular in $u$ because $\mu_R(e)=0$. Over a closed affinoid
$u$-subdisc, write $A=\sum_{j\geq0}A_js^j$ and choose $M,\rho>0$
with $\|A_j\|\leq M\rho^{-j}$. The equation
$\partial_sG=-AG$, $G(0,u)=I$, has coefficients satisfying
$\|G_m\|\leq B^m/|m!|$ for $B\geq\max\{M,\rho^{-1}\}$,
by induction on $m$. Since $\limsup_m|m!|^{-1/m}<\infty$
in a complete non-Archimedean field of characteristic zero, $G$
and its inverse converge after shrinking the $s$-disc. This gives
a regular analytic horizontal trivialisation; the same recursion
works formally. In that frame $R_e=\partial_s$, and flatness
implies $\partial_sR_{\xi}=0$ for transverse fields $\xi$,
including $\partial_u$. Thus the full connection descends to $C$.

An isomorphism of centred germs intertwines their Higgs maps,
hence identifies the intrinsic line distributions $\ker\mu_R$
and their leaf spaces. It induces an isomorphism of transverse
centred connections: in the regular leafwise horizontal frames
constructed above the intertwining map is independent of the leaf
coordinate. Choose one common formal or analytic affine coordinate
$s$ for the two suspensions, normalised to vanish at their base
points. Pullback of the transverse isomorphism now intertwines the
connections after adding $d(s/u)I$ on both sides. No intrinsic
normalisation of $s$ is inferred from the unlabelled centred germ. 
For a geometric factor, \Cref{lem:unit-spectral-cover-product}
allows us to replace the original representative by one on the
same spectral component with $a(b)=0$: translation of $t^0$ by
$-a(b)$ makes the trace scalar zero at the new point. At that
point the previous formula reads $\nabla=p^*R_C+d(s/u)I$.
For a specified horizontal pairing the two scalar terms cancel
on $\cA\otimes\nu^*\cA$. In the leafwise horizontal frames
the pairing matrix has zero $s$-derivative, so it is the pullback
of its non-degenerate transverse restriction. Thus an intertwiner
preserving the centred pairing restricts to a pairing-preserving
transverse isomorphism and extends to a pairing-preserving
suspension isomorphism. The same reasoning respects compatible
$\mathcal G$-actions.
\end{proof}

\begin{theorem}\label{thm:residual-reconstruction-loss}\label{prop:CY-unmarked-atom-recovery}
Let $(\cH,\nabla,\psi)/S$ be a positive-rank $F$-bundle in the formal,
complex-analytic or non-Archimedean category, with a specified
non-degenerate horizontal pairing
$\psi:\cH\otimes\nu^*\cH\to\cO$. Put
$a=-\operatorname{tr}(\kappa_\nabla)/\operatorname{rk}\cH$ and
$\nabla^0=\nabla-d(a/u)I$.
\begin{enumerate}[label=\textup{(\roman*)}]
\item The trace-labelled presentation $(a,\cH,\nabla^0,\psi)$
recovers the connection, lattice and pairing, functorially on
pairing-preserving morphisms.
\item Forgetting $a$ may identify distinct paired connections on
a fixed base. In the complex category, where half-turn transport
is defined, $\chi_{\nabla^0}=e^{-2a}\chi_\nabla$; the asymmetry
operator is unchanged, but the Euler form may be lost.
\item For the geometric $G$-atomic $A$-model factors of~\cite{KKPYY},
with pairings induced by a globally specified non-degenerate
horizontal pairing on each full $A$-model, forgetting $a$ induces
a bijection from paired geometric atomic classes to the classes
of their unlabelled centred presentations. On both sides, the
equivalence is generated by pairing-preserving isomorphisms and
representatives on the same connected component of the reduced
unramified spectral cover, as in~\cite[Definition~5.21]{KKPYY}.
In particular, for a connected smooth projective Calabi--Yau
variety with $H^1(X,\C)=0$, its unique paired geometric atomic
class is recovered (\cite[Lemma~5.24]{KKPYY}). For fixed dimension parity, the cohomological
Serre operator $(-1)^{\dim X}I$ is preserved.
\end{enumerate}
In~\textup{(iii)} the paired equivalence uses the specified geometric sources and horizontal pairings. The affine unit coordinate, a half-turn Euler form and an integral $\Gamma$-structure are additional data.
\end{theorem}

\begin{proof}
For~\textup{(i)}, the scalar correction is closed and preserves the
$F$-bundle pole conditions. Since $\nu^*d(a/u)=-d(a/u)$, its
contributions to the pairing connection cancel, including in
base directions. Addition of $d(a/u)I$ is the inverse on
objects and morphisms by
\Cref{prop:trace-centred-splitting,prop:enhanced-residualisation}.
For~\textup{(ii)}, the fixed-base examples and the fibre of the
forgetful map are given by
\Cref{prop:forgetting-branch-loses-information};
\Cref{lem:scalar-half-turn} proves the half-turn formula and
preservation of the asymmetry operator.

For~\textup{(iii)}, the unit field $e$ of a geometric spectral factor
satisfies $e(a)=1$. By
\Cref{lem:unit-equivariant-spectral-factors,lem:unit-spectral-cover-product},
unit translation by $-a(b)$ replaces a pointed representative $b$
by one on the same spectral component with $a=0$. The centred Higgs
map has kernel $\cO e$, and the leafwise horizontal gauge of
\Cref{lem:transverse-unit-suspension} identifies its connection
locally with $p^*R_C$, where $p:C\times D_s\to C$ is the
projection and $D_s$ is a small analytic unit-leaf disc. An isomorphism of unlabelled
centred germs carries $\ker\mu_R$ to $\ker\mu_{R'}$ and induces an
isomorphism $R_C\simeq R'_{C'}$ of their transverse connections.
Pulling it back and adding $d(s/u)I$ on both sides identifies the
original germs in their normalised representatives. This proves
injectivity on geometric atomic classes; surjectivity onto the
classes of unlabelled presentations is part of their definition.
The inverse respects the two generating relations: the trace is
preserved by original-germ isomorphisms, while the unit translations
used above remain within each spectral component.

The specified global $A$-model pairing is horizontal by
\Cref{lem:rigid-poincare-horizontal}, or by hypothesis for other
pairings, and restricts non-degenerately to the spectral factors by
\Cref{lem:paired-spectral-orthogonality}. In a leafwise horizontal
frame its matrix satisfies $\partial_s\psi=0$, so
$\psi=p^*\psi_C$. An isomorphism preserving the centred pairing
therefore restricts to an isometry of the transverse connections
with their pairings; its suspension preserves the pairings because
the two scalar corrections cancel on $\cA\otimes\nu^*\cA$.
Thus the two generating relations for the specified paired
geometric sources are preserved in both directions. For a
Calabi--Yau source, the scalar operator $(-1)^{\dim X}I$ is also
preserved.
\end{proof}

\section{Spectral projectors and lattice recomposition}\label{sec:spectral-eigenline-collision}

Quantum boundary lattices provide a further instance of spectral recomposition. In the calculation below the eigenlines belong to the Higgs operator, not to a decomposition into horizontal $F$-bundle summands.

\begin{proposition}\label{prop:projective-spectral-lattice-collision}
Let $X$ be a smooth Fano variety of dimension $n\geq1$ and Picard number one, with $-K_X=rH$. Assume that $1,h,\ldots,h^n$ form a basis of $H^*(X,k)$, where $h=c_1(H)$ and $k$ is a characteristic-zero field. Write $q$ for the Novikov parameter and suppose that the specialised small quantum ring $QH^*(X)|_{q=1}$ is semisimple. After a finite extension of $k$, let
$F(T)=T^{n+1}+\sum_{a\geq1,\,ar\leq n+1}c_aT^{n+1-ar}$ be the monic relation of quantum multiplication by $h$ at $q=1$, and assume its $n+1$ roots $\lambda_0,\ldots,\lambda_n$ belong to $k$. Put $R=k[[x]]$ and $q=x^r$. In the quantum-power basis $e_j=h^{\star j}$, $0\leq j\leq n$, let $H_R=\bigoplus_{j=0}^nRe_j$. For the eigenspace decomposition of $L_h$, its recomposition defect is
\begin{equation*}
 D_{\mathrm{Higgs}}(H_R)\simeq\bigoplus_{j=0}^n R/(x^{n-j}),
 \qquad \ell(D_{\mathrm{Higgs}}(H_R))=\frac{n(n+1)}2.
\end{equation*}
The least common pole exponent of the projectors is $n$. The elementary modification $H'_R=R(xe_0)\oplus\bigoplus_{j=1}^nRe_j$ is preserved, together with $H_R$, by the pulled-back logarithmic small quantum connection. Its defect is
\[
 D_{\mathrm{Higgs}}(H'_R)\simeq R/(x^{n-1})\oplus\bigoplus_{j=1}^nR/(x^{n-j}).
\]
Thus the defect length drops by one and the least projector pole exponent drops to $n-1$. The same polynomial lattice calculation holds over a sufficiently small closed rigid disc $R_\rho=k\langle x/\rho\rangle$ when $k$ is complete non-Archimedean and contains the roots of $F$. The hypotheses hold for $\mathbb P^n$ and for smooth odd-dimensional quadrics $Q^n$ with $n\geq3$; in the latter case $r=n$ and $F(T)=T^{n+1}-4T$.
\end{proposition}

\begin{proof}
Since the cohomology is generated by $h$, the quantum powers $e_j$ give a basis over $k[q]$: their specialisations at $q=0$ are the ordinary powers $h^j$, and the grading makes the change of basis triangular. The monic relation is homogeneous for $\deg h=1$ and $\deg q=r$, hence after $q=x^r$ it is $x^{n+1}F(h/x)=0$. Its roots are $x\lambda_i$ and are distinct over $k((x))$. The eigenline for $x\lambda_i$ is generated by
$v_i=F_x(h)/(h-x\lambda_i)=\sum_{j=0}^nx^{n-j}b_{ji}e_j$, where $F_x(T)=x^{n+1}F(T/x)$ and the $b_{ji}\in k$ are the coefficients of $F(T)/(T-\lambda_i)$. The $e_n$-coefficient is $1$, hence $H_R\cap k((x))v_i=Rv_i$. Their coefficient matrix is $\operatorname{diag}(x^n,\ldots,x,1)B$, with $B=(b_{ji})\in\GL_{n+1}(k)$: the quotient polynomials are linearly independent by evaluation at the distinct roots. Thus their saturated sum is $\bigoplus_j Rx^{n-j}e_j$, proving the first formula and the pole bound.

Every $v_i$ is primitive in $H'_R$: its $e_0$-coefficient lies in $xR$, while its $e_n$-coefficient is $1$. The first diagonal exponent therefore becomes $n-1$. To check compatibility with the connection, note that $L_h(e_j)=e_{j+1}$ for $j<n$ and $L_h(e_n)$ is a linear combination of $x^{ar}e_{n+1-ar}$, with an $e_0$-coefficient (if present) divisible by $x^{n+1}$. The quantum-power basis agrees with the ordinary power basis modulo $x$, so the operator $x\partial_x$ preserves both lattices; in particular, its coefficients from $e_j$ to $e_0$ are divisible by $x$ for $j>0$. The grading operator likewise has no such coefficient modulo $x$. Hence the logarithmic divisor operator $x\partial_x+(r/u)L_h$ and the vertical operator assembled from the grading and $c_1(X)\star=rL_h$ preserve both lattices. The connection matrices are polynomial in $x$, so the same lattice calculation holds on a sufficiently small rigid disc.

For $\mathbb P^n$, $F(T)=T^{n+1}-1$ and $r=n+1$. For an odd quadric of dimension $n\geq3$, the standard quantum relation is $h^{\star(n+1)}=4qh$ (see~\cite{PechRietschQuadric}); thus $F(T)=T^{n+1}-4T$, which has simple roots in characteristic zero. Both satisfy the hypotheses; their Higgs projectors need not be horizontal.
\end{proof}

Picard number one need not imply cyclicity (even-dimensional quadrics provide examples). The calculation is relative to the chosen quantum boundary lattice.

\subsection{The recomposition defect}\label{sec:spectral-recomposition-defect}

The cokernel of the saturated spectral intersections in a lattice measures its projector pole bound; formal projectors need not have uniformly bounded pole orders.

Let $R$ be a noetherian ring, let $x\in R$ be a non-zero-divisor, and write $R_x=R[1/x]$.  Let $H$ be a finite locally free $R$-module.  Suppose that a finite collection of idempotents $e_\alpha\in\End_{R_x}(H_x)$ satisfies
\[
 e_\alpha e_\beta=\delta_{\alpha\beta}e_\alpha,
 \qquad \sum_\alpha e_\alpha=I,
 \qquad H_x=\bigoplus_\alpha V_\alpha,
 \quad V_\alpha=e_\alpha H_x.
\]
Here the idempotents are assumed to be defined over $R_x$, not merely as formal series whose coefficients individually belong to $R_x$.

\begin{theorem}\label{thm:spectral-recomposition-defect}
Under these assumptions the intersections $L_\alpha=H\cap V_\alpha$, taken inside $H_x$, are finite $R$-modules and are $x$-saturated in $H$.  Their sum is direct, and there is a canonical exact sequence
\begin{equation}\label{eq:recomposition-exact-sequence}
 0\longrightarrow\bigoplus_\alpha L_\alpha
 \longrightarrow H\longrightarrow D_{\mathrm{sp}}(H,\{e_\alpha\})
 \longrightarrow0.
\end{equation}
The module $D_{\mathrm{sp}}(H,\{e_\alpha\})$ is finitely presented and annihilated by a power of $x$.  More precisely, for every integer $d\geq0$,
\begin{equation}\label{eq:recomposition-pole-exponent}
 x^dD_{\mathrm{sp}}=0
 \quad\Longleftrightarrow\quad
 x^de_\alpha(H)\subseteq H\quad\text{for every }\alpha.
\end{equation}
In particular, $D_{\mathrm{sp}}=0$ if and only if all the projectors preserve $H$.  Its Fitting ideals are intrinsic to the pair consisting of the lattice and the unlabelled spectral decomposition, and their formation commutes with arbitrary base change as ideals of the base-changed module. 
If $R\to R'$ is flat and the image of $x$ remains a non-zero-divisor, then
\[
 L_\alpha\otimes_RR'
 \simeq (H\otimes_RR')\cap (V_\alpha\otimes_{R_x}R'_x),
 \qquad D_{\mathrm{sp}}\otimes_RR'\simeq D'_{\mathrm{sp}}.
\]
For a non-flat base change these identifications need not hold, even when $x$ remains a non-zero-divisor.
\end{theorem}

\begin{proof}
The intersection $L_\alpha$ is the kernel of the $R$-linear map
\[
 H\longrightarrow H_x,\qquad h\longmapsto (I-e_\alpha)h.
\]
It is therefore a finite $R$-module.  Since $H/L_\alpha$ embeds into $H_x/V_\alpha$, multiplication by $x$ is injective on $H/L_\alpha$, proving saturation.  The sum of the $L_\alpha$ is direct because the $V_\alpha$ form a direct sum over $R_x$. 
Each $e_\alpha$ is represented locally by a matrix whose entries belong to $R_x$.  Since there are finitely many projectors, one can choose $d$ such that $x^de_\alpha(H)\subseteq H$ for every $\alpha$.  As $x^de_\alpha(h)$ then belongs to both $H$ and $V_\alpha$, we have
\[
 x^dh=\sum_\alpha x^de_\alpha(h)\in\sum_\alpha L_\alpha.
\]
Thus $x^dD_{\mathrm{sp}}=0$.  Conversely, if $x^dH\subseteq\bigoplus_\alpha L_\alpha$, applying $e_\beta$ to this inclusion gives $x^de_\beta(H)\subseteq L_\beta\subseteq H$.  This proves \eqref{eq:recomposition-pole-exponent}.  Finite presentation follows from noetherianity, and the assertion about Fitting ideals follows from their definition by minors of any finite presentation.  Permuting the projectors does not change the submodule $\bigoplus_\alpha L_\alpha\subseteq H$.

For flat base change, tensor the kernel description of $L_\alpha$ with $R'$.  Flatness preserves kernels, and localisation commutes with tensor product.  This identifies the base change of $L_\alpha$ with the indicated intersection.  Tensoring \eqref{eq:recomposition-exact-sequence} then identifies the cokernels as asserted.  The final statement is illustrated in \Cref{ex:recomposition-nonflat-base-change}.
\end{proof}

\begin{proposition}\label{prop:wall-control-spectral-defect}
Keep the meromorphic spectral decomposition and its projectors $e_\alpha$ fixed as in \Cref{thm:spectral-recomposition-defect}. Let $H'$ be another finite locally free $R$-lattice in $H_x$, with $H'_x=H_x$, and suppose that for integers $a,b\geq0$ one has
\[
 x^aH\subseteq H'\subseteq x^{-b}H.
\]
Write $D_H=D_{\mathrm{sp}}(H,\{e_\alpha\})$ and define $D_{H'}$ analogously. If $d(H)$ is the least non-negative integer with $x^{d(H)}D_H=0$, then
\[
 |d(H')-d(H)|\leq a+b.
\]
In particular, an elementary modification $xH\subseteq H'\subseteq H$ changes the least projector pole exponent by at most one. Along a chain of $N$ elementary modifications, its total variation is at most $N$.

When $H'\subseteq H$, put $L_\alpha=H\cap V_\alpha$ and $L'_\alpha=H'\cap V_\alpha$. The inclusions induce a canonical exact sequence
\begin{equation}\label{eq:wall-defect-comparison}
 0\longrightarrow\ker(D_{H'}\to D_H)
 \longrightarrow\bigoplus_\alpha L_\alpha/L'_\alpha
 \longrightarrow H/H'
 \longrightarrow\operatorname{coker}(D_{H'}\to D_H)
 \longrightarrow0.
\end{equation}
All assertions are independent of an ordering of the spectral summands. If the projectors are horizontal for a meromorphic connection on the common localisation and both lattices satisfy the chosen pole conditions, the comparison uses the same horizontal decomposition before and after the lattice modification. In particular, scalar trace-centring does not alter any of these modules or bounds.
\end{proposition}

\begin{proof}
By \eqref{eq:recomposition-pole-exponent}, $x^{d(H)}e_\alpha(H)\subseteq H$ for all $\alpha$. From $H'\subseteq x^{-b}H$ we obtain
\[
 e_\alpha(H')\subseteq x^{-b}e_\alpha(H)
 \subseteq x^{-b-d(H)}H
 \subseteq x^{-a-b-d(H)}H'.
\]
Thus $d(H')\leq d(H)+a+b$. The inclusions also give $x^bH'\subseteq H\subseteq x^{-a}H'$. Applying the displayed projector inclusions with $(H,H',a,b)$ replaced by $(H',H,b,a)$ gives $d(H)\leq d(H')+a+b$. Taking $(a,b)=(1,0)$ gives the elementary bound, which adds along a chain of modifications.

For $H'\subseteq H$, the canonical inclusions of the saturated spectral lattices give a commutative diagram of the exact recomposition sequences
\[
\begin{array}{ccccccccc}
0&\to&\bigoplus_\alpha L'_\alpha&\to&H'&\to&D_{H'}&\to&0\\
 &&\downarrow&&\downarrow&&\downarrow&&\\
0&\to&\bigoplus_\alpha L_\alpha&\to&H&\to&D_H&\to&0.
\end{array}
\]
The snake lemma gives \eqref{eq:wall-defect-comparison}. The final assertions follow because the intersections and the idempotents depend only on the specified meromorphic decomposition and the lattices. Horizontality is preserved under scalar shifts, as in \Cref{prop:recomposition-horizontal-pairing} below.
\end{proof}

\begin{proposition} \label{prop:rigid-boundary-comparison}
Let $A$ be a noetherian affinoid algebra over a complete non-Archimedean field, let $x\in A$ be a non-zero-divisor defining a chosen analytic boundary, and let $H$ be a finite locally free $A$-lattice with a meromorphic connection of the prescribed pole type.  Suppose that the $F$-bundle admits, after restricting to $A[1/x]$ (and, if necessary, a fixed finite etale splitting base change), a finite family of mutually orthogonal \emph{horizontal} spectral projectors
\[
 e_\alpha\in\End_{A[1/x]}(H[1/x]),
 \qquad \sum_\alpha e_\alpha=I.
\]
Then the saturated spectral intersections define a finite coherent $x$-power-torsion defect, and its annihilator exponent equals the least common $x$-pole order of the projectors.  Under an elementary change of the prescribed boundary lattice with the projectors fixed, this exponent changes by at most one.  Trace-centring preserves the projectors, their intersections and the defect.  The assertion is invariant under a permutation of the factors; after a flat splitting base change it is compatible with that base change.
\end{proposition}

\begin{proof}
The affinoid algebra $A$ is noetherian.  The finite-pole assumption means precisely that the projectors belong to $\End_{A[1/x]}(H[1/x])$, rather than merely having convergent entries away from the divisor.  The exact sequence, coherence and pole-order characterisation follow from \Cref{thm:spectral-recomposition-defect} applied to $R=A$.  The elementary-wall inequality follows from \Cref{prop:wall-control-spectral-defect}.  For $\nabla^{\mathrm{cent}}=\nabla-d(a/u)I$, the scalar correction commutes with the projectors, so they remain horizontal; the lattice and the meromorphic summands do not change.  The defect is independent of the numbering of the projectors, and the flat base-change claim is part of \Cref{thm:spectral-recomposition-defect}.
\end{proof}

These are horizontal $F$-bundle projectors, not the Higgs projectors of the Fano example. Bounded-pole extension is an additional hypothesis, not guaranteed by~\cite{HYZZ,KKPYY}.

\begin{example} \label{ex:sharp-horizontal-wall-defect}
Let $R=k[[x]]$, $K=k((x))$ and $m\geq1$. In $K^2$ take
\[
 e_1=\begin{pmatrix}1&x^{-m}\\0&0\end{pmatrix},\qquad
 e_2=I-e_1,\qquad H=R^2,\qquad
 H'=R(1,0)^{\mathsf T}\oplus Rx(0,1)^{\mathsf T}.
\]
The two spectral lines are $V_1=K(1,0)^{\mathsf T}$ and $V_2=K(-1,x^m)^{\mathsf T}$. The saturated intersections are
$
 L_1=R(1,0)^{\mathsf T},\ L_2=R(-1,x^m)^{\mathsf T},\
 L'_1=L_1,\quad L'_2=L_2.
$
Taking determinants in the respective lattice bases gives
\[
 D_H\simeq R/(x^m),\qquad D_{H'}\simeq R/(x^{m-1}),
\]
where $R/(1)=0$. Hence $d(H)=m$ and $d(H')=m-1$, so the elementary wall bound is sharp. 
Moreover, consider the meromorphic connection
\[
 \nabla=d+\begin{pmatrix}0&0\\0&-m\end{pmatrix}\frac{dx}{x}.
\]
Both $e_1$ and $e_2$ are horizontal: for example, the off-diagonal entry of $de_1$ is $-mx^{-m-1}dx$, which cancels the corresponding entry of $[\operatorname{diag}(0,-m)dx/x,e_1]$. The connection is logarithmic on $H$, and in the basis $( (1,0)^{\mathsf T},x(0,1)^{\mathsf T})$ of $H'$ its residue is $\operatorname{diag}(0,1-m)$. Thus the two lattices give logarithmic extensions of the same meromorphic connection with the same horizontal spectral projectors, but with different defects. In particular the case $m=1$ changes the defect from $R/(x)$ to zero.
\end{example}

\begin{remark}\label{rem:spectral-defect-not-atomic}
The logarithmic extensions in \Cref{ex:sharp-horizontal-wall-defect} have identical meromorphic connections and horizontal projectors, but different recomposition defects. The defect records the relative position of a boundary lattice and a meromorphic splitting; its comparison under geometric correspondences involves the lattices as well as the splitting.
\end{remark}

\begin{proposition}\label{prop:recomposition-horizontal-pairing}
Suppose that $H$ has a connection $\nabla$ with $\nabla(H)\subseteq H\otimes_R P$, where $P$ is a finite locally free module of permitted meromorphic one-forms.  Suppose also that the projectors $e_\alpha$ are horizontal over $R_x$.  Then
\[
 \nabla(L_\alpha)\subseteq L_\alpha\otimes_RP
\]
for every $\alpha$.  Replacing $\nabla$ by a scalar shift $\nabla-\omega I$ does not change the projectors, intersections or defect module, provided the shifted connection is defined with the prescribed pole bounds.
\end{proposition}

\begin{proof}
Since $e_\alpha$ is horizontal, $\nabla(V_\alpha)$ lies in $V_\alpha\otimes_{R_x}P_x$.  For $h\in L_\alpha$, the pole condition on $H$ also gives $\nabla h\in H\otimes_RP$.  Local freeness of $P$ yields the intersection identity
\[
 (H\otimes_RP)\cap(V_\alpha\otimes_{R_x}P_x)
 =L_\alpha\otimes_RP
\]
inside $H_x\otimes_{R_x}P_x$.  Hence $\nabla h\in L_\alpha\otimes_RP$.  A scalar shift commutes with every projector and does not change the underlying module or its intersections with the spectral submodules.
\end{proof}

\begin{example}\label{ex:recomposition-nonflat-base-change}
Let $R=k[[s,x]]$ with $k$ a field, and take $H=R^2$.  Over $R_x$ the vectors $v_1=(x,s)^{\mathsf T}$ and $v_2=(0,1)^{\mathsf T}$ form a basis.  Their saturated intersections with $H$ are
\[
 L_1=R(x,s)^{\mathsf T},\qquad L_2=R(0,1)^{\mathsf T},
 \qquad D_{\mathrm{sp}}\simeq R/(x).
\]
The first equality uses the fact that $s$ and $x$ are relatively prime in $k[[s,x]]$.  On the non-flat specialisation $s=0$, the two generic summands become the coordinate lines, which do split $H'=k[[x]]^2$.  Thus the defect recomputed after specialisation is zero, whereas
\[
 D_{\mathrm{sp}}\otimes_R k[[x]]\simeq k[[x]]/(x).
\]
In particular, the Fitting ideals of the base-changed module still satisfy the usual base-change formula, but do not describe the defect obtained by saturating anew on the special fibre.
\end{example}

A horizontal formal projector $\Pi_\alpha\in\End_{k((x))}(k((x))^r)[[u]]$ may have pole orders increasing with its $u$-degree, so it need not lie in $\End(k[[x,u]][1/x]^r)$. For each $m$, however, it gives idempotents over $R_m[1/x]$, where $R_m=k[[x,u]]/(u^m)$, and a finite $x$-power-torsion defect $D_m$ by \Cref{thm:spectral-recomposition-defect}. For $\ell\leq m$ there is a natural surjection
\[
 D_m\otimes_{R_m}R_\ell\twoheadrightarrow D_\ell,
\]
since reduction maps each saturated intersection $L_{\alpha,m}$ into $L_{\alpha,\ell}$. Equality need not hold: saturation can change under this non-flat reduction. Only idempotents, not connections on the truncated modules, are used.

\begin{example}\label{ex:unbounded-formal-projector}
The distinction between finite jets and a uniformly bounded meromorphic projector is already visible in rank two.  In $k((x))[[u]]$ put $f=(x-u)^{-1}=\sum_{j\geq0}u^jx^{-j-1}$ and
\[
 E=\begin{pmatrix}1&f\\0&0\end{pmatrix},\qquad E^2=E.
\]
For each $m$, the images of $E$ and $I-E$ over $R_m[1/x]$ have saturated intersections generated respectively by $(1,0)^{\mathsf T}$ and $(-1,x-u)^{\mathsf T}$.  Consequently
\[
 D_m\simeq R_m/(x-u),\qquad x^mD_m=0,
 \qquad x^{m-1}D_m\ne0.
\]
There is no uniform bound on the exponents of the modules $D_m$.  The formal idempotent $E$ therefore does not belong to $\End_{k[[x,u]][1/x]}(k[[x,u]][1/x]^2)$.  This is a purely algebraic illustration of the boundedness issue, not a claim that $E$ is a projector of a quantum connection.
\end{example}

The Higgs defect alone does not imply extension of horizontal projectors.

\subsection{Uniform pole bounds for formal projectors}\label{sec:uniform-formal-spectral-defect}

Uniform bounds on the finite-jet defects characterise formal meromorphic extension, not convergence in $u$. Let $k$ have characteristic zero, $R=k[[x,u]]$, $R_m=R/(u^m)$, $H=R^r$ and $K=k((x))$. For orthogonal idempotents $\Pi_\alpha\in\operatorname{Mat}_r(K[[u]])$ summing to $I$, denote by $\Pi_{\alpha,m}$ and $D_m$ their reductions and recomposition defects over $R_m[1/x]$.

\begin{theorem}\label{thm:uniform-formal-spectral-defect}
The following conditions are equivalent.
\begin{enumerate}[(i)]
 \item There is an integer $d\geq0$ such that $x^dD_m=0$ for every $m\geq1$.
 \item There is an integer $d\geq0$ such that $x^d\Pi_{\alpha,m}(H_m)\subseteq H_m$ for every $\alpha$ and $m\geq1$.
 \item Every $\Pi_\alpha$ belongs to $\operatorname{Mat}_r(R[1/x])$.
\end{enumerate}
The least possible $d$ in \textup{(i)} and \textup{(ii)} is the least integer for which $x^d\Pi_\alpha\in\operatorname{Mat}_r(R)$ for all $\alpha$.

When these conditions hold, the modules
\[
 L_\alpha=H\cap\operatorname{im}(\Pi_\alpha|_{H[1/x]})
\]
form a canonical finite decomposition lattice, and $D=H/\bigoplus_\alpha L_\alpha$ is finitely presented and annihilated by $x^d$.  For every $m\geq1$ there is a natural surjection $D/u^mD\twoheadrightarrow D_m$. The surjection need not be an isomorphism, because saturated intersections may change after reduction modulo $u^m$.
\end{theorem}

\begin{proof}
For each $m$, \Cref{thm:spectral-recomposition-defect} applied over $R_m$ gives
\[
 x^dD_m=0\quad\Longleftrightarrow\quad
 x^d\Pi_{\alpha,m}(H_m)\subseteq H_m
 \quad\text{for all }\alpha.
\]
This proves the equivalence of \textup{(i)} and \textup{(ii)} with the same $d$.

Write $\Pi_\alpha=\sum_{j\geq0}u^jP_{\alpha,j}$ with $P_{\alpha,j}\in\operatorname{Mat}_r(K)$.  Condition \textup{(ii)} says that $x^dP_{\alpha,j}$ has entries in $k[[x]]$ for every $j<m$, uniformly in $m$.  It is therefore equivalent to $x^dP_{\alpha,j}\in\operatorname{Mat}_r(k[[x]])$ for every $j$, or $x^d\Pi_\alpha\in\operatorname{Mat}_r(R)$.  Since there are finitely many $\alpha$, this is equivalent to \textup{(iii)}.  The same argument proves the assertion about the least exponent.

Under \textup{(iii)}, the idempotents act on $H[1/x]$.  The finite lattice and its defect are those of \Cref{thm:spectral-recomposition-defect}.  Reduction sends each $L_\alpha$ into the intersection lattice for $\Pi_{\alpha,m}$, because the defining idempotent equation survives reduction.  Passing to the quotients gives the asserted surjection.  The possible failure of equality is the non-flatness of $R\to R_m$, already apparent in \Cref{ex:recomposition-nonflat-base-change}.
\end{proof}

The bound concerns the annihilator, not length: $E=\left(\begin{smallmatrix}1&x^{-1}\\0&0\end{smallmatrix}\right)$ has $D_m\simeq R_m/(x)$, of $k$-length $m$ but exponent one. In \Cref{ex:unbounded-formal-projector} the exponents instead grow with $m$. Neither case implies analytic convergence.

\begin{corollary}\label{cor:quantum-boundary-jet-test}
In \Cref{prop:projective-spectral-lattice-collision}, let $H=H_R[[u]]$ and $H'=H'_R[[u]]$ over $R=k[[x,u]]$. Suppose the Higgs eigenprojectors lift to mutually orthogonal horizontal projectors $\Pi_i\in\operatorname{Mat}_{n+1}(k((x))[[u]])$ summing to $I$. If $d_m,d'_m$ denote their defect exponents on $H/u^mH,H'/u^mH'$, then for all $m\geq1$,
\[
 d_m\geq n,\qquad d'_m\geq n-1,\qquad |d_m-d'_m|\leq1.
\]
Moreover, both sequences are unbounded. Thus these horizontal projectors do not belong to $\operatorname{Mat}_{n+1}(R[1/x])$: the formal splitting has no uniformly meromorphic extension across $x=0$. Existence of the lifts is assumed; no analytic convergence in $u$ is asserted.
\end{corollary}
\begin{proof}
The inclusions $xH\subseteq H'\subseteq H$ persist modulo $u^m$, so \Cref{prop:wall-control-spectral-defect} gives the last inequality. Reduction modulo $u$ gives the lower bounds from \Cref{prop:projective-spectral-lattice-collision}.

If either sequence were bounded, the wall inequality and \Cref{thm:uniform-formal-spectral-defect} would give $\Pi_i\in\operatorname{Mat}_{n+1}(R[1/x])$ for every $i$. Work over $k((u))$ and write the pulled-back small quantum connection in the quantum-power basis as
$\nabla_{x\partial_x}=x\partial_x+A(x,u)$, with $A(0,u)=(r/u)N$, where $N=L_h|_{x=0}$ is a single nilpotent Jordan block. The change from the constant cohomological basis to the quantum-power basis is regular and invertible at $x=0$, so it does not alter this residue. For a nonzero horizontal projector $P=\sum_{j\geq-d}x^jP_j(u)$ with $d>0$, its lowest coefficient would satisfy
\[
 (-dI+(r/u)\operatorname{ad}N)P_{-d}=0.
\]
Since $\operatorname{ad}N$ is nilpotent, this forces $P_{-d}=0$. Thus every meromorphic horizontal projector is regular at $x=0$. Its reduction $P_0$ is an idempotent commuting with $N$. The centraliser of a single Jordan block is $k((u))[N]$, a local algebra, so $P_0$ is either $0$ or $I$. An idempotent in $k((u))[[x]]$ reducing to $0$ is itself zero, by repeated divisibility by $x$; applying this to $I-P$ covers the other case. Thus $P$ is $0$ or $I$, contradicting its generic rank between $1$ and $n$. Hence no nontrivial horizontal projector has a uniform pole bound, proving the claim.
\end{proof}

\par\medskip
\noindent\textbf{Acknowledgements.} I am grateful to Simone Noja for many stimulating discussions and helpful suggestions. I am partially supported by the Universit\`a degli Studi di Bari Aldo Moro.
\par\medskip

\noindent\textbf{AI disclosure.} The mathematical results and arguments in this article are my own. I used ChatGPT solely to assist with proofreading and to improve the clarity of the exposition.
\par

\end{document}